\documentclass{report}
\usepackage{graphicx} % Required for inserting images

\usepackage[utf8]{inputenc}
\usepackage{xcolor}
\usepackage[explicit]{titlesec}
\usepackage{soul}
\usepackage[margin=1.0in]{geometry}
\usepackage{geometry}                % See geometry.pdf to learn the layout options. There are lots.
\usepackage{graphicx}
\usepackage{amssymb}
\usepackage{epstopdf}
\usepackage{tikz}
\usetikzlibrary{shapes.geometric, arrows}
\usepackage{tikz-qtree,tikz-qtree-compat}
\usetikzlibrary{arrows}
\usepackage{amsmath}
\usepackage{wrapfig}
\usepackage{lipsum}
\usepackage{enumitem}
\usepackage{pbox}

\usepackage{amsthm}
\newtheorem{theorem}{Theorem}[section]
\newtheorem{corollary}{Corollary}[theorem]
\newtheorem{lemma}[theorem]{Lemma}
\newtheorem{proposition}[theorem]{Proposition}
\newtheorem{definition}{Definition}[section]
\newtheorem*{remark}{Remark}

\newtheorem{example}{Example}

\title{Bounded Trajectories of the $3x+\gamma$ Problem}
\author{Benjamin Bairrington}
\date{June 2023}

\begin{document}

\maketitle

\chapter*{Abstract}

The $3x+1$ Problem was first formally presented during the 1950 International Congress of Mathematics\cite{lagarias19853x+} by Dr. Collatz. In this thesis, we study the bounded trajectories of the $3x+\gamma$ Problem, which is a generalization of the $3x+1$ Problem.  Fix $\gamma \in \mathbb{Z}_{>0}^{odd}$ and $n\in\mathbb{Z}_{>0}$. We define the function $C_\gamma:{\mathbb Z}_{>0}\to {\mathbb Z}_{>0}$ such that if $n$ is odd, $C_\gamma(n)=3n+\gamma$; and if $n$ is even, $C_\gamma(n)=n/2$.  We call the sequence 
$\{n,C_\gamma(n),C_\gamma^2(n),\cdots\}$ a trajectory of $n$. The Collatz conjecture says that
for any $n\in {\mathbb Z}_{>0}$, there exists $k\in {\mathbb Z}_{>0}$ such that $C_1^k(n)=1$.
We define the characteristic mapping $\chi_\gamma: {\mathbb Z}_{>0}\to \{0,\, 1\}$ to be 
$\chi_\gamma(n)\equiv C_\gamma(n)\, {\rm mod}\, 2$. The sequence generated by compositions of 
$\chi_\gamma$ acting on $n$ is called the characteristic trajectory associated to $n$.
Assume that the trajectory of $n$ is bounded above. 
Then there exists some $n'$ contained in the trajectory and the smallest $N$ in $\mathbb{Z}_{>0}$ such  that  $C_\gamma^N(n') = n'$.  We call sequence $\{n'=C_\gamma^0(n'),\, C_\gamma(n'),\cdots, C_\gamma^{N-1}(n')\}$ an integral loop of length $N$ starting at $n'$ and associated with the $3x+\gamma$ Problem.  The existence of an integral loop other than $\{1,4,2\}$ would be a counterexample of the Collatz Conjecture.

Let $n$ start an integral loop of length $N$ associated with the $3x+\gamma$ Problem.  Let $\rho$ and $\nu$ be the count of the number of zeros and ones in a single period of $B = \left(\chi_\gamma^i(n)\right)_{i\geq 0}$.  In a single period of $B$, let $m_j$ denote the number of zeros between the $(j-1)$th and $j$th one.  Let $\mathcal{M}_n$ be the matrix associated to $n$ whose elements are the sequential products of $2^{m_j}$ (e.g. $(2^{m_0},2^{m_0+m_1},2^{m_0+m_1+m_2},...)$).  Let $p$ be a prime factor for all the terms in the integral loop starting with $n$ with multiplicity $a>0$.  Suppose also that $p$ is a prime factor of $\gamma$ and $2^\rho - 3^\nu$ with multiplicity $b$ and $c$, respectively.  Finally assume that $c>b-a$.  Then $\det(\mathcal{M}_n) \equiv 0 \pmod p$.  We do find examples of this property.  Let $\nu$ be prime.  Let $z_j = 2^{(m_j+2m_{j+1}+3m_{j+2}+...)/\nu}$ be a weighted arithmetic average of the $m_j$.  We prove that if $p$ is a prime factor of $2^\rho-3^\nu$ distinct from $\nu$ so that the residue class $p \pmod \nu$ generates the whole group $\mathbb{Z}_\nu^{*}$ then $p\nmid \det(\mathcal{M}_n)$ if and only if $p\nmid(z_1+...+z_\nu)$ and for any $1\leq i<j\leq \nu$ we have $z_i \not\equiv z_j \pmod p$.  By this, we give an interesting property for
the integral loop.

\vspace{0.5cm}

Keywords: Integral Loops, Trajectories, Characteristic Trajectories, Circulant Matrix, Number Theoretic Functions, $3x+1$ Problem.

\chapter{Introduction}

\section{Organization of This Thesis}

The chapters of this thesis are organized to direct the reader to Theorem \ref{central_theorem_X}.  Chapter one provides the historical, contemporary, and current research into the $3x+1$ Problem.  It closes by presenting a summary of the key results found, as well as relevant notation.  Chapter two presents the Integral Loop Formula in Equation (\ref{integral_loop_formula_equ}) and the integral loop vector in Equation (\ref{int_loop_vec}).  The integral loop vector is the primary mathematical tool used throughout this thesis.  Chapter three focuses on presenting the Integral Loop Formula in terms of the matrix $\mathcal{M}_n$ given in Definition \ref{m_n}.  We find that the the integral loop vector is the solution to the matrix product presented in Equation (\ref{central_2}).  We find the associated circulatant matrix $\mathcal{R}_n$ whose determinant, up to a sign change, is identical to $\mathcal{M}_n$.  Finally, this property was exploited yielding the result of Theorem \ref{central_theorem_X}.  Chapter four explores some additional topics related to the findings here, but were not suitable for the thesis proper.

%In chapter four we show how the Hadamard product can be used to generate a matrix whose columns multiply to one as seen in Equation (\ref{cross+dot}).  Then, we find that the determinant of $\mathcal{M}_n$ oftentimes contains the same prime factors as $\gamma$.  Finally, this property was exploited yielding the result of Theorem \ref{no3}.  Chapter five explores some additional topics related to the findings here, but were not suitable for the thesis proper.

\section{Historical Background}

The story of the $3x+1$ Problem formally begins at the 1950 International Congress of Mathematics with its creator: Dr. Collatz\cite{lagarias19853x+}.  At the Congress, Dr. Collatz told many of his colleagues about the $3x+1$ Problem.  Most notably  were Dr. Coxeter, Dr. Kakutani, and Dr. Ulam \cite{lagarias2010overview}.  In 1952 Dr. Brian Thwaites independently created the $3x+1$ Problem \cite{lagarias19853x+}.  Additionally in 1952, Dr. Collatz accepted a position at Hamburg University and presented the $3x+1$ Problem to his colleague Dr. Hasse \cite{whiteman1991memoriam}.  It is through these four (Dr. Coxeter, Dr. Kakutani, Dr. Ulam, and Dr. Hasse) that the problem would spread by word of mouth to many academic institutions worldwide.  This process began when Dr. Hasse presented the $3x+1$ Problem to Syracuse University, leading to the nick-names \textit{The Syracuse Problem} as well as \textit{Hasse's Algorithm} \cite{lagarias2010overview}.  In 1960 Dr. Kakutani presented the problem at Yale and the University of Chicago leading to the nick-name \textit{Kakutani's Problem} \cite{lagarias19853x+}.  From there it went to the Massachusetts Institute of Technology where the $3x+1$ Problem was checked for all starting values less than sixty million \cite{gardner1972mathematical}.

It was not until 1971 when the $3x+1$ Problem first appeared in print in a set of memorial lecture notes by Dr. H.S.M. Coxeter \cite{coxeter2010cyclic}.  Nevertheless, it remained an underground problem appearing as a novelty in Scientific American in 1972 \cite{gardner1972mathematical}.  In 1975 Dr. Hasse wrote about the problem \cite{hasse1975unsolved}, again putting it into print. In the mid-seventies Dr. Ulam spread the problem to Los Alamos National Laboratory \cite{lagarias2010overview} where it was received by Dr. Everett who wrote about the problem in 1977 \cite{everett1977iteration}.  Around this time the $3x+1$ Problem become an internationally known problem.  The $3x+1$ Problem remained unsolved to this day.  

%\cite{nurnberger2010gunter}

\section{Modern Progress}

\begin{definition}\normalfont
The following mappings $C$ (Collatz) and $T$ ($3x+1$) map the integers to the integers in the following form:
\begin{equation}
\begin{aligned}
    C(n)&:=\left\{
                \begin{array}{ll}
                 3n+1, n\equiv 1\pmod{2},\\
                  \frac{n}{2},  n\equiv 0\pmod{2}
                \end{array}
              \right.
\end{aligned},
\begin{aligned}
    T(n)&:=\left\{
                \begin{array}{ll}
                 \frac{3n+1}{2}, n\equiv 1\pmod{2},\\
                  \frac{n}{2},  n\equiv 0\pmod{2}
                \end{array}
              \right.
\end{aligned}.
\end{equation}
\end{definition}

The conjecture regarding either $T$ or $C$ is whether for every positive integer $n$, there exists a positive integer $k$ such that $C^k(n)=1$, or $T^k(n)=1$, respectively.  This makes the conjecture a $\Pi_2$ statement that can be written in the form: $(\forall n)(\exists k)(C^k (n)=1)$ \cite{michel2010generalized}.
 
There are variants of `Collatz-like' problems that are sufficiently complex to be considered undecidable.  Fix $p\in \mathbb{Z}_{\geq 2}$ and $a_0,b_0,...a_{p-1},b_{p-1} \in \mathbb{Q}$.  We call a function $g:\mathbb{Z}\rightarrow\mathbb{Z}$ a \textit{generalized Collatz mapping}\cite{michel2010generalized} if for all $i$, $0\leq i \leq p-1$, and all $n\in \mathbb{Z}$,

\begin{equation}\label{ainbi}
g(n) = a_in+b_i ,\text{ if } n\equiv i \pmod p.\end{equation}

\noindent In this definition of $g$, $a_i ,b_i$ are chosen so that $a_i n+b_i$ is always integral when $n\equiv i \pmod p$.  Conway proved that given a function $g$ and an integer $n$, there exists no algorithm which determines if there exists positive integer $k$ such that $g^k (n)=1$ \cite{van2005collatz}\cite{michel2010generalized}\cite{conway1972iterations}.  This question was demonstrated to be undecidable by comparing it to fractional-linear mapping similar to Equation (\ref{ainbi}) with all $b_i = 0$ \cite{conway1972iterations} \cite{michel2010generalized}.  The fractional-linear mapping itself is undecidable because it is a special case of the FRACTRAN algorithm, which had been demonstrated to be a universal computer \cite{michel2010generalized}.  

Now consider a second kind of generalization.  Fix $p\in\mathbb{Z}_{>0}$ and $a_0,b_0,...,a_{p-1},b_{p-1}\in\mathbb{Q}$.  We call a function $h:\mathbb{Z}_{>0}\rightarrow\mathbb{Z}_{>0}$ a \textit{Collatz-type iteration function} \cite{rawsthorne1985imitation} if for all $i$, $0\leq i \leq p-1$ and all $n\in\mathbb{Z}_{>0}$,

\begin{equation}\label{cziterate}
h(n) = \frac{a_i n + b_i}{p}, \text{ if } n \equiv i \pmod p .
\end{equation}

In this definition of $h$, $a_i,b_i$ are chosen so that $a_i n+b_i \equiv 0 \pmod p$.  Let $n\in\mathbb{Z}_{>0}$ be given.  We call the sequence of iterations $\{n,h(n),h^2(n),h^3(n),...\}$ the \textit{trajectory of $n$}.  Additionally, we say the trajectory of $n$ \textit{converges to an integral loop} if there exists a $k, n'>0$ so that $n'$ is contained in the trajectory of $n$ and $h^k(n') = n'$.  It is proven that if $h$ distributes evenly among all of the residue classes modulo $p$, then almost all trajectories converge to an integral loop\cite{rawsthorne1985imitation}.

A similar `Collatz-like' mapping appears in Equation (\ref{3xt}):

\begin{equation}\label{3xt}
\begin{aligned} f: \mathbb{Z} &\rightarrow \mathbb{Z} \\
    f(n)&=\left\{
                \begin{array}{ll}
                  \frac{n}{2},\text{ } n \equiv 0 \pmod 2\\
                  3n+t,\text{ } n\in A_t
                \end{array}
              \right.
\end{aligned},
\end{equation}

\noindent where $A_t$ is a recursive set partitioning the odd positive integers, and $t$ = -9,-8,...,8,9 \cite{everett1977iteration}\cite{LEHTONEN2008596}.  By encoding Turing machines into sequences consisting of zeros and twos in base three, one can choose two integers in the set $\{-9,...,9\}$ ensuring that when the Turing machine halts, $f$ descends down to one without overlapping other Turing machines encoded in this manner\cite{LEHTONEN2008596}.  Thus a never halting computation is coded by a trajectory of $f$ diverging to infinity\cite{michel2010generalized}.  Suppose we have some algorithm that determines whether given $n\in\mathbb{Z}_{>0}$, there exists some $k>0$ so that $f^k(n)=1$.  Then given a universal Turing machine, such an algorithm would decide the halting problem for this machine.  This is a contradiction\cite{michel2010generalized}.

Consider the characteristic mapping $\chi$ of $T$ defined below.

\begin{definition} The characteristic mapping $\chi$ takes $N$ iterative steps (indexed by the variable $i$) of the $T$ mapping to a boolean variable determined by the parity of the value of the function at each step.
\begin{equation}
\begin{aligned}
    \chi_i(n)&=\left\{
                \begin{array}{ll}
                  1,\text{ } T^i(n) \equiv 1 \pmod2\\
                  0,\text{ } T^i(n)\equiv 0 \pmod2
                \end{array}
              \right. \\
n &\rightarrow \{\chi_0,\chi_1,\chi_2,...,\chi_{N-1}\}
\end{aligned}.
\end{equation}
\end{definition}

Induction can prove that the binary sequence of length $N$ is unique for integers $n < 2^N$ \cite{everett1977iteration}.  Hence, there exists an injective mapping between the integers $n<2^N$ and binary sequences of length $N$ \cite{everett1977iteration}.  Furthermore from a statistical standpoint, for every $n<2^N$, the probability that each inductive step of $\chi_i$ being a 1 or 0 approaches one-half (like a coin flip) for large $N$ \cite{dolan1987generalization}.  This property allows for the exploration of an `average growth rate' of the $T$ mapping. The average growth rate of a trajectory where $N-1$ odd numbers appears converges to 3/4 as $N$ goes to infinity  \cite{van2005collatz}.  It is shown that for the starting integers $n$ for which a value less than $n^{\ln(3)/\ln(4)}$ is reached after a finite number of steps along the trajectory of $n$ has an asymptotic density of 1 \cite{ivan1994density}.

%One could consider addressing classes of integers which converge under different kinds of patterns \cite{colussi2011convergence}.  Various trajectories can be reversed by generating a power series.  These convergence classes create upper and lower bounds on trajectories which determine whether an integer may diverge, converge, or behave periodically \cite{alves2005linear}.  Similarly, many trajectories can be bound within an $n\times n$ matrix.  Iterations of the Collatz mapping are represented by matrix products \cite{alves2005linear}.  The problem regarding periodic cycles can be reduced to subsets of the integers \cite{colussi2011convergence}.

\section{Current Developments in Research}

In the last twenty years, even further progress has been made with respect to the $3x+1$ Problem.  Oftentimes, this progress is made by considering further generalizations of the original question.  Let $p$ and $g$ be coprime integers with $g>p\geq 2$.  Let $h(x)$ be a periodic function satisfying: $h(x+p) = h(x)$,  $x+h(x) \equiv 0 \pmod p$, and $0 < |h(x)| < g$ for all $x$ not divisible by $p$.  Then consider the following mapping of the integers to the integers,

\begin{equation}\label{t-prime}
\tilde{T}(x) = \frac{gx+h(gx)}{p^k},
\end{equation}

\noindent where $k$ is uniquely chosen so that the result is not divisible by $p$.  It can be shown that for maps of this form \cite{kontrovich2006structure} infinity is either a point of convergence or a repelling point of the typical trajectory.  That is to say a typical trajectory returns to the origin if $\log (g) - \frac{p}{p-1} \log(p) <0$.

The $3x+1$ Problem can also be classified in the collection of Residue-Class-Wise Affine functions \cite{KOHL2007322}.  A mapping $f:\mathbb{Z}\rightarrow\mathbb{Z}$ is called \textit{Residue-Class-Wise Affine} if there is a positive integer $m$ such that $f$ is affine on the residue classes modulo $m$.  Let the smallest such $m$ be the referred to as the \textit{modulus} of $f$ and write Mod($f$).  Additionally, $f$ is called \textit{tame} if the set $\{\text{Mod}(f^{(k)}) | k \in \mathbb{Z}_{>0}\}$ is bounded, and \textit{wild} otherwise.  In this context, the $3x+1$ mapping T can be shown to be wild\cite{KOHL2007322}.

Another emerging occurrence of the $3x+1$ Problem is in Computational Theory.  In fact, it is proven\cite{Michel2014SimulationOT} that the $3x+1$ Conjecture is true if and only if for all positive integers (written in base three), there exists an $n \geq 0$ so that Turing Machine $M_1$\cite{Michel2014SimulationOT} eventually reaches the configuration $^\omega b 0^n (A1)b^\omega$.

%Recall that a Turing Machine is an abstract representation of a computer with the following properties \cite{Michel2014SimulationOT}

%\begin{enumerate}
%\item It reads/writes on one tape, infinite on both sides made of cells containing symbols.
%\item One reading and writing head.
%\item A set $Q = \{A,B,...\}$ of states with a halting state.
%\item A set $\Sigma = \{b,0,1,...\}$ of symbols, where $b$ is the blank symbol (oftentimes $\Sigma = \{0,1\}$ with $0$ serving the role of the blank symbol).
%\item A next move function $\sigma : Q \times \Sigma \rightarrow \Sigma %\times \{L,R\} \times (Q \cup \{Halt\})$.
%\end{enumerate}

%For example, consider the following machine referred to as $M_1$ \cite{Michel2014SimulationOT}:

%\begin{figure}[!h]
%\centering
%\begin{tabular}{|c||c|c|c|c|}
%\hline
%$M_1$ & b & 0 & 1 & 2 \\
%\hline
%\hline
%A & bLC & 0RA & 0RB & 1RA \\
%\hline
%B & 2LC & 1RB & 2RA & 2RB \\
%\hline
%C & bRA & 0LC & 1LC & 2LC \\
%\hline
%\end{tabular}

%\caption{The three state four symbol Turing Machine $M_1$.}
%\label{fig:m_1}
%\end{figure}

This is not a unique occurrence.  Collatz-like functions have their place in theories regarding the \textit{Busy Beaver Game}.  Consider the collection of all Turing Machines with $n$ states (not including the halting state) and two symbols (not including the blank symbol) initialized on a completely blank tape that halts after a finite number of steps.  Let $S(n)$ denote the maximum number of computation steps made by each machine before halting, and $\Sigma(n)$ the maximum number of symbols 1 left remaining \cite{Michel2015}.  It has been proven that $S$ and $\Sigma$ are non-computable functions, so that for any computable function $f$ there exists an integer $N$ such that for all $n \geq N$, $S(n) > \Sigma(n) > f(n)$ \cite{Michel2015}.  Some, but not all of these machines reach the maximum value of $S(n)$ by simulating Collatz-like functions on their input.  It is hypothesized \cite{BB2020Aaronson} that progress on the Collatz Conjecture and Collatz-like functions go hand in hand with progress on determining the Busy Beaver champions.

%This is true, for example the current champion of the five-state Busy Beaver \cite{BB2020Aaronson}:

%\begin{equation}
%\begin{aligned}
%    g(n)&:=\left\{
%                \begin{array}{ll}
%                 \frac{5n+18}{3}, \text{if $x \equiv 0 \mod 3$},\\
%                 \frac{5n+22}{3}, \text{if $x \equiv 1 \mod 3$},\\
%                  \text{Halt}, \text{if $x \equiv 2 \mod 3$.}
%                \end{array}
%              \right.
%\end{aligned}
%\end{equation}

%The machine appears as follows (for compactness the alphabet appears vertically, and the states horizontally):

%\begin{figure}[!h]
%\centering

%\begin{tabular}{|c||c|c|c|c|c|}
%\hline
%$n=5$ & A & B & C & D & E \\
%\hline
%\hline
%0 & 1RB & 1RC & 1RD & 1LA & Halt \\
%\hline
%1 & 1LC & 1RB & 0LE & 1LD &0LA \\
%\hline
%\end{tabular}

%\caption{The five state two symbol machine simulating $g$.}
%\label{fig:g_n}
%\end{figure}

Last, but certainly not least was the development regarding the density of orbits under the $3x+1$ mapping.  Let $N\in\mathbb{Z}_{>0}$ be given.  Let $C_{min}(N) = min_{i\geq 0}\{C^i(N)\}$.  It was shown \cite{tau2020} that for any function $f:\mathbb{Z}_{>0} \rightarrow \mathbb{R}$ with the property $\lim_{N\rightarrow \infty} f(N) = +\infty$, we have $C_{min} (N) \leq f(N)$ for almost all $N \in \mathbb{Z}_{>0}$.

\section{This Thesis' Contribution}

One could argue that the heart of this thesis is to be the spiritual successor of the work made by J. F. Alves, M.M. Gra\c ca, et. al\cite{alves2005linear}.  Their work framed the $3x+1$ Problem in the terms of linear algebra.  They succeeded in defining the compositions of the mapping $C$ in terms of matrix products.  Here we study a generalization of the $3x+1$ Problem, the $3x+\gamma$ Problem.

Fix $\gamma \in\mathbb{Z}_{>0}^{odd}$.  Let $C_\gamma : \mathbb{Z}_{>0} \rightarrow \mathbb{Z}_{>0}$ be defined as,

  \begin{equation}
    C_\gamma(n) =
    \begin{cases}
      3n+\gamma & \text{if } n \equiv 1 \pmod 2, \\
     \frac{n}{2}        & \text{if } n \equiv 0 \pmod 2.
    \end{cases}
  \end{equation}

\noindent Furthermore, for $k\geq 0$, denote function composition as $C_\gamma^{k+1}(n) = C_\gamma^k\circ C_\gamma(n)$ and $C_\gamma^0(n) = n$.  Let $n\in\mathbb{Z}_{>0}$ be given.  Let $\{C_\gamma^0(n),C_\gamma^1(n),...\}$ denote the \textit{trajectory} of $n$.  Here, we consider the limiting case when a given trajectory of a positive integer $n$ is bounded.  That is to say, for $\gamma\in\mathbb{Z}_{0}^{odd}$ and $n\in\mathbb{Z}_{>0}$ given, there exists some finite $\Omega>0$ such that the trajectory of $n$ is bounded above by $\Omega$.  This necessarily implies there exists some $n',N\in\mathbb{Z}_{>0}$ so that $n'$ is contained in the trajectory of $n$ and $C_\gamma^N(n') = n'$.  

Fix $n\in\mathbb{Z}_{>0}$.  If there exists some $N\in\mathbb{Z}_{>0}$ so that $C_\gamma^N(n) = n$ then we say $n$ starts an \textit{integral loop}.  Additionally, let $N = \min(\{i \in \mathbb{Z}_{>0} | C_\gamma^i(n) = n\})$ be the least positive integer so that $C_\gamma^N(n) = n$.  Then we say $n$ starts an integral loop of \textit{length} $N$.

Fix $\gamma\in\mathbb{Z}_{>0}^{odd}$. Let $\chi_\gamma: \mathbb{Z}_{>0}\rightarrow \{0,1\}$ be defined as,

  \begin{equation}
    \chi_\gamma(n) =
    \begin{cases}
      1 & \text{if } C_\gamma(n) \equiv 1 \pmod 2, \\
      0        & \text{if } C_\gamma(n) \equiv 0 \pmod 2.
    \end{cases}
  \end{equation}
  
\noindent Let $\chi_\gamma^0(n) \equiv n \pmod 2$.  Furthermore, for $k\geq 0$, denote function composition as $\chi_\gamma^{k+1}(n) = \chi_\gamma \circ C_\gamma^{k+1}(n)$.  Fix $n\in\mathbb{Z}_{>0}$.  Let $\{\chi_\gamma^0(n),\chi_\gamma^1(n),...\}$ denote the \textit{characteristic trajectory} of $n$.

\begin{example}
    The first eight terms of the trajectory and characteristic trajectory induced by $n=7$ and $\gamma=1$ appears below respectively,

    \begin{equation*}
        \{7,22,11,34,17,52,26,13\} \text{ and } \{1,0,1,0,1,0,0,1\}.
    \end{equation*}
\end{example}

Let $\{0,1\}^{**}$ denote the set of binary sequences of infinite length.  Fix $\gamma\in\mathbb{Z}_{>0}^{odd}$ and $n\in\mathbb{Z}_{>0}$.  We say the binary sequence $B\in\{0,1\}^{**}$ is \textit{associated} to the characteristic trajectory of $n$ if $B = (\chi_\gamma^i(n))_{i\geq 0}$. 

Suppose $n$ starts an integral loop of length $N$.  Then the binary sequence $B$ associated to $n$ has a repeating period of length $N$.  Let $\rho$ be the number of zeros in a single period of $B$.  Let $\nu$ be the number of ones in a single period of $B$.  Let $m_0$ denote the number of leading zeros.  Let $m_j$ denote the number of zeros between the $(j-1)$th and the $j$th one.  Let $m_\nu$ be the remaining zeros after the last one.  Note that $\sum_j m_j = \rho$.

\begin{example}
    The integral loop induced by $n=157$ and $\gamma=175$ appears below,
    
    \begin{equation*}
        \{157,646,323,1144,572,286,143,604,302,151,628,314\}.
    \end{equation*}
\end{example}

\begin{proposition}\label{AAA}\textit{
Fix $\gamma \in \mathbb{Z}_{>0}^{odd}$ and $n \in \mathbb{Z}_{>0}$.  Let $B\in \{0,1\}^{**}$ such that $B = (\chi_\gamma^i(n))_{i\geq 0}$.  Suppose that $N$ is the smallest positive
integer such that $C_\gamma^{N+1}(n) = n$.  Let $\rho$ and $\nu$ be the count of the number of zeros and ones, respectively, in $B$ for all $i$ between $0$ and $N$ inclusive.
\begin{enumerate}
\item If n is odd, let $m_1,...,m_{v-1}$ be the number of zeros between the $(j-1)$th and $j$th one.  Let $m_0 = 0$ and $m_\nu$ be the remaining zeros after the last one.
\item If n is even, let $m_0$ be the number of leading zeros.  Observe that $n = 2^{m_0}\tilde{n}$ for $\tilde{n}$ being odd.  So we obtain $m_1,...,m_{\nu-1}$ by (1) and $\tilde{n}$.
\end{enumerate}
Then $n$ satisfies the following equality}

\begin{equation} \label{AAB}
n = \frac{\gamma2^{m_0}(3^{\nu-1}+\sum_{r=1}^{\nu-1} 3^{\nu-1-r}2^{\sum_{j=1}^r m_j})}{2^\rho - 3^\nu}.
\end{equation}
\end{proposition}

\begin{remark} Proposition \ref{AAA} can be deduced as a corollary of Proposition \ref{tt_Cg}.
\end{remark}

\begin{example}
    Consider again the integral loop induced by $n=157$ and $\gamma=175$.  With respect to Equation (\ref{AAB}) we find that,

    \begin{equation*}
    \begin{aligned}
        n &= \frac{175(3^{4-1}+3^{4-2}\times 2^{1}+3^{4-3}\times 2^{1+3}+3^{4-4}\times 2^{1+3+2})}{2^8-3^4}, \\
        &=\frac{175(27+9\times 2 + 3\times 16 + 1\times 64)}{175} = 157.
    \end{aligned}
    \end{equation*}
\end{example}

Fix $\gamma\in\mathbb{Z}_{>0}^{odd}$, $n\in\mathbb{Z}_{>0}^{odd}$, and $N\in\mathbb{Z}_{>0}$.  Suppose that $n$ starts an integral loop of length $N$.  Let $i\geq 0$ be the indexing of $B$ so that $B = (\chi_\gamma^i(n))_{i\geq 0}$.  Let $i_j$ be the sub-indexing of $i$ so that for all $0 \leq j \leq \nu-1$, we have $\chi_\gamma^{i_j}(n) \equiv 1 \pmod 2$.  Note that $i_0 = 0$.  Let $e_j$ denote the standard basis of $\mathbb{R}^\nu$.  We call the array $[n]\in\mathbb{Z}^\nu$ the \textit{integral loop vector}, which appears in the following form:

\begin{equation}
[n] = \begin{bmatrix} C_\gamma^{i_0}(n) \\ C_\gamma^{i_1}(n) \\ \vdots \\ C_\gamma^{i_{\nu-1}}(n) \end{bmatrix} = \sum_{j=1}^\nu C_\gamma^{i_{j-1}}(n)e_j .
\end{equation}

\noindent Let $\overrightarrow{\mathfrak{3}}^\nu \in \mathbb{Z}^\nu$ be the vector whose entries are ascending powers of three.

\begin{equation}
    \overrightarrow{\mathfrak{3}}^\nu = \sum_{j=1}^\nu 3^{\nu-j} e_j .
\end{equation}

\begin{example}
    Let $\nu = 5$.

\begin{equation*}\overrightarrow{\mathfrak{3}}^5 = \begin{bmatrix} 3^4 \\ 3^3 \\ 3^2 \\3^1 \\3^0 \end{bmatrix}.\end{equation*}
\end{example}

\noindent Let $\mathcal{P}$ be the $\nu \times \nu$ permutation matrix characterized by the following formula:
 
\begin{equation}
\mathcal{P} = \sum_{j=1}^\nu e_{j,|j|_\nu+1},
\end{equation}

\noindent where $|j|_\nu$ denotes $j \mod \nu$.  We define the $\nu\times\nu$ matrix $\mathcal{L}_n$ associated to $n$ to be the diagonal matrix whose entries are powers of two determined by the following formula:

\begin{equation}
\mathcal{L}_n = \sum_{j=1}^\nu 2^{m_j} e_{j,j}.
\end{equation}

\noindent We define the $\nu\times\nu$ matrix $\mathcal{D}_n$ associated to $n$ to be the result of the matrix product of $\mathcal{L}_n$ and $\mathcal{P}$.

\begin{equation}\label{DDD}
\mathcal{D}_n = \mathcal{L}_n \times \mathcal{P} = \sum_{j=1}^\nu 2^{m_j}e_{j,|j|_\nu+1}.
\end{equation}

\noindent We define the following $\nu \times \nu$ matrix associated to $n$, referred to as $\mathcal{M}_n$, below:

\begin{equation}\label{BBB}
\mathcal{M}_n = \sum_{k=1}^\nu \mathcal{D}_n^{k-1}\sum_{j=1}^\nu e_{j,k}.
\end{equation}

\noindent Where $\mathcal{D}_n^k$ denote the $k$th power of $\mathcal{D}_n$ and $\sum_{j=1}^\nu e_{j,k}$ denotes the matrix whose $k$th column is all ones and is zero everywhere else.

\begin{example}
    Consider the integral loop induced by $n=85$ and $\gamma=47$.  The matrix $\mathcal{M}_{85}$ appears below,

    \begin{equation*}
        \mathcal{M}_{85} = 
        \begin{bmatrix} 
        1 & 2 & 8 & 16 \\
        1 & 4 & 8 & 64 \\
        1 & 2 & 16 & 32 \\
        1 & 8 & 16 & 64
        \end{bmatrix}.
    \end{equation*}
\end{example}

\begin{proposition}\textit{Fix $\gamma \in \mathbb{Z}^{odd}_{>0}$, $n\in \mathbb{Z}^{odd}_{>0}$, and $N\in\mathbb{Z}_{>0}$.  Suppose further that $n$ starts an integral loop of length $N$.  Let $B \in \{0,1\}^{**}$ such that $B = (\chi_\gamma^i(n))_{i \geq 0}$.  Let $\rho$ and $\nu$ be the count of the number of zeros and ones in a single period of $B$.  Let $[n]$ be the associated integral loop vector of $n$.  Let the matrix $\mathcal{M}_n$ be given by Equation (\ref{BBB}).  Then $[n]$ and $\mathcal{M}_n$ satisfy the following equality,}

\begin{equation}\label{FFF}
[n] = \frac{\gamma}{2^\rho - 3^\nu}\mathcal{M}_n \overrightarrow{\mathfrak{3}}^\nu.
\end{equation}

\noindent \textit{Where $\mathcal{M}_n \overrightarrow{\mathfrak{3}}^\nu$ denotes the product of the $\nu \times \nu$ matrix $\mathcal{M}_n$ and the $\nu \times 1$ column vector $\overrightarrow{\mathfrak{3}}^\nu$.}

\end{proposition}

\begin{example}
    Again consider the integral loop induced by $n=85$ and $\gamma=47$.  We verify Equation (\ref{FFF}),

    \begin{equation*}
        [85] = \begin{bmatrix} 85 \\ 151 \\ 125 \\ 211 \end{bmatrix} = \frac{47}{47} \begin{bmatrix} 
        1 & 2 & 8 & 16 \\
        1 & 4 & 8 & 64 \\
        1 & 2 & 16 & 32 \\
        1 & 8 & 16 & 64
        \end{bmatrix} \begin{bmatrix} 27 \\ 9 \\ 3 \\ 1 \end{bmatrix} = \frac{47}{2^7-3^4}\begin{bmatrix} 
        1 & 2 & 8 & 16 \\
        1 & 4 & 8 & 64 \\
        1 & 2 & 16 & 32 \\
        1 & 8 & 16 & 64
        \end{bmatrix} \begin{bmatrix} 27 \\ 9 \\ 3 \\ 1 \end{bmatrix}.
    \end{equation*}
\end{example}

\begin{proposition}\label{delta_2}\textit{
    Fix $\gamma,n \in \mathbb{Z}_{>0}^{odd}$ and $N\in\mathbb{Z}_{>0}$.  Let $n$ start an integral loop of length $N$.  Let $B\in \{0,1\}^{**}$ such that $B = (\chi_\gamma^i(n))_{i\geq 0}$.  Let $\rho$ and $\nu$ be the count of the number of zeros and ones in a single period of $B$.  Let $[n]$ be the associated integral loop vector of $n$.  Let $\mathcal{M}_n$ be constructed as by Equation (\ref{m_n}).  Let $(i_j)_{j=0}^{\nu-1}$ be the indexing sequence associated to $n$.  Let $s = \gcd(n,C_\gamma^{i_1}(n))$ and $p$ a prime factor of $s$.  Define $a,d\in\mathbb{Z}_{\geq 0}$ as the maximal exponents such that $p^a$ divides every component of $[n]$ and $p^d$ divides every component of $\mathcal{M}_n \overrightarrow{\mathfrak{3}}^\nu$.  Define $b,c\in\mathbb{Z}_{\geq 0}$ as the maximal exponents so that $p^b$ divides $\gamma$ and $p^c$ divides $(2^\rho - 3^\nu)$.  Then $a,b,c,$ and $d$ satisfy the following equality:}

    %Define $a,b,c,d \in \mathbb{Z}_{\geq 0}$ and $[n'], {\mathcal{M}_n \overrightarrow{\mathfrak{3}}^\nu}' \in {\mathbb{Z}_{>0}^\nu}^{odd}$, $\gamma', (2^\rho-3^\nu)' \in \mathbb{Z}_{>0}^{odd}$ such that $[n] = p^a [n']$, $\gamma = p^b \gamma'$, $2^\rho - 3^\nu = p^c(2^\rho - 3^\nu)'$, and $\mathcal{M}_n \overrightarrow{\mathfrak{3}}^\nu = p^d{\mathcal{M}_n \overrightarrow{\mathfrak{3}}^\nu}'$.

    \begin{equation*}
        a-b+c-d = 0.
    \end{equation*}
\end{proposition}

\begin{theorem}
    \textit{We will use the notation presented in Proposition \ref{delta_2}.  Furthermore assume that $c>b-a$. Fix $\gamma, n \in \mathbb{Z}_{>0}^{odd}$ and $N\in \mathbb{Z}_{>0}$.  Suppose further $n$ starts an integral loop of length $N$.  Let $B \in \{0,1\}^{**}$ such that $B = (\chi_\gamma^i(n))_{i\geq 0}$.  Let $\rho$ and $\nu$ be the count of the number of zeros and ones in a single period of $B$.  Let $[n]$ be the associated integral loop vector of $n$.  Let $\mathcal{M}_n$ be constructed as by Equation (\ref{m_n}).  Let $p$ be a prime factor of $[n]$ of multiplicity $a>0$, then $\det(\mathcal{M}_n) \equiv 0 \pmod p$.}
\end{theorem}

%\begin{theorem}\label{CgC}\textit{Fix $\gamma,\gamma' \in \mathbb{Z}_{>1}^{odd}$, $n \in \mathbb{Z}_{>0}^{odd}$, and $N\in\mathbb{Z}_{>0}$.  Suppose further that $n$ starts an integral loop of the $3x+\gamma$ problem of length $N$.  Let $B \in \{0,1\}^{**}$ such that $B = (\chi_\gamma^i(n))_{i\geq 0}$.  Let $\rho$ and $\nu$ be the count of the number of zeros and ones in a single period of $B$.  Let $\mathcal{M}_n$ be constructed as by Equation (\ref{BBB}), and assume it is invertible over $\mathbb{C}$.  Finally suppose the ratios $\kappa = \frac{\gamma}{\gamma'}$ and $\theta = \frac{2^\rho-3^\nu}{\gamma'}$ are both integers.  If $[n]/\kappa$ is also an integral loop vector of the $3x+\gamma'$ problem, then $\det(\mathcal{M}_n) \equiv 0 \pmod \theta$.}
%\end{theorem}

\begin{example}\label{ex1}
Consider the integral loop induced by $n=65$ and $\gamma = 65455$.  We find that $\overrightarrow{\mathfrak{3}}^4$ is in the nullspace of $\mathcal{M}_{65} \pmod {5, 13}$.

\begin{equation*}
    \begin{bmatrix} 
    1 & 2 & 4 & 8 \\
    1 & 2 & 4 & 32768 \\
    1 & 2 & 16384 & 32768 \\
    1 & 8192 & 16384 & 32768
    \end{bmatrix}
    \begin{bmatrix}
        27 \\ 9 \\ 3 \\ 1
    \end{bmatrix}
    =
    \begin{bmatrix}
        65 \\ 32825 \\ 81965 \\ 155675
    \end{bmatrix}
    \equiv
    \begin{bmatrix}
        0 \\ 0 \\ 0 \\ 0
    \end{bmatrix} \pmod {5 \text{ and } 13}.
\end{equation*}

\noindent Now consider the following integral loops, which we denote by $A$, $B$, $C$, and $D$:

\begin{tabular}{c|c|c|l}
   & $n$ & $\gamma$ \\
   \hline
  $A$ & 65 & 65455 & \begin{tabular}{l}$\{65, 65650, 32825, 163930, 81965, 311350, 155675,$ \\ $532480, 266240, 133120, 66560, 33280, 16640, 8320, $ \\ $ 4160, 2080, 1040, 520, 260, 130\}$\end{tabular} \\
  $B$ & 13 & 13091 & \begin{tabular}{l}$\{13, 13130, 6565, 32786, 16393, 62270, 31135, 106496,$ \\ $53248, 26624, 13312, 6656, 3328, 1664, 832, 416,$ \\ $208, 104, 52, 26\}$\end{tabular} \\
  $C$ & 5  &  5035 & \begin{tabular}{l}$\{5, 5050, 2525, 12610, 6305, 23950, 11975, 40960, 20480,$ \\ $10240, 5120, 2560, 1280, 640, 320, 160, 80, 40, 20, 10\}$\end{tabular} \\
  $D$ & 1  &  1007 & \begin{tabular}{l} $\{1, 1010, 505, 2522, 1261, 4790, 2395, 8192, 4096, 2048,$ \\$1024, 512, 256, 128, 64, 32, 16, 8, 4, 2\}$ \end{tabular}
\end{tabular}

\noindent We then find that the integral loop $B = A/5$, $C =A/13$, and $D = A/65 = B/13 = C/5$.
\end{example}

\begin{remark}
A discussion of this example is presented in chapter three, section two.
\end{remark}

%\begin{proposition}\label{CCC}\textit{
%    Let $\gamma \in \mathbb{Z}_{>1}^{odd}$, $n \in \mathbb{Z}_{>0}^{odd}$, and $N\in\mathbb{Z}_{>0}$ be given.  Suppose further that $n$ starts an integral loop of the $3x+\gamma$ problem of length $N$.  Let $B \in \{0,1\}^{**}$ such that $B = (\chi_\gamma^i(n))_{i\geq 0}$.  Let $\rho$ and $\nu$ be the count of the number of zeros and ones in a single period of $B$.  Let $\mathcal{M}_n$ be constructed as by Equation (\ref{BBB}), and assume it is invertible over ${\mathbb C}$.  Finally suppose the quantity $(2^\rho-3^\nu)$ is prime.  If $[n]/\gamma$ is also an integral loop vector of the $3x+1$ problem, then $\det(\mathcal{M}_n) \equiv 0 \pmod {2^\rho-3^\nu}$.}
%\end{proposition}

\begin{example}
    Consider the integral loop induced by $n=7$ and $\gamma = 7$.  We find that $\overrightarrow{\mathfrak{3}}^2$ is in the nullspace of $\mathcal{M}_7 \pmod 7$.

\begin{equation*}
\mathcal{M}_{7}\overrightarrow{\mathfrak{3}}^2 = 
    \begin{bmatrix}
        1 & 4 \\
        1 & 4
    \end{bmatrix} \begin{bmatrix} 3 \\ 1 \end{bmatrix} = \begin{bmatrix} 7 \\ 7 \end{bmatrix} \equiv \begin{bmatrix} 0 \\ 0 \end{bmatrix} \pmod 7 .
\end{equation*}

\noindent Consider the integral loop induced by $n=37$ and $\gamma = 37$.  We find that $\overrightarrow{\mathfrak{3}}^3$ is in the nullspace of $\mathcal{M}_{37} \pmod {37}$.

\begin{equation*}
\mathcal{M}_{37}\overrightarrow{\mathfrak{3}}^3 = 
    \begin{bmatrix}
        1 & 4 & 16 \\
        1 & 4 & 16 \\
        1 & 4 & 16
    \end{bmatrix} \begin{bmatrix}9 \\ 3 \\ 1 \end{bmatrix} = \begin{bmatrix} 37 \\ 37 \\ 37 \end{bmatrix} \equiv \begin{bmatrix} 0 \\ 0 \\ 0 \end{bmatrix} \pmod {37} .
\end{equation*}

\begin{tabular}{c|c|c|l}
     & n & $\gamma$ & \\
     \hline
    A & 7 & 7 & $\{7,28,14\}$ \\
    B & 37 & 37 & $\{37, 148, 74\}$\\
    C & 1 & 1 & $\{1,4,2\}$
\end{tabular}

\noindent We then find that the integral loop $C = B/37 = A/7$.
\end{example}

Fix $\gamma, n \in \mathbb{Z}_{>0}^{odd}$ and $N\in \mathbb{Z}_{>0}$.  Suppose further $n$ starts an integral loop of length $N$.  Let $B = \{0,1\}^{**}$ such that $B = (\chi_\gamma^i(n))_{\geq 0}$.  Let $\rho$ and $\nu$ be the count of the number of zeros and ones in a single period of $B$.  Let the quantities $m_j$ be given by Proposition \ref{AAA}.  Let $z_1,...,z_n$ be defined by,

    \begin{equation}\label{z_deff}
        z_i = 2^{\frac{1}{\nu}\sum_{k=1}^{\nu-1}k m_{|i+k-2|_\nu+1}}.
    \end{equation}

    \noindent Where $|i+k-2|_\nu$ is $i+k-2$ modulo $\nu$.  Then we define the circulant matrix $\mathcal{R}_n$ associated to $n$ as,

    \begin{equation}\label{r_nn}
        \mathcal{R}_n = \mathcal{R}(z_1,...,z_n) = \sum_{j=1}^\nu z_j \mathcal{P}^{j-1}.
    \end{equation}

\begin{proposition}
    \textit{Fix $\gamma, n \in \mathbb{Z}_{>0}^{odd}$ and $N\in \mathbb{Z}_{>0}$.  Suppose further $n$ starts an integral loop of length $N$.  Let $B = \{0,1\}^{**}$ such that $B = (\chi_\gamma^i(n))_{\geq 0}$.  Let $\rho$ and $\nu$ be the count of the number of zeros and ones in a single period of $B$.  Let $\mathcal{M}_n$ be constructed as by Equation (\ref{BBB}).  Let $\mathcal{R}_n$ be constructed by Equation (\ref{r_nn}).  Then the following equality is satisfied,}

    \begin{equation*}
        \left|\det(\mathcal{M}_n)\right| = \left|\det(\mathcal{R}_n)\right|.
    \end{equation*}
\end{proposition}

\begin{theorem}\label{central_theorem_XX}
    \textit{Fix $\gamma, n \in \mathbb{Z}_{>0}^{odd}$ and $N\in \mathbb{Z}_{>0}$.  Suppose further $n$ starts an integral loop of length $N$.  Let $B = \{0,1\}^{**}$ such that $B = (\chi_\gamma^i(n))_{\geq 0}$.  Let $\rho$ and $\nu$ be the count of the number of zeros and ones in a single period of $B$.  Assume $\nu$ is prime.  Let the quantities $m_j$ from Proposition \ref{AAA} satisfy Equation (\ref{z_deff}) so that $z_1,...,z_\nu$ are integers.  Let $\mathcal{R}_n$ be constructed by Equation (\ref{r_nn}). Let $p$ be a prime factor of $2^\rho - 3^\nu$ distinct from $\nu$ so that the residue class $p \pmod \nu$ generates the whole group $\mathbb{Z}_\nu^{*}$.  Then $p\nmid \det(\mathcal{M}_n)$ if and only if $p\nmid(z_1+...+z_\nu)$ and for some $1\leq i<j\leq \nu$ we have $z_i \not\equiv z_j \pmod p$.}
\end{theorem}

\begin{example}
    Consider the integral loop induced by $n=373$ and $\gamma = 1048333$.  In this case $m_1 = 4, m_2 = 2, m_3 = 2, m_4 = 4$ and $m_5 = 14$ such that $2^{20}-3^5 = 1048333 = 11\times 13\times 7331$.  We find that $z_1 = 8, z_2 = 8192, z_3 = 2048, z_4 = 256$ and $z_5 = 32$.  Thus $\mathcal{R}_{373}$ has all integer coefficients.  The prime factors 11 and 7331 are both in the residue class of 1, and cannot be used. However, the residue class of 13 is 3, which indeed generates the whole group $\mathbb{Z}_5^{*}$.  We find that $8+8192+2048+256+32 = 10536 \equiv 6 \pmod {13}$.  Furthermore $32 \equiv 6 \pmod {13}$ and $8 \equiv 8 \pmod {13}$.  Thus by the theorem we conclude that $\det(\mathcal{M}_{373})$ is nonsingular modulo 13.  Furthermore, $\det(\mathcal{M}_{373}) = 36861972294903300096 = \det(\mathcal{R}_{373}) \equiv 6 {\pmod {13}}$.  Finally, $\det(\mathcal{M}_{373})\not\equiv 0 {\pmod {2^{20}-3^5}}$.
\end{example}

\section{A Note Regarding Chapters Two Through Four}
The remainder of this thesis is organized into three chapters.  Chapter two presents the background material and essential notation for the remainder of the thesis.  Chapter three presents the matrix notation and various properties used to express integral loop vectors.  Chapter four is a muesli of ideas related to chapter three.

%Chapter four discusses the conditions which such matrices are singular.

\chapter{Fundamentals of Bounded Trajectories}

\section{Preliminary Definitions}

\begin{definition}\label{C}\normalfont
Fix $\gamma \in \mathbb{Z}_{>0}^{odd}$.  Let $C_\gamma: \mathbb{Z}_{>0} \rightarrow \mathbb{Z}_{>0}$ be defined as,

  \begin{equation}
    C_\gamma(n) =
    \begin{cases}
      3n+\gamma & \text{if } n \equiv 1 \pmod 2, \\
     \frac{n}{2}        & \text{if } n \equiv 0 \pmod 2.
    \end{cases}
  \end{equation}
  
Let $C_\gamma^0(n) = n$.  Furthermore, for $k\geq 0$, denote function composition as $C_\gamma^{k+1}(n) = C_\gamma^k \circ C_\gamma(n)$.  Finally, if $\gamma = 1$ then we write $C$ instead of $C_1$. 

\end{definition}

\begin{definition}\label{D}\normalfont

Fix $\gamma \in \mathbb{Z}_{>0}^{odd}$.  Let $\chi_\gamma: \mathbb{Z}_{>0} \rightarrow \{0,1\}$ be defined as,

  \begin{equation}
    \chi_\gamma(n) =
    \begin{cases}
      1 & \text{if } C_\gamma(n) \equiv 1 \pmod 2, \\
      0        & \text{if } C_\gamma(n) \equiv 0 \pmod 2.
    \end{cases}
  \end{equation}
  
Let $\chi_\gamma^0(n) \equiv n \pmod 2$.  Furthermore, for $k\geq 0$, denote function composition as $\chi_\gamma^{k+1}(n) = \chi_\gamma \circ C_\gamma^{k+1}(n)$.  Finally, if $\gamma = 1$ then we write $\chi$ instead of $\chi_1$. 
  
\end{definition}

\begin{definition}\label{A} \normalfont 

Fix $\gamma \in \mathbb{Z}_{>0}^{odd}$ and $n \in \mathbb{Z}_{>0}$.  Let $\{ C_\gamma^0(n),C_\gamma^1(n),...\}$ denote the \textit{trajectory} of $n$.
\end{definition}

\begin{definition}\label{B} \normalfont 
Fix $\gamma \in \mathbb{Z}_{>0}^{odd}$ and $n \in \mathbb{Z}_{>0}$.  Let $\{ \chi_\gamma^0(n),\chi_\gamma^1(n),...\}$ denote the \textit{characteristic trajectory} of $n$.
\end{definition}

\begin{example}\normalfont

The first four terms of the trajectory and characteristic trajectory induced by $n=5$ and $\gamma = 1$ appears below respectively,

\begin{equation*} \{5,16,8,4\} \text{ and } \{1,0,0,0\} . \end{equation*}

\end{example}

\noindent Let $\{0,1\}^{**}$ denote the set of binary sequences of infinite length.

\begin{definition}\normalfont

Fix $\gamma \in \mathbb{Z}_{>0}^{odd}$ and $n \in \mathbb{Z}_{>0}$.  We say $B \in \{0,1\}^{**}$ is \textit{associated} to the characteristic trajectory of $n$ if $B = (\chi_\gamma^i(n))_{i\geq0}$.
\end{definition}

\begin{proposition}\label{tt_Cg} \textit{Fix $\gamma \in \mathbb{Z}_{>0}^{odd}$, $n \in \mathbb{Z}_{>0}$, and $N \in \mathbb{Z}_{>0}$.  Let $B\in \{0,1\}^{**}$ such that $B = (\chi_\gamma^i(n))_{i\geq 0}$.  Let $\rho$ and $\nu$ be the count of the number of zeros and ones, respectively in $B$ for all $i$ between zero and $N-1$ inclusive.
\begin{enumerate}
\item If n is odd, let $m_1,...,m_{\nu - 1}$ be the number of zeros between the $(j-1)$th and $j$th one.  Let $m_0 = 0$ and $m_\nu$ be the remaining zeros after the last one.
\item If n is even, let $m_0$ be the number of leading zeros.  Observe that $n = 2^{m_0} \tilde{n}$ for $\tilde{n}$ being odd.  So we obtain $m_1,...,m_{\nu-1}$ by (1) and $\tilde{n}$.
\end{enumerate}
Then $n$ satisfies the following equality}

\begin{equation}\label{tilde_n}
C_\gamma^N(n) = \frac{3^\nu n + \gamma2^{m_0}(3^{\nu-1} + \sum_{r=1}^{\nu-1} 3^{\nu -1 -r} 2^{\sum_{j=1}^r m_j})}{2^\rho}.
\end{equation}
\end{proposition}

\begin{proof}
Let us assume first that $n$ is odd, which implies that $m_0 = 0$.  Therefore Equation (\ref{tilde_n}) reduces to the form,

\begin{equation*}
C_\gamma^{N-1}(n) = \frac{3^\nu n + \gamma(3^{\nu-1} + \sum_{r=1}^{\nu-1} 3^{\nu -1 -r} 2^{\sum_{j=1}^r m_j})}{2^\rho}.
\end{equation*}

By construction of $\rho$ and $\nu$, we make the substitution $N-1 = \rho + \nu$.  We know that $C_\gamma^{N-1}(n)$ represents $N-1$ compositions of the $C_\gamma$ mapping with starting integer $n$.  Therefore, we can advance the trajectory of $n$ by one, and subtract one step from $N-1$ in the following manner,

\begin{equation*}
C_\gamma^{N-1}(n) = C_\gamma^{\rho + \nu}(n) = C_\gamma^{\rho+\nu-1} \circ C_\gamma(n) =  C_\gamma^{\rho+\nu-1}(3n+\gamma).
\end{equation*}

\noindent Both the terms $3n$ and $\gamma$ are odd.  Therefore their sum is even.  We can divide by two, $m_1$ number of times yielding the following expression:

\begin{equation*}
C_\gamma^{N-1}(n) = C_\gamma^{\rho-m_1+\nu-1} \circ C^{m_1}_\gamma(3n+\gamma) =  C_\gamma^{\rho-m_1+\nu-1}(\frac{3n+\gamma}{2^{m_1}}).
\end{equation*}

\noindent By hypothesis, this expression is again odd.  Therefore we can apply the $C_\gamma$ map for one more step.

\begin{equation*}
C_\gamma^{N-1}(n) = C_\gamma^{\rho-m_1+\nu-2} \circ C_\gamma(\frac{3n+\gamma}{2^{m_1}}) =  C_\gamma^{\rho-m_1+\nu-2}(\frac{3^2n+3\gamma+2^{m_1}\gamma}{2^{m_1}}).
\end{equation*}

\noindent In turn, this expression is again even.  Therefore we can divide by $m_2$ powers of two.  Following this, we can reduce the expression again,

\begin{equation*}
C_\gamma^{N-1}(n) = C_\gamma^{\rho-m_1-m_2+\nu-2} \circ C^{m_2}_\gamma(\frac{3n+\gamma}{2^{m_1}}) =  C_\gamma^{\rho-m_1-m_2+\nu-2}(\frac{3^2n+3\gamma+2^{m_1}\gamma}{2^{m_1+m_2}}).
\end{equation*}

\noindent We proceed to follow this process inductively until the length of the entire finite sub-sequence of $B$ is exhausted.

\begin{equation*}
\begin{aligned}
C_\gamma^{N-1}(n) = C_\gamma^{\rho-m_1-...-m_{\nu}+\nu-1-...-1} &\circ C^{m_{\nu-1}}_\gamma(\frac{3^{\nu}n+3^{\nu-1}\gamma+3^{\nu-2}2^{m_1}\gamma +3^{\nu-3}2^{m_1+m_2}\gamma +...}{2^{m_1+m_2+...}}),   \\
&=C_\gamma^{0}(\frac{3^{\nu}n+3^{\nu-1}\gamma+3^{\nu-2}2^{m_1}\gamma +3^{\nu-3}2^{m_1+m_2}\gamma +...}{2^{\rho}}), \\
&= \frac{3^\nu n + \gamma2^{m_0}(3^{\nu-1} + \sum_{r=1}^{\nu-1} 3^{\nu -1 -r} 2^{\sum_{j=1}^r m_j})}{2^\rho}.
\end{aligned}
\end{equation*}

On the other hand, suppose that $n$ is even.  Then $m_0 > 0$, and by hypothesis $n = 2^{m_0}\tilde{n}$ for an odd integer $\tilde{n}$.  Starting from Equation (\ref{tilde_n}) we find by taking $m_0$ steps,

\begin{equation*}
C_\gamma^{N-1}(n) = C_\gamma^{\rho + \nu}(n) = C_\gamma^{\rho -m_0+\nu} \circ C^{m_0}_\gamma(n) =  C_\gamma^{\rho-m_0+\nu}(\frac{n}{2^{m_0}}).
\end{equation*}

\noindent Recall by hypothesis $n = 2^{m_0}\tilde{n}$ implies $\frac{n}{2^{m_0}} = \tilde{n}$.  Substituting this quantity in yields:

\begin{equation*}
C_\gamma^{N-1}(n) =  C_\gamma^{\rho-m_0+\nu}(\frac{n}{2^{m_0}}) = C_\gamma^{\rho-m_0+\nu}(\tilde{n}).
\end{equation*}

\noindent We then consider the previous argument starting with $\tilde{n}$, $\tilde{\rho} = \rho - m_0$, and $\nu$ unchanged.

\begin{equation*}
\begin{aligned}
C_\gamma^{N-1-m_0}(\tilde{n}) &= \frac{3^\nu \tilde{n} + \gamma(3^{\nu-1} + \sum_{r=1}^{\nu-1} 3^{\nu -1 -r} 2^{\sum_{j=1}^r m_j})}{2^{\tilde{\rho}}}, \\
&= \frac{3^\nu \frac{n}{2^{m_0}} + \gamma(3^{\nu-1} + \sum_{r=1}^{\nu-1} 3^{\nu -1 -r} 2^{\sum_{j=1}^r m_j})}{2^{\rho - m_0}}, \\
&=\Big(\frac{2^{m_0}}{2^{m_0}} \Big) \frac{3^\nu \frac{n}{2^{m_0}} + \gamma(3^{\nu-1} + \sum_{r=1}^{\nu-1} 3^{\nu -1 -r} 2^{\sum_{j=1}^r m_j})}{2^{\rho - m_0}}, \\
&=\frac{3^\nu n + \gamma2^{m_0}(3^{\nu-1} + \sum_{r=1}^{\nu-1} 3^{\nu -1 -r} 2^{\sum_{j=1}^r m_j})}{2^{\rho}}. \\
\end{aligned}
\end{equation*}

\noindent Therefore we find that,

\begin{equation*}
C^{N-1}_\gamma(n) = C^{N-1-m_0}_\gamma (\tilde{n}) =\frac{3^\nu n + \gamma2^{m_0}(3^{\nu-1} + \sum_{r=1}^{\nu-1} 3^{\nu -1 -r} 2^{\sum_{j=1}^r m_j})}{2^{\rho}}.
\end{equation*}

\end{proof}

\begin{example}\normalfont
Consider the first nine terms of the characteristic trajectory induced by $n=19$ an $\gamma = 1$ below,

\begin{equation*}\{1,0,1,0,0,0,1,0,1\}.\end{equation*}

We verify Equation (\ref{tilde_n}) for $C^9(19)$.

\begin{equation*} 
\begin{aligned}
C^9(19) 	&= \frac{3^4 \times 19 + 2^0 (3^3 + 3^2 \times 2^1 + 3^1 \times 2^{1+3} + 3^0 \times 2^{1+3+1})}{2^5}, \\
		&= \frac{1539 + 27 + 18 + 48 + 32}{32},\\
		&= \frac{1664}{32} = 52.
\end{aligned}
\end{equation*}

\end{example}

\begin{definition}\label{E}\normalfont
Fix $\gamma \in \mathbb{Z}_{>0}^{odd}$ and $n \in \mathbb{Z}_{>0}$.  If there exists a $N \in \mathbb{Z}_{>0}$ so that $C_\gamma^N(n) = n$ then we say $n$ starts an \textit{integral loop}.
\end{definition}

\begin{definition}\label{F}\normalfont
Fix $\gamma \in \mathbb{Z}_{>0}^{odd}$ and $n \in \mathbb{Z}_{>0}$.  Suppose $n$ starts an integral loop.  Let $N = \min (\{ i \in \mathbb{Z}_{>0} | C_\gamma^i (n) = n \} )$ be the least positive integer so that $C_\gamma^N(n) = n$.  Then we say $n$ starts an integral loop of \textit{length} $N$.
\end{definition}

\begin{example}\normalfont

The integral loop induced by $n=23$ and $\gamma = 37$ appears below,

\begin{equation*} \{23,106,53,196,98,49,184,92,46\}. \end{equation*}

\end{example}

Suppose $n$ starts an integral loop of length $N$.  Then the binary sequence $B$ associated to $n$ has a repeating period of length $N$.  Let $\rho$ be the number of zeros in a single period of $B$.  Let $\nu$ be the number of ones in a single period of $B$.  Let $m_j$ denote the number of zeros between the $(j-1)$th and the $j$th one.  Let $m_\nu$ representing the remaining zeros after the last one.  Note that $\sum_j m_j = \rho$.

\begin{proposition}\label{integral_loop_formula} \textit{
Fix $\gamma \in \mathbb{Z}_{>0}^{odd}$ and $n \in \mathbb{Z}_{>0}$.  Let $B\in \{0,1\}^{**}$ such that $B = (\chi_\gamma^i(n))_{i\geq 0}$.  Suppose that $N$ is the smallest positive
integer such that $C_\gamma^{N+1}(n) = n$.  Let $\rho$ and $\nu$ be the count of the number of zeros and ones, respectively, in $B$ for all $i$ between $0$ and $N$ inclusive.
\begin{enumerate}
\item If n is odd, let $m_1,...,m_{v-1}$ be the number of zeros between the $(j-1)$th and $j$th one.  Let $m_0 = 0$ and $m_\nu$ be the remaining zeros after the last one.
\item If n is even, let $m_0$ be the number of leading zeros.  Observe that $n = 2^{m_0}\tilde{n}$ for $\tilde{n}$ being odd.  So we obtain $m_1,...,m_{\nu-1}$ by (1) and $\tilde{n}$.
\end{enumerate}
Then $n$ satisfies the following equality}

\begin{equation} \label{integral_loop_formula_equ}
n = \frac{\gamma2^{m_0}(3^{\nu-1}+\sum_{r=1}^{\nu-1} 3^{\nu-1-r}2^{\sum_{j=1}^r m_j})}{2^\rho - 3^\nu}.
\end{equation}
\end{proposition}

\begin{proof}
Equation (\ref{integral_loop_formula_equ}) can be derived from Equation (\ref{tilde_n}) by setting $C_\gamma^N(n) = n$.
\end{proof}

\noindent We will refer to Equation (\ref{integral_loop_formula_equ}) as the \textit{Integral Loop Formula}.

\begin{example}\normalfont

Consider again the integral loop induced by $n=23$ and $\gamma = 37$.  With respect to Equation (\ref{integral_loop_formula_equ}) we find that,

\begin{equation*}
\begin{aligned}
n &= \frac{37(3^{3-1}+3^{3-2}\times 2^1 + 3^{3-3}\times 2^{1+2})}{2^6 - 3^3}, \\
& = \frac{37(9+3\times 2 + 1\times 8)}{37} = 23.
\end{aligned}
\end{equation*}
\end{example}

\begin{definition}
\normalfont
Fix $\gamma \in \mathbb{Z}_{>0}^{odd}$.  The trajectory of an integer $n \in \mathbb{Z}_{>0}$ is said to be \textit{bounded} if there exists some $\Omega \in \mathbb{Z}_{>0}$ such that $\sup\limits_{i>0}( \{C^i_\gamma(n)\}) <\Omega$.
\end{definition}
 
\noindent With these precursory definitions established, we will proceed to study the properties of bounded trajectories.

 %---------------------------------------------------------------
 %///////////////////////////////////////////////////////////////
 %---------------------------------------------------------------

\section{ The Integral Loop Vector}
 
\begin{definition}\label{index_seq}\normalfont
Fix $\gamma \in \mathbb{Z}_{>0}^{odd}$, $n \in \mathbb{Z}_{>0}^{odd}$, and $N \in \mathbb{Z}_{>0}$.  Suppose further that $n$ starts an integral loop of length $N$.  Let $B \in \{0,1\}^{**}$ such that $B = (\chi_\gamma^i(n))_{i \geq 0}$.  Let $\rho$ and $\nu$ be the count of the number of zeros and ones in a single period of $B$.  Let $i_j$ be the subindexing of $i$ so that for all $0 \leq j \leq \nu-1$ we have $\chi_\gamma^{i_j}(n) \equiv 1 \pmod 2$.  Note also $i_0 = 0$.  We call $(i_j)_{j=0}^{\nu-1}$ the \textit{indexing sequence} associated to $n$.
 \end{definition}
 
 \begin{definition}\label{2B} \normalfont
Fix $\gamma \in \mathbb{Z}_{>0}^{odd}$, $n \in \mathbb{Z}_{>0}^{odd}$, and $N \in \mathbb{Z}_{>0}$.  Suppose further that $n$ starts an integral loop of length $N$.  Let $B \in \{0,1\}^{**}$ such that $B = (\chi_\gamma^i(n))_{i\geq 0}$.  Let $\rho$ and $\nu$ be the count of the number of zeros and ones in a single period of $B$.  Let $(i_j)_{j=0}^{\nu-1}$ be the indexing sequence associated to $n$.  Let $e_j$ denote the standard basis of $\mathbb{R}^\nu$.  We call $[n]$ the \textit{integral loop vector} expressed in Equation (\ref{int_loop_vec}):
 
\begin{equation}\label{int_loop_vec}
[n] = \begin{bmatrix}C_\gamma^{i_0}(n) \\ C_\gamma^{i_1}(n) \\ \vdots \\ C_\gamma^{i_{\nu-1}}(n) \end{bmatrix} = \sum_{j=1}^\nu C_\gamma^{i_{j-1}}(n) e_{j} .
\end{equation}
\end{definition}

\begin{example}\normalfont
Consider the integral loop induced by $n=421$ and $\gamma = 13$.  With respect to Equation (\ref{int_loop_vec}), we write the integral loop vector [421] as follows:

\begin{equation*}
[421] = \begin{bmatrix}421 \\ 319\\ 485 \\ 367 \\ 557 \end{bmatrix}.
\end{equation*}

\end{example}

\begin{proposition}\label{telescope} \textit{Fix $\gamma \in \mathbb{Z}_{>0}^{odd}$, $n \in \mathbb{Z}_{>0}^{odd}$, and $N \in \mathbb{Z}_{>0}$.  Suppose further that $n$ starts an integral loop of length $N$.  Let $B \in \{0,1\}^{**}$ such that $B = (\chi_\gamma^i(n))_{i\geq 0}$.  Let $\rho$ and $\nu$ be the count of the number of zeros and ones in a single period of $B$.  Let $(i_j)_{j=0}^{\nu-1}$ be the index sequence associated to $n$.  Then the quantities of $m_j$ from Proposition \ref{integral_loop_formula} and $i_j$ from Definition \ref{index_seq} can be related by the following equations:}

  \begin{equation}\label{a_s}
    \begin{cases}
      m_j &= i_j-i_{j-1}-1, \text{ if } 1 \leq j \leq \nu-1, \\
      m_\nu        &= N-1-i_{\nu-1}.
    \end{cases}
  \end{equation}
\end{proposition}

\begin{proof}
For $1\leq j \leq \nu-1$ the statement follows directly from construction.  That leaves only the last term to check.  We observe that by taking the sum of all terms, we have a telescoping sum: 

\begin{equation*} 
\begin{aligned}
m_1 +...+m_{\nu-1}+m_{\nu} &= i_1 - i_0 - 1 + i_2 - i_1 - 1 + ... + i_{\nu-1} - i_{\nu-2}-1 + N - i_{\nu-1}-1, \\
&= N - i_0 - \nu = \rho+\nu - 0 -\nu = \rho.
\end{aligned}
\end{equation*}

\noindent Thus we see the last term necessarily counts the remaining number of even numbers before the period repeats.
\end{proof}
 
Observe that for $j\geq 1$, we have $[n]_j = C_\gamma^{i_{j-1}}(n)$.  This highlights the relationship between the subindex $j$ and the indexing sequence $(i_j)$.  This naturally extends to an infinite dimensional array.  However, in the case of integral loops, we consider only the first period, which has dimension $\nu$.
 
 \begin{definition}\label{similar_v}\normalfont
Fix $\gamma, n, \tilde{n} \in \mathbb{Z}_{>0}^{odd}$, and $N\in\mathbb{Z}_{>0}$.  Suppose both $n$ and $\tilde{n}$ start integral loops of length $N$.  Let $B, \tilde{B} \in \{0,1\}^{**}$ be the binary sequences associated to the characteristic trajectories of $n$ and $\tilde{n}$, respectfully.  Let $\rho, \tilde{\rho}, \nu \text{ and } \tilde{\nu}$ be the count for the number of zeros and ones in a single period of $B$ and $\tilde{B}$, respectfully.  Suppose further that $\nu = \tilde{\nu}$.  Let $[n]$ and $[\tilde{n}]$ be the associated integral loop vectors.  We state the following:
\begin{enumerate}
\item If for all $j$, $[n]_j = [\tilde{n}]_j$ we say that $[n]$ and $[\tilde{n}]$ are \textit{equal} and write $[n] = [\tilde{n}]$.
\item If there exists a $0<k\leq\nu-1$, so that for all $j$, $[n]_{j+k} = [\tilde{n}]_j$, we say that $[n]$ and $[\tilde{n}]$ are \textit{similar} and write $[n] \approx [\tilde{n}]$.
\item If integral loop vectors $[n]$ and $[\tilde{n}]$ are neither equal or similar, then we say they are \textit{distinct}.
\end{enumerate}
\end{definition}

\begin{example}\normalfont

Consider the integral loop induced by $n = 133$ and $\gamma = 943$.  The integral loop vector associated to $133$ appears as:

\begin{equation*}
[133] = \begin{bmatrix} 133 \\ 671 \\ 739 \\ 395 \end{bmatrix}.
\end{equation*}

\noindent Furthermore, we find that $[133] \approx [671] \approx [739] \approx [395]$.  All other integral loop vectors are distinct.

\end{example}

%We would like to have some way to transform two distinct integral loop vectors $[n]$ and $[\tilde{n}]$ into each other.  There are two ways that this can be done.  We will find that each has its own merits and flaws.  Moreover, in combination they yield new information regarding integral loops.

\chapter{Matrix Constructions Associated to Integral Loops}

 \section{Construction of Matrix $\mathcal{D}_n$}
 
\noindent Let $e_{j,k}$ ($1\leq j,k \leq \nu$) denote the standard basis of $\mathbb{R}^{\nu\times\nu}$.
 
 \begin{definition}\normalfont
 
Fix $\nu \in \mathbb{Z}_{>0}$.  Let $\mathcal{P}$ be the $\nu \times \nu$ permutation matrix characterized by the following formula:
 
 \begin{equation}\label{cal_p}
 \mathcal{P} = \sum_{j=1}^\nu e_{j,|j|_\nu+1},
 \end{equation}
 
\noindent where $|j|_\nu$ denotes $j$ modulo $\nu$.
 
 \end{definition}
 
\begin{example}\normalfont
Let $\nu = 5$.
 
\begin{equation*}
\mathcal{P} = \begin{bmatrix} 0 & 1 & 0 & 0& 0 \\ 0 &0 & 1 & 0& 0 \\ 0 & 0 & 0 & 1& 0 \\ 0 & 0 & 0 & 0& 1 \\ 1 & 0 & 0 & 0& 0 \end{bmatrix}.
\end{equation*}

\end{example}
 
\begin{definition}\label{3A}
\normalfont
Fix $\gamma \in \mathbb{Z}_{>0}^{odd}$, $n \in \mathbb{Z}_{>0}^{odd}$, and $N \in \mathbb{Z}_{>0}$.  Suppose further that $n$ starts an integral loop of length $N$.  Let $B \in \{0,1\}^{**}$ such that $B = (\chi_\gamma^i(n))_{i\geq 0}$.  Let $\rho$ and $\nu$ be the count of the number of zeros and ones in a single period of $B$.  Let the quantities $m_j$ be given by  by Proposition \ref{integral_loop_formula}.  We define the $\nu\times\nu$ matrix $\mathcal{L}_n$ associated to $n$ as follows:

\begin{equation}
\mathcal{L}_n = \sum_{j=1}^\nu 2^{m_j} e_{j,j}.
\end{equation}
\end{definition}
 
\begin{definition}\label{def_d_n}\normalfont
 
Fix $\gamma \in \mathbb{Z}_{>0}^{odd}$, $n \in \mathbb{Z}_{>0}^{odd}$, and $N \in \mathbb{Z}_{>0}$.  Suppose further that $n$ starts an integral loop of length $N$.  Let $B \in \{0,1\}^{**}$ such that $B = (\chi_\gamma^i(n))_{i\geq 0}$.  Let $\rho$ and $\nu$ be the count of the number of zeros and ones in a single period of $B$.  Let the quantities $m_j$ be given by Proposition \ref{integral_loop_formula}.  We define the $\nu\times\nu$ matrix $\mathcal{D}_n$ associated to $n$ to be the result of the matrix product of $\mathcal{L}_n$ and $\mathcal{P}$:
 
\begin{equation}\label{d_n}
\mathcal{D}_n = \mathcal{L}_n \times \mathcal{P} = \sum_{j=1}^\nu 2^{m_j}e_{j,|j|_\nu+1}.
\end{equation}

\noindent Where $|j|_\nu$ denotes $j$ modulo $\nu$.
\end{definition}

\begin{remark} \normalfont
The matrix constructed above in Definition \ref{def_d_n} is a close cousin to the \textit{circulant matrix} \cite{davisphilip}.  The distinction being that the rows of a circulant matrix, apart from a right-shifting, are identical.  In our case, we encoded the finite binary sequence associated to the integral loop into the off-diagonal elements.
\end{remark}

\begin{example}\normalfont

Consider the integral loop induced by $n = 19$ and $\gamma = 37$.  Then the matrix $\mathcal{D}_{19}$ appears in the following form:

\begin{equation*}
\mathcal{D}_{19} = \begin{bmatrix} 0 & 2 & 0 \\ 0 & 0 & 2 \\ 16 & 0 & 0\end{bmatrix}.
\end{equation*}

\end{example}

\begin{proposition}\label{exp_rule}\textit{Fix $\gamma \in \mathbb{Z}_{>0}^{odd}$, $n \in \mathbb{Z}_{>0}^{odd}$, and $N\in\mathbb{Z}_{>0}$.  Suppose further that $n$ starts an integral loop of length $N$.  Let $B \in \{0,1\}^{**}$ such that $B = (\chi_\gamma^i(n))_{i\geq 0}$.  Let $\rho$ and $\nu$ be the count of the number of zeros and ones in a single period of $B$.  Let the quantities $m_j$ be given by Proposition \ref{integral_loop_formula} and the matrix $\mathcal{D}_n$ given by Equation (\ref{d_n}).  Finally, let $k\in \mathbb{Z}_{> 0}$ be given.  Then the matrix $\mathcal{D}_n$ has the following exponent rule:}
 
 \begin{equation*}
\mathcal{D}_n^k = \sum_{j=1}^\nu \left(2^{\sum_{r=0}^{k-1} m_{|j+r-1|_\nu+1}}\right)e_{j,|j+k-1|_\nu +1}.
\end{equation*}

\noindent \textit{In particular, this implies $\mathcal{D}_n^\nu$ = $2^\rho \mathbb{I}_\nu$, where $\mathbb{I}_\nu$ is the $\nu \times \nu$ identity matrix.}
\end{proposition}

\begin{proof}
Observe first that if $k= 1$ we arrive back to Equation (\ref{d_n}).  Suppose then the exponent rule holds up to some quantity $k^*$.  Then for $k^*+1$ we have:

\begin{equation*}
\begin{aligned}
\mathcal{D}_n^{k^*+1} 
&= \mathcal{D}_n^{k^*}\times \mathcal{D}_n, \\
& = \mathcal{D}_n^{k^*}\times \mathcal{L}_n \times \mathcal{P}.
\end{aligned}
\end{equation*}

\noindent Recall that $\mathcal{L}_n$ is a diagonal matrix, so the matrix product simply adds the respective $m_j$ to the end of summand in the exponents for each entry of $\mathcal{D}_n^{k^*}$.  Similarly, the matrix $\mathcal{P}$ is a permutation matrix that right-shifts all entries by one element.  In this way we arrive at the desired form.  Finally, we see that when the matrix $D_n$ is taken to the $\nu$ power, that every diagonal element is exactly $\sum_j m_j = \rho$.
\end{proof}

\begin{example}\normalfont

Consider the integral loop induced by $n = 23$ and $\gamma = 37$.  The first three powers of $\mathcal{D}_{23}$ appear as follows:

\begin{equation*}
\begin{matrix}
\mathcal{D}_{23}^1 = \begin{bmatrix} 0 & 2 & 0 \\ 0 & 0 & 4 \\ 8 & 0 & 0 \end{bmatrix}, &
\mathcal{D}_{23}^2 = \begin{bmatrix} 0 & 0 & 8 \\ 32 & 0 & 0 \\ 0 & 16 & 0 \end{bmatrix}, &
\mathcal{D}_{23}^3 = \begin{bmatrix} 64 & 0 & 0 \\ 0 & 64 & 0 \\ 0 & 0 & 64 \end{bmatrix}.
\end{matrix}
\end{equation*}

\end{example}

 \section{Construction of Matrix $\mathcal{M}_n$}

 %---------------------------------------------------------------
 %---------------------------------------------------------------

\subsection{Preliminary Notation}
 
\begin{definition}\normalfont

Fix $\nu \in \mathbb{Z}_{>0}$.  Let $\overrightarrow{\mathfrak{3}}^\nu$ be the $\nu$-column vector whose entries are the ascending powers of three, appearing as follows:
 
\begin{equation} 
\overrightarrow{\mathfrak{3}}^\nu = \sum_{j=1}^\nu 3^{\nu-j}e_j .
\end{equation}
\end{definition}

\begin{example}

Let $\nu = 5$.

\begin{equation*}
\overrightarrow{\mathfrak{3}}^5 = \begin{bmatrix} 3^4 \\ 3^3 \\ 3^2 \\3^1 \\3^0 \end{bmatrix}.
\end{equation*}

\end{example}
 
\begin{definition}\normalfont
Fix $\gamma \in \mathbb{Z}_{>0}^{odd}$, $n \in \mathbb{Z}_{>0}^{odd}$, and $N\in\mathbb{Z}_{>0}$.  Suppose further that $n$ starts an integral loop of length $N$.  Let $B \in \{0,1\}^{**}$ such that $B = (\chi_\gamma^i(n))_{i\geq 0}$.  Let $\rho$ and $\nu$ be the count of the number of zeros and ones in a single period of $B$.  Let the matrix $\mathcal{D}_n$ be given by Equation (\ref{d_n}).  We define the following $\nu \times \nu$ matrix $\mathcal{M}_n$ below:

\begin{equation}\label{m_n}
\mathcal{M}_n = \sum_{k=1}^\nu \mathcal{D}_n^{k-1}\sum_{j=1}^\nu e_{j,k}.
\end{equation}

\noindent Where $\mathcal{D}^k_n$ denotes the $k$th power of $\mathcal{D}_n$ and $\sum_{j=1}^\nu e_{j,k}$ denotes the matrix whose $k$th column is all ones and is zero everywhere else.
\end{definition}

\begin{example}\normalfont

Consider the integral loop induced by $n = 65$ and $\gamma = 943$.  Then the matrix $\mathcal{M}_{65}$ appears as follows:

\begin{equation*}
\mathcal{M}_{65} = \begin{bmatrix} 1 & 2 & 4 & 8 \\ 1 & 2& 4 & 512 \\ 1 & 2 & 256 &512 \\ 1 & 128 & 256 & 512
\end{bmatrix}.
\end{equation*}

\end{example}

\begin{proposition}\label{3v} \textit{Fix $\gamma \in \mathbb{Z}_{>0}^{odd}$, $n \in \mathbb{Z}_{>0}^{odd}$, and $N\in\mathbb{Z}_{>0}$.  Suppose further that $n$ starts an integral loop of length $N$.  Let $B \in \{0,1\}^{**}$ such that $B = (\chi_\gamma^i(n))_{i\geq 0}$.  Let $\rho$ and $\nu$ be the count of the number of zeros and ones in a single period of $B$.  Let $[n]$ be the associated integral loop vector of $n$.  Let the matrix $\mathcal{M}_n$ be given by Equation (\ref{m_n}).  Then $[n]$ and $\mathcal{M}_n$ satisfy the following equality,}

\begin{equation}\label{central_2}
[n] = \frac{\gamma}{2^\rho - 3^\nu}\mathcal{M}_n \overrightarrow{\mathfrak{3}}^\nu.
\end{equation}

\noindent \textit{Where $\mathcal{M}_n \overrightarrow{\mathfrak{3}}^\nu$ denotes the product of the $\nu \times \nu$ matrix $\mathcal{M}_n$ and the $\nu \times 1$ column vector $\overrightarrow{\mathfrak{3}}^\nu$.}

\end{proposition}

\begin{proof}
This follows directly from Proposition \ref{integral_loop_formula}.
\end{proof}

\begin{example}\normalfont

Consider the integral loop induced by $n = 23$ and $\gamma = 37$.  We verify Proposition \ref{3v} as follows:

\begin{equation*}
[23] = \begin{bmatrix} 23 \\ 53 \\ 49 \end{bmatrix} = \frac{37}{37} \begin{bmatrix} 1 & 2 & 8 \\ 1 & 4 & 32 \\ 1 & 8 & 16 \end{bmatrix} \begin{bmatrix} 9 \\ 3 \\ 1 \end{bmatrix} = \frac{37}{2^6-3^3} \begin{bmatrix} 1 & 2 & 8 \\ 1 & 4 & 32 \\ 1 & 8 & 16 \end{bmatrix} \begin{bmatrix} 9 \\ 3 \\ 1 \end{bmatrix}.
\end{equation*}

\end{example}

%\begin{remark}\normalfont 
%We will find that this is because the matrices $\mathcal{M}_n$, are not always invertible.
%\end{remark}

%-------------------------------------------------------------
%-------------------------------------------------------------

\subsection{Elementary Properties of Integral Loops}

\begin{lemma}\label{alpha}\textit{
    Fix $\gamma,n \in \mathbb{Z}_{>0}^{odd}$ and $N\in \mathbb{Z}_{>0}$.  Let $n$ start an integral loop of length $N$.  Let $B\in\{0,1\}^{**}$ such that $B = (\chi_\gamma^i(n))_{i\geq 0}$.  Let $\rho$ and $\nu$ be the count of the number of zeros and ones in a single period of $B$.  Let $(i_j)_{j=0}^{\nu-1}$ be the indexing sequence associated to $n$.  Let $p>2$ be a prime.  If $p$ divides two of $C_\gamma^{i_j}(n)$, $C_\gamma^{i_{j+1}}(n)$, or $\gamma$ then $p$ divides the third quantity as well.}
\end{lemma}

\begin{proof}
    Suppose first that $p$ divides $C_\gamma^{i_j}(n)$ and $\gamma$.

    \begin{equation*}
        \begin{aligned}
            C \circ C_\gamma^{i_j}(n) = 3C_\gamma^{i_j}(n)+\gamma &= 2^{m_{j+1}}C_\gamma^{i_{j+1}}(n), \\
            p \left( \frac{3C_\gamma^{i_j}(n)+\gamma}{p} \right) & = 2^{m_{j+1}}C_\gamma^{i_{j+1}}(n).
        \end{aligned}
    \end{equation*}

    \noindent It is an elementary result in number theory that $C_\gamma^{i_{j+1}}(n)$ must have $p$ as a prime factor.  Now suppose $p$ divides $C_\gamma^{i_{j+1}}(n)$ and $\gamma$.

    \begin{equation*}
        \begin{aligned}
            2^{m_{j+1}}C_\gamma^{i_{j+1}}(n) &= 3C_\gamma^{i_j}(n) + \gamma, \\
            p\left(\frac{2^{m_{j+1}}C_\gamma^{i_{j+1}}(n)-\gamma}{p}\right) &= 3C_\gamma^{i_j}(n).
        \end{aligned}
    \end{equation*}

    \noindent $C_\gamma^{i_j}(n)$ must have $p$ as a prime factor.  Finally suppose that $p$ divides $C_\gamma^{i_j}(n)$ and $C_\gamma^{i_{j+1}}(n)$.

    \begin{equation*}
        \begin{aligned}
            C \circ C_\gamma^{i_j}(n) &= 3C_\gamma^{i_j}(n)+\gamma = 2^{m_{j+1}}C_\gamma^{i_{j+1}}(n), \\
            \gamma &= p \left(\frac{2^{m_{j+1}}C_\gamma^{i_{j+1}}(n) - 3C_\gamma^{i_j}(n)}{p}\right) .
        \end{aligned}
    \end{equation*}

    \noindent Thus $\gamma$ must have $p$ as a prime factor.  As we have tested all three possibilities, we conclude that if $p$ divides two of $C_\gamma^{i_j}$, $C_\gamma^{i_{j+1}}(n)$, or $\gamma$ then $p$ divides the third quantity as well.
\end{proof}

\begin{lemma}\textit{
    Fix $\gamma,n \in \mathbb{Z}_{>0}^{odd}$ and $N\in \mathbb{Z}_{>0}$.  Let $n$ start an integral loop of length $N$.  Let $B\in\{0,1\}^{**}$ such that $B = (\chi_\gamma^i(n))_{i\geq 0}$.  Let $\rho$ and $\nu$ be the count of the number of zeros and ones in a single period of $B$.  Let $(i_j)_{j=0}^{\nu-1}$ be the indexing sequence associated to $n$.  Let $p>2$ be a prime.  If $p$ divides only one of $C_\gamma^{i_j}(n)$, or $\gamma$ then $p$ cannot divide $C_\gamma^{i_{j+1}}(n)$.}
\end{lemma}

\begin{proof}
    This follows directly from Lemma \ref{alpha}.
\end{proof}

\begin{example}
    \noindent Consider the integral loop induced by $n = 629$ and $\gamma = 5^4 = 625$.  This integral loop is of length $N=144$ with $\nu=43$ odd positive integers.  None of which are a multiple of five.
\end{example}

\begin{lemma}\label{beta}\textit{
    Fix $\gamma,n \in \mathbb{Z}_{>0}^{odd}$ and $N\in\mathbb{Z}_{>0}$.  Let $n$ start an integral loop of length $N$ in the $3x+\gamma$ Problem.  Let $s = \gcd(n,\gamma)$ and let $p$ be a prime factor of $s$. Then $n/p$ starts an integral loop of length $N$ in the $3x+\gamma/p$ Problem.}
\end{lemma}

\begin{proof}Let $A$ be the trajectory of $n$ and $B$ be the trajectory of $n'$.
    \begin{equation*}
        \begin{aligned}
            A   &= \{C_\gamma^0(n), C_\gamma^1(n),...,C_\gamma^{N-1}(n),C_\gamma^N(n),...\}, \\
                &= \{n, 3n+\gamma,...,2n,n,...\}, \\
                &= \{p \frac{n}{p},3p \frac{n}{p}+p\frac{\gamma}{p},...,2  \frac{n}{p},p\frac{n}{p},...\}, \\
                &= p\{\frac{n}{p},3\frac{n}{p}+\frac{\gamma}{p},...,2\frac{n}{p},\frac{n}{p},...\}, \\
                &= p\{C_{\frac{\gamma}{p}}^0(\frac{n}{p}), C_{\frac{\gamma}{p}}^1(\frac{n}{p}),...,C_{\frac{\gamma}{p}}^{N-1}(\frac{n}{p}),C_{\frac{\gamma}{p}}^N(\frac{n}{p}),...\}, \\
                &= p B.
        \end{aligned}
    \end{equation*}
We find that the trajectory starting at $n/p$ has a period of length $N$, e.g. $C_{\frac{\gamma}{p}}^N(\frac{n}{p}) = \frac{n}{p}$.  Thus, by definition $n/p$ starts an integral loop of length N in the $3x+\gamma/p$ Problem.
\end{proof}

\begin{remark}
    We have shown that if the starting integer $n$ and the additive constant $\gamma$ both share a prime factor $p$, then we can divide by $p$, and create a new integral loop.
\end{remark}

\begin{lemma}\textit{
    Fix $\gamma,n \in\mathbb{Z}_{>0}^{odd}$ and $N \in\mathbb{Z}_{>0}$.  Let $n$ start an integral loop of length $N$.  Let $B\in \{0,1\}^{**}$ such that $B = (\chi_\gamma^i(n))_{i\geq 0}$.  Let $\rho$ and $\nu$ be the count of the number of zeros and ones in a single period of $B$.  Let $(i_j)_{j=0}^{\nu-1}$ be the indexing sequence associated to $n$.  Let $s = gcd(C_\gamma^{i_0}(n),...,C_\gamma^{i_\nu}(n))$, and let $s$ have prime factors $p_1,...,p_k$  Define $a_1,...,a_k \in\mathbb{Z}_{\geq 0}$ as the maximal exponents so that $p_j^{a_j}$ divides each component of $[n]$.  Then the number of integral loops that can be constructed by removing prime factors from the trajectory of $n$ is expressed by the following formula,}
\end{lemma}

\begin{equation*}
    \prod_{j=1}^k (1+a_j) .
\end{equation*}

\begin{proof}
    This is a very classic problem in combinatorics where one chooses $k$ objects from a set of $n$ objects with repeated elements.
\end{proof}

\begin{lemma}\textit{
    Fix $\gamma,n \in \mathbb{Z}_{>0}^{odd}$ and $N\in \mathbb{Z}_{>0}$.  Let $n$ start an integral loop of length $N$.  Let $B\in \{0,1\}^{**}$ such that $B = (\chi_\gamma^i(n))_{i\geq 0}$.  Let $\rho$ and $\nu$ be the count of the number of zeros and ones in a single period of $B$.  Let $(i_j)_{j=0}^{\nu-1}$ be the indexing sequence associated to $n$.  Let $s_j = \gcd(C_\gamma^{i_j}(n),C_\gamma^{i_{j+1}}(n))$ for all $0\leq j \leq \nu-1$.  Let $0\leq j'\leq \nu-1$ be fixed.  If $p$ is a prime factor of $s_{j'}$ of multiplicity $k>0$, then it is a prime factor for all $s_j$ of multiplicity $k$.}
\end{lemma}

\begin{proof}
    Recall from Lemma \ref{alpha} that if $C_\gamma^{i_{j'}}(n)$ and $C_\gamma^{i_{j'+1}}(n)$ both have a prime factor $p$, then $\gamma$ must also have the same prime factor.  Again as a consequence of Lemma \ref{alpha} we conclude that for all $0\leq j \leq \nu-1$, $C_\gamma^{i_j}(n)$ has prime factor $p$.  By Lemma \ref{beta} we may divide the integral loop starting with $C_\gamma^{i_{j'}}(n)$ by $p$, yielding a new integral loop.  If the multiplicity of $p$ is one, then we are done.  Else, we reapply Lemma \ref{alpha} to show $\gamma$ has another factor of $p$.  Again by Lemma \ref{alpha} we find that for all $0\leq j \leq \nu-1$, $C_\gamma^{i_j}(n)$ has another copy of $p$.  We then apply Lemma \ref{beta} to divide the integral loop by another copy of $p$.  We proceed inductively until all $k$ copies of $p$ have been exhausted.

    Suppose there exists another $0\leq j'' \leq \nu-1$ such that $s_{j''}$ has prime factor $p$ but with multiplicity $k'' > k$.  By the preceding argument we would have been able to divide $C_\gamma^{i_{j'}}(n)$ and $C_\gamma^{i_{j'+1}}(n)$ by another factor of $p$.  However, we had already assumed that $s_{j'}$ had only $k$ copies of $p$.  This is a contradiction.  As $s_{j'}$ was arbitrary, we conclude that if $p$ is a prime factor of $s_{j'}$ of multiplicity $k>0$ then it is a prime factor for all $s_j$ of multiplicity $k$.
\end{proof}

\begin{lemma}\label{gamma}\textit{
    Fix $\gamma,n \in \mathbb{Z}_{>0}^{odd}$ and $N\in\mathbb{Z}_{>0}$.  Let $n$ start an integral loop of length $N$.  Let $s= \gcd(n,C_\gamma^1(n),\gamma)$ and $p$ a prime factor of $s$.  Define $a,\tilde{a},b \in \mathbb{Z}_{\geq 0}$ as the maximal exponents such that $p^a$ divides $n$, $p^{\tilde{a}}$ divides $C_\gamma^1(n)$, and $p^b$ divides $\gamma$.  Then $b \geq \min(a,\tilde{a})$.}
\end{lemma}

\begin{proof}
    Suppose $b<\min(a,\tilde{a})$.  Then we may apply Lemma \ref{beta} to remove $b$ copies of $p$ from the integral loop.  Then we may apply Lemma \ref{alpha} forcing $\gamma$ to have another factor of $p$.  This is a contradiction.  Therefore $b \geq \min(a,\tilde{a})$.
\end{proof}

\begin{proposition}\label{delta}\textit{
    Fix $\gamma,n \in \mathbb{Z}_{>0}^{odd}$ and $N\in\mathbb{Z}_{>0}$.  Let $n$ start an integral loop of length $N$.  Let $B\in \{0,1\}^{**}$ such that $B = (\chi_\gamma^i(n))_{i\geq 0}$.  Let $\rho$ and $\nu$ be the count of the number of zeros and ones in a single period of $B$.  Let $[n]$ be the associated integral loop vector of $n$.  Let $\mathcal{M}_n$ be constructed as by Equation (\ref{m_n}).  Let $(i_j)_{j=0}^{\nu-1}$ be the indexing sequence associated to $n$.  Let $s = \gcd(n,C_\gamma^{i_1}(n))$ and $p$ a prime factor of $s$.  Define $a,d\in\mathbb{Z}_{\geq 0}$ as the maximal exponents such that $p^a$ divides every component of $[n]$ and $p^d$ divides every component of $\mathcal{M}_n \overrightarrow{\mathfrak{3}}^\nu$.  Define $b,c\in\mathbb{Z}_{\geq 0}$ as the maximal exponents so that $p^b$ divides $\gamma$ and $p^c$ divides $(2^\rho - 3^\nu)$.  Then $a,b,c,$ and $d$ satisfy the following equality:}

    %Define $a,b,c,d \in \mathbb{Z}_{\geq 0}$ and $[n'], {\mathcal{M}_n \overrightarrow{\mathfrak{3}}^\nu}' \in {\mathbb{Z}_{>0}^\nu}^{odd}$, $\gamma', (2^\rho-3^\nu)' \in \mathbb{Z}_{>0}^{odd}$ such that $[n] = p^a [n']$, $\gamma = p^b \gamma'$, $2^\rho - 3^\nu = p^c(2^\rho - 3^\nu)'$, and $\mathcal{M}_n \overrightarrow{\mathfrak{3}}^\nu = p^d{\mathcal{M}_n \overrightarrow{\mathfrak{3}}^\nu}'$.

    \begin{equation*}
        a-b+c-d = 0.
    \end{equation*}
\end{proposition}

\begin{proof}
    This follows from Proposition \ref{3v}.
\end{proof}

\begin{remark}
    Recall from Lemma \ref{gamma} that $b\geq a$. Since $d>0$ this implies $c\geq b-a$.
\end{remark}

%-------------------------------------------------------------
%-------------------------------------------------------------

\subsection{Theorem Regarding Determinant of $\mathcal{M}_n$}

\begin{theorem}\label{mod_gamma_prime}\textit{
    We will use the notation presented in Proposition \ref{delta}.  Furthermore assume that $c>b-a$. Fix $\gamma, n \in \mathbb{Z}_{>0}^{odd}$ and $N\in \mathbb{Z}_{>0}$.  Suppose further $n$ starts an integral loop of length $N$.  Let $B \in \{0,1\}^{**}$ such that $B = (\chi_\gamma^i(n))_{i\geq 0}$.  Let $\rho$ and $\nu$ be the count of the number of zeros and ones in a single period of $B$.  Let $[n]$ be the associated integral loop vector of $n$.  Let $\mathcal{M}_n$ be constructed as by Equation (\ref{m_n}).  Let $p$ be a prime factor of $[n]$ of multiplicity $a>0$, then $\det(\mathcal{M}_n) \equiv 0 \pmod p$.}
\end{theorem}

\begin{proof}
    Recall Equation (\ref{central_2}) and Equation (\ref{integral_loop_formula_equ}).  We find that,

    \begin{equation*}
        \begin{aligned} {}
            [n] &= \frac{\gamma}{2^\rho-3^\nu} \mathcal{M}_n \overrightarrow{\mathfrak{3}}^\nu, \\
            p^a\times\left(\frac{[n]}{p^a}\right) &= p^{b-c} \times \frac{\gamma p^c}{(2^\rho-3^\nu)p^b} \mathcal{M}_n \overrightarrow{\mathfrak{3}}^\nu, \\
            \left(\frac{p^a\times p^c\frac{(2^\rho-3^\nu)}{p^c}}{p^b\frac{\gamma}{p^b}}\right)\left(\frac{[n]}{p^a}\right) &= \mathcal{M}_n \overrightarrow{\mathfrak{3}}^\nu, \\
            p^{c-(b-a)}\times\left(\frac{(2^\rho-3^\nu)p^b}{\gamma p^c}\right)\times \left(\frac{[n]}{p^a}\right) &= \mathcal{M}_n \overrightarrow{\mathfrak{3}}^\nu.
        \end{aligned}
    \end{equation*}

    \noindent On the left hand side, we have a product of three terms.  By hypothesis $c>b-a$, so the first term is a positive power of prime factor $p$.  The remaining two terms have no factor of $p$.  The quotient $[n]/p^a \in \mathbb{Z}^\nu$.  Thus the left hand side is a vector integer multiple of $p$.  Therefore we have,

    \begin{equation*}
        \overrightarrow{0} \equiv \mathcal{M}_n \overrightarrow{\mathfrak{3}}^\nu \pmod p .
    \end{equation*}

    \noindent As the vector $\overrightarrow{\mathfrak{3}}^\nu$ is descending powers of three, we can infer that $\overrightarrow{\mathfrak{3}}^\nu$ is not the zero vector.  Thus the null space of $\mathcal{M}_n$ modulo $p$ is at least one dimensional.  By the Rank-Nullity Theorem, we conclude that $\mathcal{M}_n$ is not full rank, and $\det(\mathcal{M}_n) \equiv 0 \pmod p$.

\end{proof}

\begin{example}Consider the integral loop induced by $n=18611$ and $\gamma = 3367$.  We find that $\rho = 12$ and $\nu = 6$.  With respect to the prime 37, we have $a = b = c = 1$.  Finally $\det(\mathcal{M}_{18611}) = -1910471786496 \equiv 0 {\pmod {37}}$.
\end{example}

\begin{remark}
    In Theorem \ref{mod_gamma_prime}, we require that $c>b-a$.  This is a necessary requirement.  Consider the case $c=b-a$.  If $a=0$ then the prime factor $p$ cannot divide $[n]$ and we are done.  On the other hand, consider the integral loop induced by $n=5$ and $\gamma = 5^4$.  In this case $\rho = 7$ and $\nu = 1$.  Thus $2^7 - 3^1 = 125 = 5^3$.  We find that $3 = 4-1$ but $\det(M_n) = 1 \not\equiv 0 \pmod 5$.  
\end{remark}

\begin{proposition}\textit{
    Fix $\gamma, n \in \mathbb{Z}_{>0}^{odd}$ and $N\in \mathbb{Z}_{\geq 0}$.  Suppose further $n$ starts an integral loop of length $N$.  Let $B \in \{0,1\}^{**}$ such that $B = (\chi_\gamma^i(n))_{i\geq 0}$.  Let $\rho$ and $\nu$ be the count of the number of zeros and ones in a single period of $B$.  Let $[n]$ be the associated integral loop vector of $n$.  Let $\mathcal{M}_n$ be constructed as by Equation (\ref{m_n}) and suppose it is invertible over $\mathbb{C}$.  If $\gamma | (2^\rho-3^\nu)$, then $\det(\mathcal{M}_n)$ contains all the prime factors of $[n]$.}
\end{proposition}

\begin{proof}
    Since $\gamma|(2^\rho - 3^\nu)$, $\mathcal{M}_n \overrightarrow{\mathfrak{3}}^\nu = [n] \frac{2^\rho - 3^\nu}{\gamma} \equiv \overrightarrow{0} \pmod p$ for each $p$ that divides all the terms of $[n]$.  Hence $\det(\mathcal{M}_n) \equiv 0 \pmod p$.
\end{proof}

\begin{example}
    Consider the integral loop induced by $n = 3403$ and $\gamma = 4193575$.  In this case $(2^\rho - 3^\nu) / \gamma = 1$.  We find that 83 is a prime factor for all components of $[3403]$.  Finally $\det(\mathcal{M}_{3403} = 1198591822881410672230400 \equiv 0 {\pmod {83}}$.
\end{example}

\begin{example}\label{ex2}
Consider the integral loop induced by $n=133$ and $\gamma = 175$.  We find that $\overrightarrow{\mathfrak{3}}^4$ is in the null space of $\mathcal{M}_{133} \pmod 7$,

\begin{equation*} \begin{bmatrix} 1 & 2 & 8 & 64 \\ 1 & 4 & 32 & 128 \\ 1 & 8 & 32 & 64 \\ 1 & 4 & 8 & 32 \end{bmatrix} \begin{bmatrix} 27 \\ 9 \\ 3 \\ 1 \end{bmatrix} = \begin{bmatrix} 133 \\ 287 \\ 259 \\ 119 \end{bmatrix} \equiv \begin{bmatrix} 0 \\ 0 \\ 0 \\ 0\end{bmatrix}  \pmod 7 .\end{equation*}

\noindent Now consider the integral loops denoted by $A$ and $B$:

\begin{tabular}{c|c|c|l}
     & $n$  & $\gamma$ & \\
     \hline
    A & 133 & 175 & \begin{tabular}{l} $\{133, 574, 287, 1036, 518, 259, 952, 476, 238, 119, 532, 266\}$\end{tabular} \\
    B & 19 & 25 & \begin{tabular}{l} $\{19, 82, 41, 148, 74, 37, 136, 68, 34, 17, 76, 38\}$\end{tabular}
\end{tabular}

\noindent We find that the integral loop $B = A/7$.
\end{example}

\begin{remark}
We will discuss the method to arrive at Examples (\ref{ex1}) and (\ref{ex2}).  We begin with Example (\ref{ex1}).  The finite binary sequence $B = 10101010000000000000$ has four ones and sixteen zeros.  This allows us to implement the following trick regarding differences in squares.  Let $a,b\in \mathbb{Z}_{>0}$, then $a^2-b^2 = (a-b)(a+b)$.  In our case we are allowed to do this twice: $2^{16}-3^4 = (2^4-3)(2^4+3)(2^8+9) = 13\times19\times265$.  Now fix $m\in\mathbb{Z}_{\geq 0}$.  We have the following identity regarding the shuffling of integer powers $a$ and $b$:

\begin{equation*}
    \frac{a^m-b^m}{a-b} = \sum_{i=0}^{m-1}a^{m-1-i}b^i .
\end{equation*}

\noindent In our case $a=3$, $b = 2$, and $m=\nu$.  Thus, the denominator on the left hand side vanishes.  This yields the following:

\begin{equation*}
    3^\nu-2^\nu = \sum_{i=0}^{\nu-1}3^{\nu-1-i}2^i .
\end{equation*}

\noindent Compare the above to our finite binary sequence $B$ and the Integral Loop Formula.  We would find that the sums both match.

The last trick is strictly in number theory.  The integer sequence $(3^i-2^i)_{i\geq 0}$ has a repeating pattern modulo five and thirteen.  The former appears as $\{0,1,0,4,0,1,0,4,0,...\}$.  That is to say if $i \equiv 0 \pmod 2$ then $3^i-2^i \equiv 0 \pmod 5$.  The latter appears as $\{0,1,5,6,0,3,2,5,0,9,6,2,0,...\}$.  That is to say if $i \equiv 0\pmod 4$ then $3^i-2^i \equiv 0 \pmod {13}$.  We conclude that if $i \equiv 0\pmod 4$, then $3^i-2^i$ is both a multiple of five and thirteen.  Thus our starting integer $n$ and the quantity $2^{16}-3^4$ both have two prime factors in common.

As for Example (\ref{ex2}), the construct is largely the same.  The finite binary sequence $B = 101001000100$ has four ones and eight zeros.  We apply the same difference of squares to find $2^8-3^4 = (2^2-3)(2^2+3)(2^4+9) = 1\times 7\times 25$.  The fact that $133 = 7\times 19$ is a coincidence.  However,  $B$ was chosen to demonstrate that $\mathcal{M}_n$ can be non-invertible.  Consider the following summation:

\begin{equation*}
    \begin{aligned}
        S   &= \sum_{i=0}^{m-1}(-1)^{i+1}2^{2(m-1-i)}\times 2^{\frac{i(i+1)}{2}}, \\
            &= \sum_{i=0}^{m-1}(-1)^{i+1}2^{\frac{4m+i^2-3i-4}{2}}, \\
            &= -2^{2m-2}+2^{2m-3}-2^{2m-3}+2^{2m-2}+\sum_{i=4}^{m-1}(-1)^{i+1}2^{2(m-1-i)}\times 2^{\frac{i(i+1)}{2}}.
    \end{aligned}
\end{equation*}

\noindent In our case, $m=4$ and the right-hand sum vanishes.  Thus we have,

\begin{equation*}
    S = -64+32-32+64 = 0.
\end{equation*}

\noindent That is to say, $\mathcal{M}_{133}\overrightarrow{\mathfrak{-4}}^4 = \overrightarrow{0}$.  We found this because $2^8-(-4)^4 = 0$.  This is not in the domain of the function $C_\gamma$, but is the reason $\mathcal{M}_{133}$ is non-invertible.  This will be explored again in chapter four.
\end{remark}

\begin{proposition}\label{mod_gamma}\textit{
    Fix $\gamma \in \mathbb{Z}_{>1}^{odd}$, $n \in \mathbb{Z}_{>0}^{odd}$, and $N\in\mathbb{Z}_{>0}$.  Suppose further that $n$ starts an integral loop of the $3x+\gamma$ problem of length $N$.  Let $B \in \{0,1\}^{**}$ such that $B = (\chi_\gamma^i(n))_{i\geq 0}$.  Let $\rho$ and $\nu$ be the count of the number of zeros and ones in a single period of $B$.  Let $\mathcal{M}_n$ be constructed as by Equation (\ref{m_n}).  Finally suppose the quantity $(2^\rho-3^\nu)$ is prime.  If $\gamma$ divides each component of $[n]$, then $\det(\mathcal{M}_n) \equiv 0 \pmod {2^\rho-3^\nu}$.}
\end{proposition}

\begin{proof}
    $\mathcal{M}_n \overrightarrow{\mathfrak{3}}^\nu = \frac{[n]}{\gamma} (2^\rho - 3^\nu) \equiv \overrightarrow{0} {\pmod {2^\rho - 3^\nu}}$.  Hence $\det(\mathcal{M}_n) \equiv 0 {\pmod {2^\rho - 3^\nu}}$.
\end{proof}

\begin{example}
    Consider the integral loop induced by $n = 7$ and $\gamma = 7$.  We find that $\rho = 4$ and $\nu = 2$ such that $2^4 - 3^2 = 7$.  The matrix $\mathcal{M}_7$ has two repeated rows and thus its determinant vanishes.  Consequently, $\det(\mathcal{M}_7) = 0 \equiv 0 \pmod 7$.
\end{example}

\section{Construction of Matrix $\mathcal{R}_n$}

\subsection{Classical Results of Linear Algebra}

\begin{definition}\normalfont{
    Fix $n\in\mathbb{Z}_{>0}$ and $x_1,...,x_n \in \mathbb{C}$.  We define the \textit{Vandermode matrix}\cite{vandermode} $\mathcal{V}(x_1,...,x_n)$ as,}

    \begin{equation*}
        \mathcal{V}(x_1,...,x_n) = 
        \begin{bmatrix}
            1 & x_1 & x_1^2 & \cdots & x_1^{n-1} \\
            1 & x_2 & x_2^2 & \cdots & x_2^{n-1} \\
            \vdots & \vdots & \vdots & \ddots & \vdots \\
            1 & x_n & x_n^2 & \cdots &  x_n^{n-1}
        \end{bmatrix}.
    \end{equation*}
\end{definition}

\begin{lemma}\label{vander_det}
    Fix $n \in \mathbb{Z}_{>0}$ and $x_1,...,x_n \in \mathbb{C}$.  Let $\mathcal{V}(x_1,...,x_n)$ be a Vandermode matrix. Then,

    \begin{equation}\label{vander}
        \det(\mathcal{V}(x_1,...,x_n)) = \prod_{1\leq i < j \leq n} (x_j-x_i).
    \end{equation}
\end{lemma}

\begin{proof}
    We will prove by induction\cite{vandermode}.  The case $n=1$ is trivial.  Consider the case $n=2$:

    \begin{equation*}
        \det(\mathcal{V}(x_1,x_2)) = \left| 
        \begin{matrix}
            1 & x_1 \\
            1 & x_2
        \end{matrix}
        \right| = x_2 - x_1 .
    \end{equation*} 

    \noindent Now suppose the Equation (\ref{vander}) holds up to $n-1$.  Then for a Vandermode matrix of order $n$ we have,

    \begin{equation*}
        \begin{aligned}
            \det(\mathcal{V}(x_1,...,x_n)) &= 
                \left|\begin{matrix}
                    1 & x_1 & x_1^2 & \cdots & x_1^{n-1} \\
                    1 & x_2 & x_2^2 & \cdots & x_2^{n-1} \\
                    \vdots & \vdots & \vdots & \ddots & \vdots \\
                    1 & x_n & x_n^2 & \cdots &  x_n^{n-1}
                \end{matrix}\right|_{n\times n}, \\
            &= \left|\begin{matrix}
                    1 & x_1 & x_1^2 & \cdots & x_1^{n-1} \\
                    0 & x_2-x_1 & x_2^2-x_1^2 & \cdots & x_2^{n-1}-x_1^{n-1} \\
                    \vdots & \vdots & \vdots & \ddots & \vdots \\
                    0 & x_n-x_1 & x_n^2-x_1^2 & \cdots &  x_n^{n-1}-x_1^{n-1}
                \end{matrix}\right|_{n\times n}, \\
            &= \left|\begin{matrix}
                x_2-x_1 & x_2^2-x_1^2 & \cdots & x_2^{n-1}-x_1^{n-1} \\
                \vdots & \vdots & \ddots & \vdots \\
                 x_n-x_1 & x_n^2-x_1^2 & \cdots &  x_n^{n-1}-x_1^{n-1}
            \end{matrix}\right|_{(n-1)\times (n-1)}.
        \end{aligned}
    \end{equation*}

    \noindent Where the subscript $n\times n$ indicates the size of the determinant.

    \begin{equation*}
        \begin{aligned}
            \det&(\mathcal{V}(x_1,...,x_n)) \\
            &= \left|\begin{bmatrix}
                x_2 - x_1 & 0 & 0 & \cdots & 0 \\
                0 & x_3-x_1 & 0 & \cdots & 0 \\
                0 & 0 & x_4 - x_1 & \cdots & 0 \\
                \vdots & \vdots & \vdots & \ddots & \vdots \\
                0 & 0 & 0 & \cdots & x_n - x_1
            \end{bmatrix}
            \begin{bmatrix}
                1 & x_2+x_1 & \cdots & \sum_{i=0}^{n-2}x_2^{n-2-i}x_1^i \\
                1 & x_3+x_1 & \cdots & \sum_{i=0}^{n-2}x_3^{n-2-i}x_1^i \\
                \vdots & \vdots & \ddots & \vdots  \\
                1 & x_n+x_1 & \cdots & \sum_{i=0}^{n-2}x_n^{n-2-i}x_1^i \\
            \end{bmatrix}\right|, \\
            &= \prod_{j=2}^{n} (x_j-x_1) \left| \begin{matrix}
                1 & x_2+x_1 & \cdots & \sum_{i=0}^{n-2}x_2^{n-2-i}x_1^i \\
                1 & x_3+x_1 & \cdots & \sum_{i=0}^{n-2}x_3^{n-2-i}x_1^i \\
                \vdots & \vdots & \ddots & \vdots  \\
                1 & x_n+x_1 & \cdots & \sum_{i=0}^{n-2}x_n^{n-2-i}x_1^i \\
            \end{matrix}\right|, \\
            &= \prod_{j=2}^{n} (x_j-x_1) \left| \begin{bmatrix}
                    1 & x_2 & x_2^2 & \cdots & x_2^{n-1} \\
                    1 & x_3 & x_3^2 & \cdots & x_3^{n-1} \\
                    \vdots & \vdots & \vdots & \ddots & \vdots \\
                    1 & x_n & x_n^2 & \cdots &  x_n^{n-1}
                \end{bmatrix}
                \begin{bmatrix}
                    1 & x_1 & x_1^2 & \cdots & x_1^{n-2} \\
                    0 & 1 & x_1 & \cdots & x_1^{n-3} \\
                    0 & 0 & 1 & \cdots & x_1^{n-4} \\
                    \vdots & \vdots & \vdots & \ddots & \vdots \\
                    0 & 0 & 0 & \cdots & 1
                \end{bmatrix}\right|, \\
            &= \prod_{j=2}^{n} (x_j-x_1) \left| \begin{matrix}
                    1 & x_2 & x_2^2 & \cdots & x_2^{n-1} \\
                    1 & x_3 & x_3^2 & \cdots & x_3^{n-1} \\
                    \vdots & \vdots & \vdots & \ddots & \vdots \\
                    1 & x_n & x_n^2 & \cdots &  x_n^{n-1}
                \end{matrix}\right|.
        \end{aligned}
    \end{equation*}
    \noindent We then conclude by the induction hypothesis that the matrix determinant matches Equation (\ref{vander}).  Thus we have,

    \begin{equation*}
        \det(\mathcal{V}(x_1,...,x_n)) = \prod_{j=2}^{n} (x_j-x_1)\prod_{2\leq i<j\leq n} (x_j-x_i) = \prod_{1\leq i < j \leq n} (x_j-x_i)
    \end{equation*}
\end{proof}

\begin{definition}\normalfont
    Fix $n\in\mathbb{Z}_{>0}$ and $x_0,...,x_{n-1}\in\mathbb{C}$.  We define the \textit{Circulant matrix}\cite{Hugh2004}\cite{LEHMER197343} $\mathcal{R}(x_0,...,x_{n-1})$ as,

    \begin{equation}\label{circulant_general}
        \mathcal{R}(x_0,...,x_{n-1}) = \sum_{i=0}^{n-1}x_i \mathcal{P}^i
        = \begin{bmatrix}
            x_0 & x_1 & x_2 & \cdots & x_{n-2} & x_{n-1} \\
            x_{n-1} & x_0 & x_1 & \cdots & x_{n-3} & x_{n-2} \\
            x_{n-2} & x_{n-1} & x_0 & \cdots & x_{n-4} & x_{n-3} \\
            \vdots & \vdots & \vdots & \ddots & \vdots & \vdots \\
            x_2 & x_3 & x_4 & \cdots & x_0 & x_1 \\
            x_1 & x_2 & x_3 & \cdots & x_{n-1} & x_0
        \end{bmatrix}.
    \end{equation}
    
 \end{definition}

 \begin{example}
    Consider the numbers $0,1,2$ and $3$ in that order.  Then the circulant matrix $\mathcal{R}(0,1,2,3)$ appears as,

    \begin{equation*}
        \mathcal{R}(0,1,2,3) = \sum_{j=0}^3 j \mathcal{P}^{j} = \begin{bmatrix}
            0 & 1 & 2 & 3 \\
            3 & 0 & 1 & 2 \\
            2 & 3 & 0 & 1 \\
            1 & 2 & 3 & 0
        \end{bmatrix}.
    \end{equation*}
\end{example}

 \begin{proposition}
     Fix $n\in\mathbb{Z}_{>0}$ and $x_0,...,x_{n-1}\in\mathbb{C}$.  Let $\mathcal{R}(x_0,...,x_{n-1})$ be a Circulant matrix\cite{circulant_proof}.  Let $\epsilon = e^{\frac{2\pi i}{n}}$ be a primitive $n$-th root of unity.  Then,

     \begin{equation}\label{circ_det_gen}
        \det(\mathcal{R}(x_0,...,x_{n-1})) = \prod_{k=0}^{n-1}\left(\sum_{h=0}^{n-1}x_h \epsilon^{kh}\right).
     \end{equation}
 \end{proposition}

 \begin{proof}
    Let $\omega$ be any $n$-th root of unity and consider the vector,

    \begin{equation*}
        \overrightarrow{v} = 
        \begin{bmatrix} 1 \\ \omega \\ \omega^2 \\ \vdots \\ \omega^{n-1}
        \end{bmatrix}
    \end{equation*}

    \noindent Consider the first component in the product $\mathcal{R}(x_0,...,x_{n-1})\overrightarrow{v}$,

    \begin{equation*}
        x_0 + x_1\omega+x_2\omega^2+...+x_{n-1}\omega^{n-1} := \lambda.
    \end{equation*}

    \noindent The second component in the product $\mathcal{R}(x_0,...,x_{n-1})\overrightarrow{v}$ is a multiple of $\lambda$,

    \begin{equation*}
        \begin{aligned}
        x_{n-1} + x_0\omega+x_1\omega^2+...+x_{n-2}\omega^{n-1} &= (x_{n-1}\omega^{n-1} + x_0+x_1\omega+...+x_{n-2}\omega^{n-2})\omega, \\
        &= \lambda \omega.
        \end{aligned}
    \end{equation*}

    \noindent In general we find the $i$-th component of $\mathcal{R}(x_0,...,x_{n-1})\overrightarrow{v}$ is $\omega^{i-1}\lambda$.  Thus $\mathcal{R}(x_0,...,x_{n-1})\overrightarrow{v} = \lambda \overrightarrow{v}$.  As $\overrightarrow{v}$ is not the zero vector, we conclude it is an eigenvector with eigenvalue $\lambda$.  The same property is true for any $n$-th root of unity.  Thus let $\omega = \epsilon^k$ for $0\leq k\leq n-1$.  Thus,

    \begin{equation*}
        \overrightarrow{v_k} = 
        \begin{bmatrix}
            1 \\ \epsilon^k \\ \epsilon^{2k} \\ \vdots \\ \epsilon^{(n-1)k}
        \end{bmatrix}.
    \end{equation*}

    \noindent is an eigenvector with eigenvalue,

    \begin{equation*}
    \lambda_k = x_0 + x_1\epsilon^k +x_2\epsilon^{2k}+...+x_{n-1}\epsilon^{(n-1)k}.
    \end{equation*}

    Furthermore, the vectors $\overrightarrow{v_k}$ are linearly independent.  Consider the matrix $B$ whose columns the vectors $\overrightarrow{v_k}$.

    \begin{equation*}
        B = \begin{bmatrix}
            1 & 1 & 1 & \cdots & 1 \\
            1 & \epsilon & \epsilon^2 & \cdots & \epsilon^{n-1} \\
            1 & \epsilon^2 & \epsilon^4 & \cdots & \epsilon^{2(n-1)} \\
            \vdots & \vdots & \vdots &\ddots & \vdots \\
            1 & \epsilon^{n-1} & \epsilon^{2(n-1)} & \cdots & \epsilon^{(n-1)(n-1)} \\
        \end{bmatrix}.
    \end{equation*}
    \noindent $B$ is a Vandermode matrix.  By Lemma \ref{vander_det} we have proven the determinant of $B$ is $\prod_{0\leq i<j\leq n-1} (\epsilon^j-\epsilon^i) \neq 0$.  Thus $B$ is nonsingular, and $\lambda_k$ for $0\leq k \leq n-1$ are the eigenvalues of $\mathcal{R}(x_0,...,x_{n-1})$.  As the determinant of a matrix is the product of its eigenvalues.  Thus we have,

    \begin{equation*}
        \det(\mathcal{R}(x_0,...,x_{n-1})) =\prod_{k=0}^{n-1}\lambda_k = \prod_{k=0}^{n-1}\left(\sum_{h=0}^{n-1}x_h \epsilon^{kh}\right).
    \end{equation*}
 \end{proof}

 \begin{definition}\normalfont
    Fix $n\in\mathbb{Z}_{>0}$ and $x_0,...,x_{n-1}\in \mathbb{C}$.  Let $\mathcal{R}(x_0,...,x_{n-1}$ be a Circulant matrix.  We call $f(w) = x_0+x_1w+x_2w^2+...+x_{n-1}w^{n-1}$ the \textit{associated polynomial} to $\mathcal{R}(x_0,...,x_{n-1})$.
 \end{definition}

 \begin{lemma}\label{w_zero}
     Fix $n\in\mathbb{Z}_{>0}$ and $x_0,...,x_{n-1}\in \mathbb{C}$.  Let $\mathcal{R}(x_0,...,x_{n-1})$ be a Circulant matrix.  Then the rank of $\mathcal{R}(x_0,...,x_{n-1})$ is $n-d$, where $d$ is the degree of polynomial $g(w) = \gcd(f(w),w^n-1)$\cite{3498906} .
 \end{lemma}

 \begin{proof}
 Let $\mu_n$ be the set of complex solutions to $w^n-1$.  It is immediate that we can rewrite Equation (\ref{circ_det_gen}) in terms of $f(w)$.
    \begin{equation*}
        \det(\mathcal{R}(x_0,...,x_{n-1})) = \prod_{k=0}^{n-1}\left(\sum_{h=0}^{n-1}x_h \epsilon^{kh}\right) = \prod_{\zeta \in \mu_n} f(\zeta).
    \end{equation*}

    \noindent Thus $f(w)$ vanishes if and only if $w^n-1$ have a common zero.
 \end{proof}

 \begin{lemma}\label{w_zero_char_p}
     Fix $n,p\in\mathbb{Z}_{>0}$ and $x_0,...,x_{n-1}\in \mathbb{C}$.  Let $p$ be prime with $p\nmid n$.  Let $\mathcal{R}(x_0,...,x_{n-1})$ be a Circulant matrix.  Then $\det(\mathcal{R}(x_0,...,x_{n-1})) \equiv 0 \pmod p$ if and only if $f(w)$ and $w^n-1$ share a common zero\cite{3498906} .
 \end{lemma}

 \begin{proof}
     Write $n = p^t m$ for some $t,m \in \mathbb{Z}$ and $p \nmid m$.  Then $w^n-1 = (w^m-1)^{p^t}$.  Following from Lemma \ref{w_zero},

     \begin{equation*}
         \det(\mathcal{R}(x_0,...,x_{n-1})) = \prod_{\zeta \in \mu_m} f(\zeta)^{p^t}.
     \end{equation*}

     \noindent Where $\mu_m$ is the $m$ roots of unity in the algebraic closure of the field $\mathbb{F}_p$.
 \end{proof}

 %\begin{lemma}\label{prime}\textit{
 %   Let $n$ and $p$ be distinct primes such that the residue class $p \pmod n$ generates the whole group $\mathbb{Z}_n^{*}$.  Then for any Circulant matrix $\mathcal{R}(z_1,...,z_n) = z_1\mathcal{P}^0+z_2\mathcal{P}^1+...z_n\mathcal{P}^{n-1}$ with $z_1,...,z_n \in \mathbb{Z}$, the condition $p | \det(\mathcal{R}(z_1,...,z_n))$ is satisfied if and only if either $p|(z_1+...+z_n)$ or $z_1 \equiv ... \equiv z_n \pmod p$.\cite{SBURLATI2010100}}
%\end{lemma}

%^^^^^^^^^^^^^^^^^^^^^^^^^^^^^^^^^^^^^^^^^^^^^^^^^^^^^^^^^^^^^^^^^^^^^^^^
%^^^^^^^^^^^^^^^^^^^^^^^^^^^^^^^^^^^^^^^^^^^^^^^^^^^^^^^^^^^^^^^^^^^^^^^^
%^^^^^^^^^^^^^^^^^^^^^^^^^^^^^^^^^^^^^^^^^^^^^^^^^^^^^^^^^^^^^^^^^^^^^^^^
%^^^^^^^^^^^^^^^^^^^^^^^^^^^^^^^^^^^^^^^^^^^^^^^^^^^^^^^^^^^^^^^^^^^^^^^^
%^^^^^^^^^^^^^^^^^^^^^^^^^^^^^^^^^^^^^^^^^^^^^^^^^^^^^^^^^^^^^^^^^^^^^^^^
%^^^^^^^^^^^^^^^^^^^^^^^^^^^^^^^^^^^^^^^^^^^^^^^^^^^^^^^^^^^^^^^^^^^^^^^^
%^^^^^^^^^^^^^^^^^^^^^^^^^^^^^^^^^^^^^^^^^^^^^^^^^^^^^^^^^^^^^^^^^^^^^^^^
%^^^^^^^^^^^^^^^^^^^^^^^^^^^^^^^^^^^^^^^^^^^^^^^^^^^^^^^^^^^^^^^^^^^^^^^^
%^^^^^^^^^^^^^^^^^^^^^^^^^^^^^^^^^^^^^^^^^^^^^^^^^^^^^^^^^^^^^^^^^^^^^^^^
%^^^^^^^^^^^^^^^^^^^^^^^^^^^^^^^^^^^^^^^^^^^^^^^^^^^^^^^^^^^^^^^^^^^^^^^^
%^^^^^^^^^^^^^^^^^^^^^^^^^^^^^^^^^^^^^^^^^^^^^^^^^^^^^^^^^^^^^^^^^^^^^^^^
%^^^^^^^^^^^^^^^^^^^^^^^^^^^^^^^^^^^^^^^^^^^^^^^^^^^^^^^^^^^^^^^^^^^^^^^^

\subsection{A Second Theorem Regarding the Determinant of $\mathcal{M}_n$}

\begin{theorem}\label{combo}
    Fix $n\in\mathbb{Z}_{>0}$.  Let $F(n)$ be the number of $n$-tuples $(x_0,x_1,...,x_{n-1})$ satisfying\cite{Brualdi1970AnEP},

    \begin{equation*}
        \begin{aligned}
        x_i \geq 0&, (i=0,1,...,n-1)\\
        \sum_{i=0}^{n-1}x_i &= n,\\
        \sum_{i=0}^{n-1}ix_i &\equiv 0 \pmod n . \\
        \end{aligned}
    \end{equation*}
    \noindent Then,

    \begin{equation*}
        F(n) = \frac{1}{n} \sum_{d | n} {{2d-1} \choose d} \phi\left(\frac{n}{d}\right),
    \end{equation*}

    \noindent where the summation extends over all positive integers $d$ dividing $n$, and where $\phi$ is Euler's function.
\end{theorem}

\begin{proof}
    Define $f_n(w,z) = [(1-z)(1-wz)...(1-w^{n-1}z)]^{-1}$.  Then,

    \begin{equation*}
        f_n(w,z) = \left(\sum_{k=0}^\infty z^k\right)\left(\sum_{k=0}^\infty w^kz^k\right)...\left(\sum_{k=0}^\infty w^{k(n-1)}z^k\right).
    \end{equation*}

    \noindent $F(n)$ is the sum of coefficients of $z^nw^{nt}$, for $0\leq t \leq n-1$, in $f_n(w,z)$.  Write $f_n(w,z) = \sum_{k=0}^{\infty} B_k z^k = B_k(n,w)$.  Then because,
    \begin{equation*}
        f_{n+1}(w,z) =\frac{f_n(w,z)}{1-w^nz},
    \end{equation*}

    \noindent and $f_n(w,wz) = [(1-wz)(1-w^2z)...(1-w^{n}z)]^{-1}$, then $f_n(w,wz) = (1-z)f_{n+1}(w,z) = \frac{1-z}{1-w^nz}f_n(w,z)$.  Thus,

    \begin{equation*}
        \begin{aligned}
        \sum_{k=0}^{\infty} B_k w^k z^k &= \frac{1-z}{1-w^nz}\sum_{k=0}^{\infty}B_k z^k, \\
        \sum_{k=0}^{\infty}B_k w^k z^k - \sum_{k=0}^\infty B_k w^{n+k}z^{k+1} &= \sum_{k=0}^\infty B_k z^k - \sum_{k=0}^\infty B_k z^{k+1}.
        \end{aligned}
    \end{equation*}

    \noindent Hence for $k\geq 1$, $B_k w^k - B_{k-1}w^{n+k-1} = B_k - B_{k-1}$, or 

    \begin{equation*}
        B_k = \frac{1-w^{n+k-1}}{1-w^k}B_{k-1}, (k\geq 1).
    \end{equation*}

    Since $B_0 = 1$, then $B_k = \prod_{r=1}^k \frac{1-w^{n+r-1}}{1-w^r}$, $(k\geq 0)$. With empty product being 1.  Therefore,

    \begin{equation*}
        f_n(w,z) = \sum_{k=0}^\infty \left\{\prod_{r=1}^k \frac{1-w^{n+r-1}}{1-w^r}\right\}z^k ,
    \end{equation*}
    \noindent and $F(n)$ is the sum of the coefficients of $w^{nt}$, $0\leq t \leq n-1$, in $g_n(w) = \prod_{r=1}^n \frac{1-w^{n+r-1}}{1-w^r} = \prod_{r=1}^{n-1}\frac{1-w^{n+r}}{1-w^r}$.  $g_n(w)$ is a polynomial in $w$ of degree $\sum_{r=1}^{n-1} \{n+r-r\} = n(n-1)$ and has nonnegative coefficients.  Since,

    \begin{equation*}
        \sum_{\zeta \in \mu_n} \zeta^k = \left\{\begin{matrix}
             n, \text{if $n$ divides $k$} \\ 0, \text{else,}
        \end{matrix}\right.
    \end{equation*}

    \noindent we have $nF(n) = \sum_{\zeta \in \mu_n}g_n(\zeta)$, where the summation extends over all $n$-th roots of unity.  Suppose $\zeta$ is a primitive $d$-th root of unity where $d|n$.  Since,

    \begin{equation*}
        \lim_{w\to \zeta} \frac{1-w^{n+r}}{1-w^r} = \left\{\begin{matrix}
            \frac{n+r}{r}, \text{ if $d$ divides $r$} \\ 1, \text{ else.}
        \end{matrix}\right.
    \end{equation*}

    Then we have that,

    \begin{equation*}
        g_n(\zeta) = \prod_{1\leq r \leq n-1 \\ r\equiv 0 \pmod d} \frac{n+r}{r} = \prod_{s=1}^{\frac{n}{d}-1}\frac{n+sd}{sd} = {{2\frac{n}{d}-1}\choose {\frac{n}{d}}}.
    \end{equation*}

    \noindent Therefore, since there are $\phi(d)$ $n$-th roots of unity which are primitive $d$-th roots of unity,

    \begin{equation*}
        F(n) = \frac{1}{n} \sum_{d | n} {{2\frac{n}{d}-1} \choose {\frac{n}{d}}} \phi\left(d\right) = \frac{1}{n} \sum_{d | n} {{2d-1} \choose d} \phi\left(\frac{n}{d}\right).
    \end{equation*}
\end{proof}

%\begin{definition}
%    Fix $n\in\mathbb{Z}_{>0}$ and $z_1,...,z_n \in \mathbb{C}$.  Then the matrix $\mathcal{R}(z_1,...,z_n) = \sum_{i=1}^n z_i \mathcal{P}^{i-1}$ is called a \textit{circulant matrix}\cite{Hugh2004}\cite{LEHMER197343}.
%\end{definition}

\begin{definition}\label{z_i}
    Fix $\gamma, n \in \mathbb{Z}_{>0}^{odd}$ and $N\in \mathbb{Z}_{>0}$.  Suppose further $n$ starts an integral loop of length $N$.  Let $B = \{0,1\}^{**}$ such that $B = (\chi_\gamma^i(n))_{\geq 0}$.  Let $\rho$ and $\nu$ be the count of the number of zeros and ones in a single period of $B$.  Let the quantities $m_j$ be given by Proposition \ref{integral_loop_formula}.  Let $z_1,...,z_\nu$ be defined by,

    \begin{equation}\label{z_def}
        z_i = 2^{\frac{1}{\nu}\sum_{k=1}^{\nu-1}k m_{|i+k-2|_\nu+1}}.
    \end{equation}

    \noindent Where $|i+k-2|_\nu$ is $i+k-2$ modulo $\nu$.  Then we define the circulant matrix $\mathcal{R}_n$ associated to $n$ as,

    \begin{equation}\label{r_n}
        \mathcal{R}_n = \mathcal{R}(z_1,...,z_\nu) = \sum_{j=1}^\nu z_j \mathcal{P}^{j-1}.
    \end{equation}
\end{definition}

\begin{example}
    Consider the integral loop induced by $n= 19$ and $\gamma = 485$.  In this case a single period of $B = 101010000000$ with $m_1 = 1, m_2 = 1$ and $m_3= 7$.  Solving for the coefficients of $z_1,z_2$, and $z_3$ we have:

    \begin{equation*}
        \begin{aligned}
            z_1 &= 2^{\frac{m_1 + 2m_2}{3}} = 2^{\frac{1 + 2\times1}{3}} = 2, \\
            z_2 &= 2^{\frac{m_2 + 2m_3}{3}} = 2^{\frac{1 + 2\times7}{3}} = 32, \\
            z_3 &= 2^{\frac{m_3 + 2m_1}{3}} = 2^{\frac{7 + 2\times1}{3}} = 8. \\
        \end{aligned}
    \end{equation*}

    \noindent Finally $\mathcal{R}_{19}$ appears as,

    \begin{equation*}
        \mathcal{R}_{19} = \mathcal{R}(2,32,8) = 
        \begin{bmatrix}
            2 & 32 & 8 \\
            8 & 2 & 32 \\
            32 & 8 & 2
        \end{bmatrix}.
    \end{equation*}
\end{example}

For the following, let $S_n$ be the symmetric group acting on $n$ elements.

\begin{lemma}\label{a_b}\textit{
    Fix $n\in\mathbb{Z}_{>0}$.  Let $a_0,...,a_{n-1},b_0,...,b_{n-1}\in\mathbb{R}_{>0}$.  Define $\mathcal{R}(a_0,...,a_{n-1})$ as by Equation (\ref{circulant_general}) and $\mathcal{M}(b_0,...,b_{n-1})$ by Equation (\ref{gen_m}) below,}

\begin{equation}\label{gen_m}
    \mathcal{M}(b_0,...,b_{n-1}) = 
    \begin{bmatrix}
        1 & b_0 & b_0 b_1 & ... & b_0 b_1 ... b_{n-2} \\
        1 & b_1 & b_1 b_2 & ... & b_1 b_2 ... b_{n-1} \\
        \vdots & \vdots & \ddots & \vdots \\
        1 & b_{n-2} & b_{n-2} b_{n-1} & ... & b_{n-2} b_n ... b_{n-4} \\
        1 & b_{n-1} & b_{n-1} b_0 & ... & b_{n-1} b_0 ... b_{n-3}
    \end{bmatrix}.
\end{equation}

\noindent \textit{Then there is a one to one correspondence between the polynomial terms that appear in the determinant of $\mathcal{R}(a_0,...,a_{n-1})$ and $\mathcal{M}(b_0,...,b_{n-1})$.}
\end{lemma}

\begin{proof}
    Let $r_{i,j}$ be the element of matrix $\mathcal{R}(a_0,...,a_{n-1})$ and $s_{i,j}$ be the element of matrix $\mathcal{M}(b_0,...,b_{n-1})$ at the $i$-th row and $j$-th column, respectively.  Recall the definition of the determinant.

    \begin{equation*}
        \begin{aligned}
            \det(\mathcal{R}(a_0,...,a_{n-1})) &= \sum_{\sigma \in S_n}sgn(\sigma) \prod_{i = 1}^{n} r_{i,\sigma(i)}, \\
            \det(\mathcal{M}(b_0,...,b_{n-1})) &= \sum_{\sigma \in S_n}sgn(\sigma) \prod_{i = 1}^{n} s_{i,\sigma(i)}. \\
        \end{aligned}
    \end{equation*}

    \noindent Fix $\sigma \in S_n$.  Then for some $A, A'\in\mathbb{Z}$, we have $A r_{1,\sigma(1)}...r_{n,\sigma(n)} = A a_0^{y_0}...a_{n-1}^{y_{n-1}}$ for some $y_0,...,y_{n-1} \in \mathbb{Z}_{\geq 0}$ and $A' s_{1,\sigma(1)}...s_{n,\sigma(n)} = A' b_0^{w_0}...b_{n-1}^{w_{n-1}}$ for some $w_0,...,w_{n-1} \in \mathbb{Z}_{\geq 0}$.  The former is a polynomial of degree $n$, and the latter is a polynomial of degree $n(n-1)/2$.
    %\noindent With respect to the Circulant matrix, we find each term is a homogeneous polynomial of degree $n$.  For $\sigma$ fixed, we have exponents $y_0,...,y_{n-1} \in \mathbb{Z}_{\geq 0}$ with $y_0+...+y_{n-1} = n$ such that a single term of the determinant appears as, $a_0^{y_0}...a_{n-1}^{y_{n-1}}$.  
    Consider first the permutations $\sigma_k(i) = i+k$ for $ 0 \leq k\leq n-1$.  We find that each corresponds to the monomial $a_{k}^n$.  Let $\theta \in S_n$ be the permutation $(1,n)(2,n-1)...(\frac{n+1}{2},\frac{n+1}{2})$ if $n$ is odd and $(1,n)(2,n-1)...(\frac{n}{2}-1,\frac{n}{2}+1)$ if $n$ is even.  We see that for all $\sigma \in S_n$, $\theta \circ \sigma (i) = n+1-\sigma(i)$.  As each permutation is an automorphism of $S_n$, then we can freely choose to compose $\theta$ and $\sigma$ in the definition of the determinant for $\mathcal{M}(b_0,...,b_{n-1})$ without changing the result.  We find that the matrix determinants of $\mathcal{R}(a_0,...,a_{n-1})$ and $\mathcal{M}(b_0,...,b_{n-1})$ have a one to one correspondence if and only if for each $a_k$ $(0\leq k \leq n-1)$ we have:

    \begin{equation*}
        a_k^n = b_{|n-k|_n}^1b_{|n+1-k|_n}^2b_{|n+2-k|_n}^3...b_{|n-1-k|_n}^{n-1}.
    \end{equation*}

    \noindent Where $|n-k|_n$ denotes $n-k$ modulo $n$.  Let $e_1,...,e_n$ and $\tilde{e}_1,...,\tilde{e}_n$ represent two bases for the space $\mathbb{R}^n$.  Then let us make the following substitution, $a_k^{y_k} \rightarrow y_{k-1} e_k$ and $b_k^{w_k} \rightarrow w_{k-1} \tilde{e}_k$.  We find that the exponents of each polynomial term associated to $\mathcal{R}(a_0,...,a_{n-1})$ corresponds to a polynomial term associated to $\mathcal{M}(b_0,...,b_{n-1})$ by the following linear transformation:

    \begin{equation}\label{circulant_proof}
    \frac{1}{n} 
    \begin{bmatrix}
        1 & 0 & n-1 & n-2 & \dots & 2 \\
        2 & 1 & 0 & n-1 & \dots & 3 \\
        3 & 2 & 1 & 0 & \dots & 4 \\
        \vdots & \vdots & \vdots & \vdots & \ddots & \vdots \\
        0 & n-1 & n-2 & n-3 & \dots & 1
    \end{bmatrix}
    \begin{bmatrix}
        y_0 \\
        y_1 \\
        y_2 \\
        \vdots \\
        y_{n-1}
    \end{bmatrix} = \begin{bmatrix}
        w_0 \\
        w_1 \\
        w_2 \\
        \vdots \\
        w_{n-1}
    \end{bmatrix}.
\end{equation}
    
    \noindent The determinant of the above linear transformation is $n^{n-1}(n-1)/2$, which is non-vanishing.  Conversely, if we were to start with exponents of each polynomial term associated to $\mathcal{M}(b_0,...,b_{n-1})$ and relate them to the exponents of the polynomial terms of $\mathcal{R}(a_0,...,a_{n-1})$ we would find an analogous result with the matrix inverse of Equation (\ref{circulant_proof}).
    
    \begin{equation}\label{circulant_proof_2}\small
    \begin{bmatrix}
        y_0 \\
        y_1 \\
        y_2 \\
        \vdots \\
        y_{n-1}
    \end{bmatrix} = 
    \frac{2}{n(n-1)} 
    \begin{bmatrix}
        1 & 1 & \dots & 1 & \frac{2+n(n-1)}{2} & \frac{2-n(n-1)}{2} \\
        \frac{2-n(n-1)}{2} & 1 & \dots & 1 &1 &  \frac{2+n(n-1)}{2} \\
        \frac{2+n(n-1)}{2} & \frac{2-n(n-1)}{2} & \dots & 1 & 1 & 1 \\
        \vdots & \vdots & \ddots & \vdots & \vdots & \vdots \\
        1 & 1 & \dots & \frac{2+n(n-1)}{2} & \frac{2-n(n-1)}{2} & 1
    \end{bmatrix}
    \begin{bmatrix}
        w_0 \\
        w_1 \\
        w_2 \\
        \vdots \\
        w_{n-1}
    \end{bmatrix}.
\end{equation}
    
    %\noindent However, by Theorem \ref{combo} that a necessary and sufficient condition for a polynomial term to exist in the permanent of $\mathcal{R}(a_0,...,a_{n-1})$ is that $\sum_i i y_i \equiv 0 \pmod n$.  This is equivalent to multiplying by circulant matrix in Equation (\ref{circulant_proof_2}) by the circulant matrix in Equation (\ref{circulant_proof}) scaled by $n$.  However, since these are inverses of each other, we find that the resultant product is just a scaling by $n$, thus satisfying the condition that the polynomial term exists in the permanent of $\mathcal{R}(a_0,...,a_{n-1})$.  We conclude that there is a one to one correspondence between the polynomial terms that appear in the determinant of $\mathcal{R}(a_0,...,a_{n-1})$ and $\mathcal{M}(b_0,...,b_{n-1})$.

    \noindent However, by Theorem \ref{combo} that a necessary and sufficient condition for a polynomial term to exist in the permanent of $\mathcal{R}(a_0,...,a_{n-1})$ is that $\sum_i i y_i \equiv 0 \pmod n$.  Call the matrix in Equation (\ref{circulant_proof_2}), $\theta_{j,k}$.  Thus we find that,

    \begin{equation*}
    \sum_{j=0}^{n-1} y_j = n, \hspace{0.5cm} \sum_{j=0}^{n-1} jy_j \equiv 0 \pmod n ,
    \end{equation*}

    \noindent if and only if,

    \begin{equation*}
    \sum_{j=0}^{n-1} \left( \sum_{k=0}^{n-1} \theta_{j+1,k+1} w_k \right) = n, \hspace{0.5cm} \sum_{j=0}^{n-1} j\left( \sum_{k=0}^{n-1} \theta_{j+1,k+1} w_k \right) \equiv 0 \pmod n .
    \end{equation*}

    \noindent  As $\theta_{j,k}$ is invertible, we conclude that there is a one to one correspondence between the polynomial terms that appear in the determinant of $\mathcal{R}(a_0,...,a_{n-1})$ and $\mathcal{M}(b_0,...,b_{n-1})$.
\end{proof}

\begin{proposition}\label{dets_r_m}
    \textit{Fix $\gamma, n \in \mathbb{Z}_{>0}^{odd}$ and $N\in \mathbb{Z}_{>0}$.  Suppose further $n$ starts an integral loop of length $N$.  Let $B = \{0,1\}^{**}$ such that $B = (\chi_\gamma^i(n))_{\geq 0}$.  Let $\rho$ and $\nu$ be the count of the number of zeros and ones in a single period of $B$.  Let $\mathcal{M}_n$ be constructed as by Equation (\ref{m_n}).  Let $\mathcal{R}_n$ be constructed by Definition \ref{z_i}.  Then the following equality is satisfied,}

    \begin{equation*}
        \left|\det(\mathcal{M}_n)\right| = \left|\det(\mathcal{R}_n)\right|.
    \end{equation*}
\end{proposition}

\begin{proof}
    Fix $\sigma \in S_\nu$.  Let $\theta \in S_n$ be the permutation $(1,\nu)(2,\nu-1)...(\frac{\nu+1}{2},\frac{\nu+1}{2})$ if $\nu$ is odd and $(1,\mu)(2,\nu-1)...(\frac{\nu}{2}-1,\frac{\nu}{2}+1)$ if $\nu$ is even.  Following from Lemma \ref{a_b} we recall that the polynomial terms of $\det(\mathcal{M}_n)$ and $\det(\mathcal{R}_m)$ have a one to one correspondence.  For a single polynomial term $\mathcal{R}_n$, we have $a_0^{y_0}...a_{\nu-1}^{y_{\nu-1}} = 2^{y_0m_1+y_1m_2+...+y_{\nu-1}m_\nu}$, and its corresponding polynomial term chosen by $\theta \circ \sigma$ in $\mathcal{M}_n$, $b_0^{w_0}...b_{\nu-1}^{w_{\nu-1}} = 2^{w_0 m_1+w_1m_2+...+w_{\nu-1}m_\nu}$.  The base of both polynomials is two, allowing us to compare the exponents directly.  The non-vanishing terms of $\det(\mathcal{R}_n)$ have exponents which coincide with the solutions to the following system of equations:

    \begin{equation}\label{sys_y}
        \begin{aligned}
            \sum_{j=1}^\nu j y_j \equiv 0 &\pmod \nu, \\
            \sum_{j=1}^\nu y_j = \nu.
        \end{aligned}
    \end{equation}

    \noindent The $y_0,...,y_{\nu-1}$ satisfying Equation (\ref{sys_y}) identically match Definition \ref{z_i}, and thus directly correspond to $w_0,...,w_{\nu-1}$.  Thus the polynomial term, apart from a sign change, is identical.  Thus $\left|\det(\mathcal{M}_n)\right| = \left|\det(\mathcal{R}_n)\right|$.
\end{proof}

 \begin{example}
    Consider the integral loop induced by $n= 19$ and $\gamma = 485$.  We find that,

        \begin{equation*}
            \det(\mathcal{M}_{19}) = \left| \begin{matrix}
                1 & 2 & 4 \\
                1 & 2 & 256 \\
                1 & 128 & 256
            \end{matrix} \right| = -31752, \hspace{0.25cm} \det(\mathcal{R}_{19}) = 
        \left|\begin{matrix}
            2 & 32 & 8 \\
            8 & 2 & 32 \\
            32 & 8 & 2
        \end{matrix}\right| = 31752.
        \end{equation*}

    Consider the integral loop induced by $n= 19$ and $\gamma = 101$.  We find that,

    \begin{equation*}
            \det(\mathcal{M}_{19}) = \left| \begin{matrix}
                1 & 2 & 4 \\
                1 & 2 & 64 \\
                1 & 32 & 64
            \end{matrix} \right| = -1800, \hspace{0.25cm} \det(\mathcal{R}_{19}) = 
        \left|\begin{matrix}
            2 & 2^{11/3} & 2^{7/3} \\
            2^{7/3} & 2 & 2^{11/3} \\
            2^{11/3} & 2^{7/3} & 2
        \end{matrix}\right| = 1800.
        \end{equation*}
    \end{example}

\begin{corollary}
    \textit{Fix $\gamma, n \in \mathbb{Z}_{>0}^{odd}$ and $N\in \mathbb{Z}_{>0}$.  Suppose further $n$ starts an integral loop of length $N$.  Let $B = \{0,1\}^{**}$ such that $B = (\chi_\gamma^i(n))_{\geq 0}$.  Let $\rho$ and $\nu$ be the count of the number of zeros and ones in a single period of $B$.  Let $\mathcal{M}_n$ be constructed as by Equation (\ref{m_n}).  Let $\mathcal{R}_n$ and be constructed by Definition \ref{r_n}.  Let $z_j$ be defined as by Equation (\ref{z_def}).  Finally let $\epsilon$ be the $n$-th root of unity.  Then the determinant\cite{LEHMER197343} of $\mathcal{M}_n$ can be expressed in the following formula:}

    \begin{equation}\label{circ_det}
        \left|\det(\mathcal{M}_n)\right| = \left|\prod_{k=0}^{\nu-1}\sum_{h=0}^{\nu-1} z_{h+1}\epsilon^{kh}\right|.
    \end{equation}
\end{corollary}

\begin{proof}
    This follows from Proposition \ref{dets_r_m} and Equation \ref{circ_det_gen}.
\end{proof}

Now that we have shown that the determinants of $\mathcal{M}_n$ and $\mathcal{R}_n$ have identical prime factorization, we can utilize the following lemma.  Assume a priori $z_1,...z_n \in \mathbb{Z}$.

\begin{lemma}\label{prime}\textit{
    Let $n$ and $p$ be distinct primes such that the residue class $p \pmod n$ generates the whole group $\mathbb{Z}_n^{*}$.  Then for any circulant $\mathcal{R}(z_1,...,z_n) = z_1\mathcal{P}^0+z_2\mathcal{P}^1+...z_n\mathcal{P}^{n-1}$ with $z_1,...,z_n \in \mathbb{Z}$, the condition $p | \det(\mathcal{R}(z_1,...,z_n))$ is satisfied if and only if either $p|(z_1+...+z_n)$ or $z_1 \equiv ... \equiv z_n \pmod p$.}\cite{SBURLATI2010100}
\end{lemma}

\begin{proof}
    Suppose $p$ divides $\det(\mathcal{R}(z_1,...,z_n))$ and recall Equation \ref{circ_det_gen}.  Consider first the case $k=0$.  We have the term $(z_1+...+z_n)$.  If $p$ divides the sum, then we are done.  Else, if $z_1 \equiv ... \equiv z_n \pmod p$ then for $k=1$ we have $(z_1+z_2\epsilon+...+z_n\epsilon^{n-1}) \equiv z_1 (1+\epsilon+...+\epsilon^{n-1}) \equiv 0 \pmod p$.  On the other hand by Lemma \ref{w_zero_char_p}, we know that $p\mid \det(\mathcal{R}(z_1,...,z_n))$ if and only if $\gcd(f(w),w^n-1) \neq c (c\in \mathbb{Z}_p^{*})$.  Assuming that $p\nmid (z_1+...+z_n)$, then 1 is not a common root.  Thus $p\mid \det(\mathcal{R}(z_1,...,z_n))$ is equivalent to $\gcd(f(w),1+w+...+w^{n-1}) \neq c \in \mathbb{Z}_p[w] (c\in \mathbb{Z}_p^{*})$.  We conclude that $p | \det(\mathcal{R}(z_1,...,z_n))$ is satisfied if and only if either $p|(z_1+...+z_n)$ or $z_1 \equiv ... \equiv z_n \pmod p$.
\end{proof}

In our case, we are concerned about the contrapositive.  Moreover, for the theorem to be effective, we require that the entries of $\mathcal{R}_n$ are integers.  Thus we can construct the following theorem.

\begin{theorem}\label{central_theorem_X}
    \textit{Fix $\gamma, n \in \mathbb{Z}_{>0}^{odd}$ and $N\in \mathbb{Z}_{>0}$.  Suppose further $n$ starts an integral loop of length $N$.  Let $B = \{0,1\}^{**}$ such that $B = (\chi_\gamma^i(n))_{\geq 0}$.  Let $\rho$ and $\nu$ be the count of the number of zeros and ones in a single period of $B$.  Assume $\nu$ is prime.  Let the quantities $m_j$ from Proposition \ref{integral_loop_formula} satisfy Equation (\ref{z_i}) so that $z_1,...,z_\nu$ are integers.  Let $\mathcal{R}_n$ be constructed by Equation (\ref{r_n}). Let $p$ be a prime factor of $2^\rho - 3^\nu$ distinct from $\nu$ so that the residue class $p \pmod \nu$ generates the whole group $\mathbb{Z}_\nu^{*}$.  Then $p\nmid \det(\mathcal{M}_n)$ if and only if $p\nmid(z_1+...+z_\nu)$ and for some $1\leq i<j\leq \nu$ we have $z_i \not\equiv z_j \pmod p$.}
\end{theorem}

\begin{proof}
    Recall from Proposition \ref{dets_r_m} that $|\det(\mathcal{M}_n)| = |\det(\mathcal{R}_n)|$.  Then consider the contrapositive of Lemma \ref{prime}.
\end{proof}

\begin{example}
    Consider the integral loop induced by $n=19$ and $\gamma = 485$.  In this case $m_1 = 1, m_2 = 1$ and $m_3 = 7$ such that $2^9-3^3 = 485 = 5\times 97$.  We find that $z_1 = 2, z_2 = 32$ and $z_3 = 8$.  Thus $\mathcal{R}_{19}$ has all integer coefficients.  The prime factor 5 is in the residue class of 2, which generates the whole group $\mathbb{Z}_3^{*}$.  We find that $2+32+8 = 42 \equiv 2 \pmod {5}$.  Furthermore $2 \equiv 2 \pmod 5$ and $8 \equiv 3 \pmod 5$.  Thus by the theorem we conclude that $\det(\mathcal{M}_{19})$ is nonsingular modulo 5.  Finally $\det(\mathcal{M}_{19}) \not\equiv 0 \pmod {2^{9}-3^3}$.
\end{example}

\begin{example}
    Consider the integral loop induced by $n=184337$ and $\gamma = 1048333$.  In this case $m_1 = 6, m_2 = 2, m_3 = 6, m_4 = 3$ and $m_5 = 3$ such that $2^{20}-3^5 = 1048333 = 11\times 13\times 7331$.  We find that $z_1 = 256, z_2 = 128, z_3 = 512, z_4 = 128$ and $z_5 = 512$.  Thus $\mathcal{R}_{184337}$ has all integer coefficients.  The prime factors 11 and 7331 are both in the residue class of 1, and cannot be used. However, the residue class of 13 is 3, which indeed generates the whole group $\mathbb{Z}_5^{*}$.  We find that $256+128+512+128+512 = 1536 \equiv 2 \pmod {13}$.  Furthermore $256 \equiv 9 \pmod {13}$ and $128 \equiv 11 \pmod {13}$.  Thus by the theorem we conclude that $\det(\mathcal{M}_{184337})$ is nonsingular modulo 13.  Finally, $\det(\mathcal{M}_{184337}) \not\equiv 0 \pmod {2^{20}-3^5}$.
\end{example}

\chapter{Perspectives}

 %---------------------------------------------------------------
 %///////////////////////////////////////////////////////////////
 %---------------------------------------------------------------

 \section{Discussion}

 Chapter five is broken up into four different discussions.  Each discussion is a segue into a topic related to the thesis, but not in the direct path towards Theorem \ref{central_theorem_X}.  Each also serves as an avenue which future research can be taken.  The first discussion is a complement to Examples (\ref{ex1}) and (\ref{ex2}).  We discuss first the general Exponential Diophatine Equation, then apply a combinatorical argument to give a bound to the ratio of ones and zeros in a finite binary sequence corresponding to an integral loop of the $3x+1$ Problem.  The second and third discussions are two different methods to map distinct integral loops of the $3x+\gamma$ problem into each other.   The fourth discussion demonstrates that the integral loop vector is the unique solution of an eigenvector problem with eigenvalue one.  
 
 %The fifth discussion is a method to map integral loops of the $3x+\gamma$ problem to the $Yx+\lambda$ problem, where $Y$ and $\lambda$ are given odd positive integers.

 \section{A Segue Into Exponential Diophantine Equations and Combinatorics}

 One could argue that this is the richest part of the thesis behind the thesis.  If one were to ask exactly in what manner any of the results here were derived, it would be by the virtue of Exponential Diophantine Equations.  Fix $a,b,x,y \in \mathbb{Z}_{>0}$.  Let the mapping $f:\mathbb{Z}_{>0}^4 \to \mathbb{Z}$ be given as,

\begin{equation*}
    f(x,y,a,b) = a^x-b^y.
\end{equation*}

\noindent Equations of this kind are exactly the ones which generate integral loops.  In our case we took $a =2$, $b=3$, $x=\rho$, and $y=\nu$.  In this thesis we consistently provided the starting integer $n$ and then chose $B\in\{0,1\}^{**}$ so that $B$ matched the characteristic trajectory of $n$.  However, that is not how one finds integral loops in practice.  It is much easier to pick $\rho$ and $\nu$ first, then choose an arrangement of ones and zeros generating $n$ by Equation (\ref{tilde_n}).  The number of integral loops follows a classic \textit{Stars and Bars} type problem\cite{aps}.  Let $\rho,\nu \in \mathbb{Z}_{>0}$ so that $2^\rho - 3^\nu \geq 0$.  Then the number of integral loops is bounded above by the following quantity:

\begin{equation*}
    \#\{\text{integral loops with $\rho$ zeros and $\nu$ ones}\} < {\rho - 1 \choose \nu-1} .
\end{equation*}

One could even suppose that the entirety of the Collatz Conjecture can be reduced down to a combinatorical question in an Exponential Diophantine Equation.  We will demonstrate the power of this kind of thinking with the following proposition:

\begin{proposition}\textit{Let $\gamma=1$.  Fix $n \in \mathbb{Z}_{>0}^{odd}$ and $N\in\mathbb{Z}_{>0}$.  Suppose $n$ starts an integral loop of length $N$.  Let $B\in \{0,1\}^{**}$ such that $B = (\chi^i(n))_{i\geq 0}$.  Let $\rho$ and $\nu$ be the count of the number of zeros and ones in a single period of $B$.  Then the ratio $\rho/\nu$ must satisfy the following inequality:}

    \begin{equation*}
        \frac{\ln(3)}{\ln(2)} < \frac{\rho}{\nu} \leq 2 .
    \end{equation*}
\end{proposition}

\begin{proof}
    The lower bound is a direct consequence of the Integral Loop Formula.  As for the upper bound, we will demonstrate by construction.  Let $\nu$ be a fixed positive integer.  We observe first that the ratio $\rho/\nu$ increases if and only if $2^\rho - 3^\nu$ increases.  Recall the indexing sequence $i_j$ of Equation (\ref{index_seq}).  In this case let $n_{j}$ denote the $(j-1)$th odd number in the integral loop starting with $n_1$.  Then we have the following system of equations:

    \begin{equation}
        \left\{\begin{array}{lll}
        2^{m_{j}} &= \frac{3n_j+1}{n_{j+1}}, & 1\leq j \leq \nu-1, \\
        2^{m_\nu} &= \frac{3n_{\nu-1}+1}{n_1}. &
        \end{array}\right.
    \end{equation}

\noindent Observe that the product of terms is $2^\rho$:

\begin{equation}
    \frac{(3n_1+1)(3n_2+1)... (3n_{\nu-1}+1)(3n_{\nu}+1)}{n_1 n_2 ... n_{\nu-1}n_{\nu}} = 2^\rho = (2^\rho-3^\nu)+3^\nu .
\end{equation}

\noindent We can reduce the product down to the following equation:

\begin{equation}\label{big_boy}
    2^\rho - 3^\nu = \frac{1+\sum_{i=1}^{\nu-1}\frac{3^i}{i!}\left( \sum_{\sigma \in S(\nu)} \prod_{k=1}^i n_{\sigma(i)}\right)}{\prod_{i=1}^\nu n_i} .
\end{equation}

\noindent Equation (\ref{big_boy}) is maximized when the denominator is minimized.  Thus, the largest value of $2^\rho - 3^\nu$ is obtained when $n_1 = n_2 = ... n_\nu = 1$.  Equation (\ref{big_boy}) reduces to:

\begin{equation}
\begin{aligned}
    2^\rho - 3^\nu &= \sum_{i=0}^{\nu-1} 3^i {\nu \choose i}, \\
    2^\rho &= \sum_{i=0}^\nu 3^i {\nu \choose i} = 4^\nu = 2^{2\nu} .
\end{aligned}
\end{equation}

\noindent We conclude that with $\nu$ being a fixed positive integer, the largest value $\rho$ can attain is $2\nu$ satisfying the upper inequality.
\end{proof}

\begin{example}
    Let $\rho = 9$ and $\nu = 4$.  Then the ratio $\rho/\nu = 9/4>2$.  We know there are a total of 56 integral loops starting with a one.  Of these 56 integral loops, we can combine them into 14 groups of four, corresponding to the same trajectory but with different starting numbers.  We observe that in each of these integral loops, there is at least one odd positive integer in the integral loop smaller than $2^9-3^4=431$.  Thus none of the integral loops within the $3x+431$ problem correspond to an integral loop of the $3x+1$ problem.

\begin{tabular}{c|l|l}
    $B$ & $3x+431$ Integral Loop & $3x+1$ Integral Loop (approx.) \\
    \hline
    \begin{tabular}{l}
        1001001001000\\
        1001001000100\\
        1001000100100\\
        1000100100100
    \end{tabular} & 
    \begin{tabular}{l}
         $\{175, 956, 478,$\\
         $ 239,1148, 574,$\\
         $287, 1292, 646,$\\
         $323, 1400, 700, 350\}$  \\
    \end{tabular} &
    \begin{tabular}{l}
         $\{0.406, 2.218, 1.109,$\\
         $ 0.554,2.663, 1.331,$\\
         $0.665, 2.997,1.498,$\\
         $ 0.749, 3.248, 1.624, 0.812\}$  \\
    \end{tabular} \\
    \hline
    \begin{tabular}{l}
        1010101000000\\
        1010100000010\\
        1010000001010\\
        1000000101010
    \end{tabular} & 
    \begin{tabular}{l}
         $\{65, 626, 313,$\\
         $ 1370,685, 2486,$\\
         $ 1243, 4160, 2080,$\\
         $ 1040, 520, 260, 130\}$\\
    \end{tabular} &
    \begin{tabular}{l}
         $\{0.150, 1.452, 0.726,$\\
         $ 3.178,1.589, 5.767,$\\
         $ 2.883, 9.651,4.825,$\\
         $ 2.412, 1.206, 0.603, 0.301\}$
    \end{tabular} \\
    \hline
    \begin{tabular}{l}
        1010100100000\\
        1010010000010\\
        1001000001010\\
        1000001010100
    \end{tabular} & 
    \begin{tabular}{l}
         $\{73, 650, 325,$\\
         $ 1406,703, 2540,$\\
         $ 1270, 635,2336,$\\
         $ 1168, 584, 292, 146\}$\\
    \end{tabular} &
    \begin{tabular}{l}
         $\{0.169, 1.508, 0.754,$\\
         $ 3.262,1.631, 5.893,$\\
         $ 2.946, 1.473,5.419,$\\
         $2.709, 1.354, 0.677, 0.338\}$
    \end{tabular}\\
    \hline
    \begin{tabular}{l}
        1001010100000\\
        1010100000100\\
        1010000010010\\
        1000001001010
    \end{tabular} & 
    \begin{tabular}{l}
         $\{103, 740, 370,$\\
         $ 185,986, 493,$\\
         $ 1910, 955,3296,$\\
         $ 1648, 824, 412, 206\}$\\
    \end{tabular} &
    \begin{tabular}{l}
         $\{0.238, 1.716, 0.858,$\\
         $ 0.429,2.287, 1.143,$\\
         $ 4.431, 2.215,7.647,$\\
         $ 3.823, 1.911, 0.955, 0.477\}$
    \end{tabular}\\
\end{tabular}
\end{example}

\begin{corollary}
    
\end{corollary}

 %---------------------------------------------------------------
 %///////////////////////////////////////////////////////////////
 %---------------------------------------------------------------

 \section{A Segue Into The Special Linear Group}

\noindent Let $\overrightarrow{\mathfrak{1}}^\nu$ denote the $\nu\times 1$ vector of all ones.

\begin{proposition}\label{dd_n}\textit{Fix $\gamma \in \mathbb{Z}_{>0}^{odd}$, $n \in \mathbb{Z}_{>0}^{odd}$, and $N\in\mathbb{Z}_{>0}$.  Suppose further that $n$ starts an integral loop of length $N$.  Let $B \in \{0,1\}^{**}$ such that $B = (\chi_\gamma^i(n))_{i\geq 0}$.  Let $\rho$ and $\nu$ be the count of the number of zeros and ones in a single period of $B$.  Let the matrix $\mathcal{D}_n$ given by Equation (\ref{d_n}).  Then the integral loop vector $[n]$ associated to $n$ and $\mathcal{D}_n$ satisfy the following equalities:}

\begin{equation}\label{central}
\begin{aligned} {}
[n] &= \frac{\gamma}{2^\rho - 3^\nu}\sum_{i=0}^{\nu-1} 3^{\nu-1-i}(D_n^i \overrightarrow{\mathfrak{1}}^\nu), \\
[n] &= \gamma(D_n - 3 \mathbb{I}_\nu)^{-1} \overrightarrow{\mathfrak{1}}^\nu.
\end{aligned}
\end{equation}

\noindent \textit{Where $\mathcal{D}_n \overrightarrow{\mathfrak{1}}^\nu$ is standard multiplication of a $\nu\times\nu$ matrix and a $\nu\times 1$ vector.}
\end{proposition}

\begin{proof}
The first equality follows from Proposition \ref{exp_rule} and Equation (\ref{integral_loop_formula_equ}).  The second equality can be derived from the first by a series expansion:

\begin{equation*}
\begin{aligned} {}
D_n[n] &= \frac{\gamma}{2^\rho - 3^\nu} \left( \sum_{i=0}^{\nu-2} 3^{\nu-1-i}(D_n^{i+1} \overrightarrow{\mathfrak{1}}^\nu) + 2^\rho D_n^0 \overrightarrow{\mathfrak{1}}^\nu \right), \\
3[n] &=  \frac{\gamma}{2^\rho - 3^\nu} \sum_{i=0}^{\nu-1} 3^{\nu-i}(D_n^i \overrightarrow{\mathfrak{1}}^\nu), \\
(D_n-3\mathbb{I}_\nu)[n] &= \gamma \overrightarrow{\mathfrak{1}}^\nu, \\
[n] &= \gamma(D_n - 3 \mathbb{I}_\nu)^{-1} \overrightarrow{\mathfrak{1}}^\nu.
\end{aligned}
\end{equation*}

We take note that $\det(D_n - 3\mathbb{I}_\nu) = (-1)^{\nu+1}(2^\rho-3^\nu) \neq 0$.  Thus, both equalities have been established.
\end{proof}

\begin{example}\normalfont
Consider the integral loop induced by $n = 133$ and $\gamma = 943$.  We verify the second equality of Equation (\ref{central}):

\begin{equation*}
\begin{aligned}
943 \left( \begin{bmatrix} -3 & 2 &  0 &  0 \\ 0 & -3 & 4 & 0 \\  0 & 0 & -3 & 8 \\ 16 & 0 & 0 & -3 \end{bmatrix} \right)^{-1} \times \begin{bmatrix} 1 \\ 1 \\ 1\\ 1 \end{bmatrix} &= 943 \begin{bmatrix} \frac{27}{943} & \frac{18}{943} &  \frac{24}{943} &  \frac{64}{943} \\ \frac{512}{943} & \frac{27}{943} & \frac{36}{943} & \frac{96}{943} \\  \frac{384}{943} & \frac{256}{943} & \frac{27}{943} & \frac{72}{943} \\ \frac{144}{943} & \frac{96}{943} & \frac{128}{943} & \frac{27}{943} \end{bmatrix} \times \begin{bmatrix} 1 \\ 1 \\ 1\\ 1 \end{bmatrix}, \\
&= 943 \begin{bmatrix} \frac{133}{943} \\ \frac{671}{943} \\ \frac{739}{943} \\ \frac{395}{943} \end{bmatrix} = \begin{bmatrix} 133 \\ 671\\ 739 \\ 395 \end{bmatrix} = [133].
\end{aligned}
\end{equation*}

\end{example}

\begin{proposition}\textit{Fix $\gamma, n, \tilde{n} \in \mathbb{Z}_{>0}^{odd}$, and $N\in\mathbb{Z}_{>0}$.  Suppose both $n$ and $\tilde{n}$ start integral loops of length $N$.  Let $B, \tilde{B} \in \{0,1\}^{**}$ be the binary sequences associated to the characteristic trajectories of $n$ and $\tilde{n}$, respectfully.  Let $\rho, \tilde{\rho}, \nu \text{ and } \tilde{\nu}$ be the count for the number of zeros and ones for a single period of $B$ and $\tilde{B}$, respectfully.  Suppose further that $\nu = \tilde{\nu}$.  Let $[n]$ and $[\tilde{n}]$ be the associated integral loop vectors and $\mathcal{D}_n$ and $\mathcal{D}_{\tilde{n}}$ constructed as by Equation (\ref{d_n}).  Then $[n]$ can be transformed into $[\tilde{n}]$ by the following matrix:}

\begin{equation}\label{sl}
[\tilde{n}] = (\mathcal{D}_{\tilde{n}}-3\mathbb{I}_\nu)^{-1}(\mathcal{D}_n-3\mathbb{I}_\nu)[n].
\end{equation}

\noindent \textit{Moreover, this transformation is an element of the special linear group $SL(\nu, \mathbb{R})$.}
\end{proposition}

\begin{proof}
The equality can be constructed directly from the second equation of Proposition \ref{dd_n}.  Secondly, as $\det(\mathcal{D}_n - 3\mathbb{I}_\nu) = (-1)^{\nu+1}(2^\rho - 3^\nu)$ and $\det(\mathcal{D}_{\tilde{n}} - 3\mathbb{I}_\nu)^{-1} = (-1)^{\nu+1}(2^\rho - 3^\nu)^{-1}$ then their product has determinant one and belongs in $SL(\nu,\mathbb{R})$.
\end{proof}

\begin{example}\normalfont

Consider the integral loops induced by $n = 23$, $\tilde{n} = 19$, and $\gamma = 37$.  Then we can transform [23] into [19] in the following manner:

\begin{equation*}
\begin{bmatrix} 19 \\ 47 \\ 89 \end{bmatrix} = \begin{bmatrix}\frac{9}{37} & \frac{6}{37} & \frac{4}{37} \\ \frac{32}{37} & \frac{9}{37} & \frac{6}{37} \\ \frac{48}{37} & \frac{32}{37} & \frac{9}{37}  \end{bmatrix}\begin{bmatrix} -3 & 2 & 0 \\ 0 & -3 & 4 \\ 8 & 0 & -3 \end{bmatrix} \begin{bmatrix} 23 \\ 53 \\ 49 \end{bmatrix}.
\end{equation*}

\noindent Furthermore, we also observe that:

\begin{equation*}
\det\left( \begin{bmatrix}\frac{9}{37} & \frac{6}{37} & \frac{4}{37} \\ \frac{32}{37} & \frac{9}{37} & \frac{6}{37} \\ \frac{48}{37} & \frac{32}{37} & \frac{9}{37}  \end{bmatrix}\begin{bmatrix} -3 & 2 & 0 \\ 0 & -3 & 4 \\ 8 & 0 & -3 \end{bmatrix} \right) = \det\left( \begin{bmatrix} \frac{5}{37} & 0 & \frac{12}{37} \\ \frac{-48}{37} & 1 & \frac{18}{37} \\ \frac{-72}{37} & 0 & \frac{101}{37} \end{bmatrix} \right) = 1.
\end{equation*}

\end{example}

 %---------------------------------------------------------------
 %///////////////////////////////////////////////////////////////
 %---------------------------------------------------------------

\section{A Segue Into The One Row Sum Group}

\begin{proposition}\textit{Fix $\gamma, n, \tilde{n} \in \mathbb{Z}_{>0}^{odd}$, and $N\in\mathbb{Z}_{>0}$.  Suppose both $n$ and $\tilde{n}$ start integral loops of length $N$.  Let $B, \tilde{B} \in \{0,1\}^{**}$ be the binary sequences associated to the characteristic trajectories of $n$ and $\tilde{n}$, respectfully.  Let $\rho, \tilde{\rho}, \nu \text{ and } \tilde{\nu}$ be the count for the number of zeros and ones for a single period of $B$ and $\tilde{B}$, respectfully.  Suppose further that $\nu = \tilde{\nu}$.  Let $[n]$ and $[\tilde{n}]$ be the associated integral loop vectors of $n$ and $\tilde{n}$ and $\mathcal{M}_n$ and $\mathcal{M}_{\tilde{n}}$ constructed as by Equation (\ref{m_n}).  Suppose $\mathcal{M}_n$ is invertible over $\mathbb{C}$.  Then $[n]$ can be transformed into $[\tilde{n}]$ by the following matrix:}

\begin{equation}\label{ors}
[\tilde{n}] = \mathcal{M}_{\tilde{n}} \mathcal{M}_n^{-1} [n].
\end{equation}

\end{proposition}

\begin{proof}
Equation (\ref{ors}) can be derived immediately from Equation (\ref{central_2}).
\end{proof}
\begin{example}\normalfont

Consider the integral loop induced by $n = 133$, $\tilde{n} = 65$, and $\gamma = 943$.  Then we can transform [133] into [65] in the following manner:

\begin{equation*}
\begin{bmatrix} 65 \\ 569 \\ 1325 \\  2459 \end{bmatrix} = \begin{bmatrix} 1 & 2 & 4 & 8 \\ 1 & 2& 4 & 512 \\ 1 & 2 & 256 &512 \\ 1 & 128 & 256 & 512
\end{bmatrix}
\begin{bmatrix}
\frac{64}{51} & \frac{-2}{17} & \frac{-1}{51} & \frac{-2}{17} \\
\frac{-1}{17} & \frac{-1}{204} & \frac{-1}{68} & \frac{4}{51} \\
\frac{1}{408} & \frac{-1}{272} & \frac{1}{102} & \frac{-1}{272} \\
\frac{-1}{544} & \frac{1}{408} & \frac{-1}{2176} & \frac{-1}{6528}
\end{bmatrix}
\begin{bmatrix} 133 \\ 671 \\ 739 \\  395 \end{bmatrix}.
\end{equation*}

\end{example}

\begin{lemma}\label{one_zero}\textit{Fix $\gamma \in \mathbb{Z}_{>0}^{odd}$, $n \in \mathbb{Z}_{>0}^{odd}$, and $N\in\mathbb{Z}_{>0}$.  Suppose further that $n$ starts an integral loop of length $N$.  Let $B \in \{0,1\}^{**}$ such that $B = (\chi_\gamma^i(n))_{i\geq 0}$.  Let $\rho$ and $\nu$ be the count of the number of zeros and ones in a single period of $B$.  Let $[n]$ be the associated integral loop vector.  Let the matrix $\mathcal{M}_n$ given by Equation (\ref{m_n}) and assume that $\mathcal{M}_n$ is invertible over $\mathbb{C}$.  Then the first row of $\mathcal{M}_n^{-1}$ sums to one, while the remaining rows sum to zero.}
\end{lemma}

\begin{proof}
Recall that the first row of $\mathcal{M}_n$ is the all ones vector $\overrightarrow{\mathfrak{1}}^\nu$.  By virtue of invertibility, we require $\mathcal{M}_n^{-1} \mathcal{M}_n = \mathbb{I}_\nu$.  This is only satisfied if $\mathcal{M}_n^{-1}\overrightarrow{\mathfrak{1}}^\nu = e_1$.  
\end{proof}

\begin{theorem}\label{one_sum}\textit{Let $\gamma, n, \tilde{n} \in \mathbb{Z}_{>0}^{odd}$, and $N\in\mathbb{Z}_{>0}$ be given.  Suppose both $n$ and $\tilde{n}$ start integral loops of length $N$.  Let $B, \tilde{B} \in \{0,1\}^{**}$ be the binary sequences associated to the characteristic trajectories of $n$ and $\tilde{n}$, respectfully.  Let $\rho, \tilde{\rho}, \nu \text{ and } \tilde{\nu}$ be the count for the number of zeros and ones for a single period of $B$ and $\tilde{B}$, respectfully.  Suppose further that $\nu = \tilde{\nu}$.  Let $\mathcal{M}_n$ and $\mathcal{M}_{\tilde{n}}$ be constructed as by Equation (\ref{m_n}), and assume that $\mathcal{M}_n$ is invertible over $\mathbb{C}$.  Then each of the row sums of $\mathcal{M}_{\tilde{n}} \mathcal{M}_n^{-1}$ is one.}
\end{theorem}

\begin{proof}
Recall that from Lemma \ref{one_zero} that the the first row sum of $\mathcal{M}_n^{-1}$ is one, while the remaining row sums are zero.  With a proper change of variables, we can rewrite $\mathcal{M}_n^{-1}$ in the following block diagram:

\begin{equation*}
\mathcal{M}_n^{-1}=\left[
\begin{array}{c|ccc}
1-\sum_{i=2}^\nu a_{1,i}  & a_{1,2}& \cdots & a_{1,\nu} \\ \hline
 -\sum_{i=2}^\nu a_{2,i}  & a_{2,2}& \cdots & a_{2,\nu} \\
\vdots & \vdots & \ddots & \vdots \\
 -\sum_{i=2}^\nu a_{\nu, i}  & a_{\nu,2}& \cdots & a_{\nu,\nu} 
\end{array}\right]
\end{equation*}

\noindent Without loss of generality, let us consider the $j$th row of the product of $\mathcal{M}_{\tilde{n}} \mathcal{M}_n^{-1}$.  Written out explicitly, we find that:

\begin{equation*}
\begin{aligned} {}
[\mathcal{M}_{\tilde{n}} \mathcal{M}_n^{-1}]_j &= \left( 1-\sum_{i=2}^\nu a_{1,i} -2^{m'_j}\sum_{i=2}^\nu a_{2,i} - 2^{m'_j+m'_{j+1}}\sum_{i=2}^\nu a_{3,i}-... \right)e_{j,1}, \\
&= \left( a_{1,2} +2^{m'_j} a_{2,2}+2^{m'_j+m'_{j+1}} a_{3,2}+... \right) e_{j,2}, \\
&= \left( a_{1,3} +2^{m'_j} a_{2,3}+2^{m'_j+m'_{j+1}} a_{3,3}+... \right) e_{j,3}, \\
& ... \\
&= \left( a_{1,\nu} +2^{m'_j} a_{2,\nu}+2^{m'_j+m'_{j+1}} a_{3,\nu}+... \right) e_{j,\nu}. \\
\end{aligned}
\end{equation*}

\noindent As a row sum, every negative term in $e_{j,1}$ cancels identically with a corresponding term further down the row.  Since the row choice was arbitrary, we conclude that the row sums of the matrix product $\mathcal{M}_{\tilde{n}} \mathcal{M}_n^{-1}$ are all ones.
\end{proof}

\begin{example}\normalfont

Consider the integral loops induced by $n = 133$, $\tilde{n} = 65$, and $\gamma = 943$.  We can observe the row-sum property below:

\begin{equation*}
\begin{aligned}
\mathcal{M}_{65}\mathcal{M}_{133}^{-1}\overrightarrow{\mathfrak{1}}^4 &= \begin{bmatrix} 1 & 2 & 4 & 8 \\ 1 & 2& 4 & 512 \\ 1 & 2 & 256 &512 \\ 1 & 128 & 256 & 512
\end{bmatrix}
\begin{bmatrix}
\frac{64}{51} & \frac{-2}{17} & \frac{-1}{51} & \frac{-2}{17} \\
\frac{-1}{17} & \frac{-1}{204} & \frac{-1}{68} & \frac{4}{51} \\
\frac{1}{408} & \frac{-1}{272} & \frac{1}{102} & \frac{-1}{272} \\
\frac{-1}{544} & \frac{1}{408} & \frac{-1}{2176} & \frac{-1}{6528}
\end{bmatrix}
\begin{bmatrix} 1 \\ 1 \\ 1 \\  1 \end{bmatrix}, \\
&= \begin{bmatrix}
\frac{227}{204} & \frac{-25}{204} & \frac{-11}{816} & \frac{19}{816} \\
\frac{19}{102} & \frac{227}{204} & \frac{-25}{102} & \frac{-11}{204} \\
\frac{-22}{51} & \frac{19}{102} & \frac{227}{102} & \frac{-50}{51} \\
\frac{-400}{51} & \frac{-22}{51} & \frac{19}{51} & \frac{454}{51}
\end{bmatrix}
\begin{bmatrix} 1 \\ 1 \\ 1 \\  1 \end{bmatrix} = \begin{bmatrix} 1 \\ 1 \\ 1 \\  1 \end{bmatrix}.
\end{aligned}
\end{equation*}

\end{example}

\noindent Theorem \ref{one_sum} indicates a property of Equation (\ref{ors}) that deserves its own attention.  Thus we proceed with the following.  

\begin{definition}
A \textit{stochastic matrix} over a field $F$ is a square matrix with entries from $F$ with the property that the entries in each of its columns add up to one\normalfont \cite{stochastic1995David}.
\end{definition}

%\begin{remark} \normalfont We can identify every transformation matrix from Equation (\ref{ors}) uniquely to a stochastic matrix simply by taking the transpose.  Since this distinction does not change any of the properties we are studying, it will not be mentioned again.
%\end{remark}   

\noindent Recall Theorem \ref{one_sum} only assumed that $\mathcal{M}_n$ was invertible.  However, if we suppose that $\mathcal{M}_{\tilde{n}}$ is also invertible, then their product $\mathcal{M}_{\tilde{n}}\mathcal{M}_n^{-1}$ is invertible.  This takes us to another definition.

\begin{definition}\normalfont The group of nonsingular stochastic $n\times n$ matrices over a field $F$ is called the \textit{stochastic group} of $n\times n$ matrices over $F$ and is denoted $S(n,F)$\cite{stochastic1995David}.
\end{definition}

\noindent Thus we see that in the case of Theorem \ref{one_sum}, if $\mathcal{M}_{\tilde{n}}\mathcal{M}_n^{-1}$ is nonsingular then $\left(\mathcal{M}_{\tilde{n}}\mathcal{M}_n^{-1}\right)^T $ is an element of $S(\nu ,\mathbb{R})$.  We provide one more property of this group, which will be useful later.  Recall that $Aff(n,F)$ is the $n\times n$ \textit{affine group} over a field $F$.

 \section{A Segue Into an Eigenvector Problem}

 Equation (\ref{sl}) and Equation (\ref{ors}) both provide a limited means to transform one integral loop vector into another integral loop vector.  So long as the respective assumptions hold, it is clear that both equations satisfy the conditions of an equivalence relation.  This is formalized in the following proposition.

\begin{proposition}\label{equiv}\textit{Fix $\gamma, n, \tilde{n}, \tilde{\tilde{n}} \in \mathbb{Z}_{>0}^{odd}$, and $N\in\mathbb{Z}_{>0}$.  Suppose $n$, $\tilde{n}$, and $\tilde{\tilde{n}}$ start integral loops of length $N$.  Let $B, \tilde{B}, \text{ and } \tilde{\tilde{B}} \in \{0,1\}^{**}$ be the binary sequences associated to the characteristic trajectories of $n$, $\tilde{n}$, and $\tilde{\tilde{n}}$, respectfully.  Let $\rho, \tilde{\rho}, \tilde{\tilde{\rho}}, \nu, \tilde{\nu} \text{ and } \tilde{\tilde{\nu}}$ be the count for the number of zeros and ones for a single period of $B$, $\tilde{B}$, and $\tilde{\tilde{B}}$, respectfully.  Suppose further that $\nu = \tilde{\nu} = \tilde{\tilde{\nu}}$.  Let $\mathcal{M}_n, \mathcal{M}_{\tilde{n}}, \text{ and } \mathcal{M}_{\tilde{\tilde{n}}}$ be constructed as by Equation (\ref{m_n}), and assume that all three are invertible over $\mathbb{C}$.  Let $\mathcal{D}_n, \mathcal{D}_{\tilde{n}}, \text{ and } \mathcal{D}_{\tilde{\tilde{n}}}$ be constructed as Equation (\ref{d_n}). Then Equation (\ref{sl}) and Equation (\ref{ors}) independently satisfy the conditions of an equivalence relation.}
\end{proposition}

\begin{proof}
We will check the properties for both equations.
\begin{enumerate}
\item Reflexivity: $[n] \sim [n]$
\begin{enumerate}
\item $[n] = \mathcal{M}_n \mathcal{M}_n^{-1} [n] = \mathbb{I}_\nu [n] = [n].$
\item $[n] = (\mathcal{D}_n - 3\mathbb{I}_\nu)^{-1}(\mathcal{D}_n - 3\mathbb{I}_\nu) [n] = \mathbb{I}_\nu [n] = [n].$
\end{enumerate}
\item Symmetry: $[n] \sim [\tilde{n}] \Leftrightarrow [\tilde{n}] \sim [n]$.
\begin{enumerate}
\item $[n] = \mathcal{M}_n \mathcal{M}_{\tilde{n}}^{-1} [\tilde{n}] \Leftrightarrow [\tilde{n}] = \mathcal{M}_{\tilde{n}} \mathcal{M}_n^{-1} [n].$
\item $[n] = (\mathcal{D}_n - 3\mathbb{I}_\nu)^{-1}(\mathcal{D}_{\tilde{n}} - 3\mathbb{I}_\nu) [n'] \Leftrightarrow [\tilde{n}] = (\mathcal{D}_{\tilde{n}} - 3\mathbb{I}_\nu)^{-1}(\mathcal{D}_n - 3\mathbb{I}_\nu) [n].$
\end{enumerate}
\item Transitivity: $[n] \sim [\tilde{n}]$ and $[\tilde{n}] \sim [\tilde{\tilde{n}}]$ then $[n] \sim [\tilde{\tilde{n}}].$
\begin{enumerate}
\item $[n] = \mathcal{M}_n \mathcal{M}_{\tilde{n}}^{-1} [\tilde{n}]$ and $[\tilde{n}] = \mathcal{M}_{\tilde{n}} \mathcal{M}_{\tilde{\tilde{n}}}^{-1} [\tilde{\tilde{n}}]$ then \\ $[n] = \mathcal{M}_n \mathcal{M}_{\tilde{n}}^{-1}\mathcal{M}_{\tilde{n}} \mathcal{M}_{\tilde{\tilde{n}}}^{-1} [\tilde{\tilde{n}}] = \mathcal{M}_n \mathcal{M}_{\tilde{\tilde{n}}}^{-1} [\tilde{\tilde{n}}].$
\item $[n] = (\mathcal{D}_n - 3\mathbb{I}_\nu)^{-1}(\mathcal{D}_{\tilde{n}} - 3\mathbb{I}_\nu) [\tilde{n}]$ and $[\tilde{n}] = (\mathcal{D}_{\tilde{n}} - 3\mathbb{I}_\nu)^{-1}(\mathcal{D}_{\tilde{\tilde{n}}} - 3\mathbb{I}_\nu) [\tilde{\tilde{n}}]$ then \\ $[n] = (\mathcal{D}_n - 3\mathbb{I}_\nu)^{-1}(\mathcal{D}_{\tilde{n}} - 3\mathbb{I}_\nu)(\mathcal{D}_{\tilde{n}} - 3\mathbb{I}_\nu)^{-1}(\mathcal{D}_{\tilde{\tilde{n}}} - 3\mathbb{I}_\nu) [\tilde{\tilde{n}}] = (\mathcal{D}_n - 3\mathbb{I}_\nu)^{-1}(\mathcal{D}_{\tilde{\tilde{n}}} - 3\mathbb{I}_\nu) [\tilde{\tilde{n}}].$
\end{enumerate}
\end{enumerate}
\end{proof}

Proposition \ref{equiv} formalizes the structure of the transformation equations.  However, it has its limitations.  If we were inclined to construct a closed loop from an integral loop vector to itself, then by transitivity and reflexivity, the transformation would reduce to the identity.  However, if we were to take a combination of Equations (\ref{sl}) and (\ref{ors}), this is no longer the case.  This is described in the following theorem.

\begin{theorem}\label{pap}\textit{Fix $\gamma, n, \tilde{n} \in \mathbb{Z}_{>0}^{odd}$, and $N\in\mathbb{Z}_{>0}$.  Suppose both $n$ and $\tilde{n}$ start integral loops of length $N$.  Let $B, \tilde{B} \in \{0,1\}^{**}$ be the binary sequences associated to the characteristic trajectories of $n$ and $\tilde{n}$, respectfully.  Let $\rho, \tilde{\rho}, \nu \text{ and } \tilde{\nu}$ be the count for the number of zeros and ones for a single period of $B$ and $\tilde{B}$, respectfully.  Suppose further that $\nu = \tilde{\nu}$.  Let $\mathcal{M}_n, \text{ and } \mathcal{M}_{\tilde{n}}$ be constructed as by Equation (\ref{m_n}), and assume that both are invertible over $\mathbb{C}$.  Let $\mathcal{D}_n \text{ and } \mathcal{D}_{\tilde{n}}$ be constructed as Equation (\ref{d_n}).  Then $[n]$, $\mathcal{D}_n$, and $\mathcal{M}_n$ satisfy the following properties:}

\begin{enumerate}
\item\begin{equation*}
[n] = \left(\mathcal{D}_n -3\mathbb{I}_\nu\right)^{-1}\left(\mathcal{D}_{\tilde{n}}-3\mathbb{I}_\nu\right) M_{\tilde{n}}M_n^{-1} [n].
\end{equation*}
\item \begin{equation*}
[n] = \mathcal{M}_n \left( \sum_{j=1}^\nu \frac{3^{\nu-j}}{2^\rho - 3^\nu} e_{j,j} \right) \left(\mathcal{D}_n -3 \mathbb{I}_\nu\right) [n] = \mathcal{A}_n [n].
\end{equation*}
\item $\mathcal{P}\mathcal{A}_n \mathcal{P}^{-1} \in Aff(\nu-1,\mathbb{R}).$
\item $\mathcal{P}[n] = \left( \mathcal{P}\mathcal{A}_n \mathcal{P}^{-1}\right)\mathcal{P}[n].$
\item $\mathcal{P}[n]$ is a fixed point of the affine transformation $\mathcal{P}\mathcal{A}_n \mathcal{P}^{-1}.$
\end{enumerate}
\end{theorem}

\begin{proof}
Property (1) follows from the symmetry relations of Proposition \ref{equiv}.  Property (2) is derived from Proposition \ref{dd_n}, Equation (\ref{central}) combined with Proposition \ref{3v} Equation (\ref{central_2}).  This is demonstrated below:

\begin{equation*}
\begin{aligned}
\overrightarrow{\mathfrak{1}}^\nu &= \frac{1}{\gamma} (\mathcal{D}_n - 3\mathbb{I}_\nu)[n], \\
\left( \sum_{j=1}^\nu 3^{\nu-j}e_{j,j} \right)\overrightarrow{\mathfrak{1}}^\nu &= \left( \sum_{j=1}^\nu \frac{3^{\nu-j}}{\gamma} e_{j,j} \right) (\mathcal{D}_n - 3\mathbb{I}_\nu)[n], \\
\overrightarrow{\mathfrak{3}}^\nu &= \left( \sum_{j=1}^\nu \frac{3^{\nu-j}}{\gamma} e_{j,j} \right) (\mathcal{D}_n - 3\mathbb{I}_\nu)[n], \\
\left( \frac{\gamma}{2^\rho-3^\nu} \mathcal{M}_n \right)\overrightarrow{\mathfrak{3}}^\nu &= \mathcal{M}_n \left( \sum_{j=1}^\nu \frac{3^{\nu-j}}{2^\rho-3^\nu} e_{j,j} \right) (\mathcal{D}_n - 3\mathbb{I}_\nu)[n], \\
[n] &= \mathcal{M}_n \left( \sum_{j=1}^\nu \frac{3^{\nu-j}}{2^\rho-3^\nu} e_{j,j} \right) (\mathcal{D}_n - 3\mathbb{I}_\nu)[n].
\end{aligned}
\end{equation*}

\noindent The quantity in parenthesis is a diagonal matrix.  It is the $\nu \times \nu$ analogue of $\overrightarrow{\mathfrak{3}}^\nu$ with descending powers of 3 along the main diagonal scaled by $1/(2^\rho-3^\nu)$.  Property (3) is demonstrated in two steps.  We claim first that $\mathcal{A}_n$ can be written in the following block diagonal form:

\begin{equation}\label{a_n}
\mathcal{A}_n=\left[
\begin{array}{c|ccc}
1  & 0\times \mathfrak{1}_\nu^T\\ \hline
b_2  & a_{22}& \cdots & a_{2\nu} \\
\vdots & \vdots & \ddots & \vdots \\
 b_\nu  & a_{\nu2}& \cdots & a_{\nu\nu} 
\end{array}\right].
\end{equation}

\noindent Then all that is required is a sequence of row and column transformations.  We choose such transformations that retain the same sequential order of terms of the integral loop vector, though not in the same entries.  This is readily satisfied with conjugation by $\mathcal{P}$.

\begin{equation*}
\mathcal{P}\mathcal{A}_n\mathcal{P}^{-1}=\left[
\begin{array}{ccc|c}
a_{22}  & \cdots & a_{2\nu} & b_2\\ 
\vdots  & \ddots& \vdots & \vdots \\ 
a_{\nu 2} & \cdots & a_{\nu\nu} & b_\nu \\ \hline
0\times \mathfrak{1}_\nu^T & & & 1
\end{array}\right].
\end{equation*}

\noindent With that established, we now need to prove that $\mathcal{A}_n$ has this desired form.  We compose the product leading to $\mathcal{A}_n$ from left to right.

\begin{equation*}
\mathcal{M}_n \left( \sum_{j=1}^\nu \frac{3^{\nu-j}}{2^\rho-3^\nu} e_{j,j} \right) = \begin{bmatrix}\frac{3^{\nu-1}}{2^\rho-3^\nu} & \frac{3^{\nu-2}2^{m_1}}{2^\rho-3^\nu} & \frac{3^{\nu-3}2^{m_1+m_2}}{2^\rho-3^\nu} & ... & \frac{2^{\rho-m_\nu}}{2^\rho-3^\nu} \\ 
\frac{3^{\nu-1}}{2^\rho-3^\nu} & \frac{3^{\nu-2}2^{m_2}}{2^\rho-3^\nu} & \frac{3^{\nu-3}2^{m_2+m_3}}{2^\rho-3^\nu} & ... & \frac{2^{\rho-m_1}}{2^\rho-3^\nu} \\
\vdots & \vdots & \vdots & \vdots & \vdots \\
\frac{3^{\nu-1}}{2^\rho-3^\nu} & \frac{3^{\nu-2}2^{m_\nu}}{2^\rho-3^\nu} & \frac{3^{\nu-3}2^{m_\nu+m_1}}{2^\rho-3^\nu} & ... & \frac{2^{\rho-m_{\nu-1}}}{2^\rho-3^\nu} \end{bmatrix}.
\end{equation*}

\noindent We can observe that the row sums yield $[n]/\gamma$.  Now composing with the final matrix we have:

\begin{equation*}
\begin{aligned}
M_n & \left( \sum_{j=1}^\nu \frac{3^{\nu-j}}{2^\rho-3^\nu} e_{j,j} \right) \left(D_n -3 \mathbb{I}_\nu\right)  \\ &= \begin{bmatrix}\frac{3^{\nu-1}}{2^\rho-3^\nu} & \frac{3^{\nu-2}2^{m_1}}{2^\rho-3^\nu} & \frac{3^{\nu-3}2^{m_1+m_2}}{2^\rho-3^\nu} & ... & \frac{2^{\rho-m_\nu}}{2^\rho-3^\nu} \\ 
\frac{3^{\nu-1}}{2^\rho-3^\nu} & \frac{3^{\nu-2}2^{m_2}}{2^\rho-3^\nu} & \frac{3^{\nu-3}2^{m_2+m_3}}{2^\rho-3^\nu} & ... & \frac{2^{\rho-m_1}}{2^\rho-3^\nu} \\
\vdots & \vdots & \vdots & \vdots & \vdots \\
\frac{3^{\nu-1}}{2^\rho-3^\nu} & \frac{3^{\nu-2}2^{m_\nu}}{2^\rho-3^\nu} & \frac{3^{\nu-3}2^{m_\nu+m_1}}{2^\rho-3^\nu} & ... & \frac{2^{\rho-m_{\nu-1}}}{2^\rho-3^\nu} \end{bmatrix} \begin{bmatrix} -3 & 2^{m_1} & 0 & \cdots & 0 & 0 \\ 0 & -3 & 2^{m_2} & 0 & \cdots & 0 \\
\vdots & \vdots & \vdots & \vdots & \vdots & \vdots \\ 
2^{m_\nu} & 0 &\cdots & \cdots & \cdots & -3 \end{bmatrix}.
\end{aligned}
\end{equation*}

\noindent Considering the first row, we find that we have the following sum:

\begin{equation*} \begin{aligned} {}
[\mathcal{A}_n]_1 &= \left(\frac{-3^\nu + 2^\rho}{2^\rho-3^\nu}\right)e_{1,1} + \left(\frac{3^{\nu-1}2^{m_1}-3^{\nu-1}2^{m_1}}{2^\rho-3^\nu}\right)e_{1,2}, \\ &+ \left(\frac{3^{\nu-2}2^{m_1+m_2}-3^{\nu-2}2^{m_1+m_2}}{2^\rho-3^\nu}\right)e_{1,3} +... \\
&= \frac{2^\rho-3^\nu}{2^\rho-3^\nu}e_{1,1} + 0\times(e_{1,2}+...+e_{1,\nu}) = e_{1,1}.
\end{aligned}
\end{equation*}

\noindent Thus $\mathcal{A}_n$ has the desired form.  Regarding Property (4) consider Property (2),

\begin{equation*}
\begin{aligned} {}
[n] &= \mathcal{A}_n [n], \\
\mathcal{P}[n] &= \mathcal{P}\mathcal{A}_n [n], \\
\mathcal{P}[n] &= \mathcal{P}\mathcal{A}_n \mathcal{P}^{-1} \mathcal{P}[n].
\end{aligned}
\end{equation*}

\noindent Property (5) follows immediately.

\end{proof}

\bibliographystyle{abbrv}
\bibliography{ref}

\begin{thebibliography}{10}

\bibitem{aps}
Ball-and-urn.

\bibitem{BB2020Aaronson}
S.~Aaronson.
\newblock The busy beaver frontier.
\newblock {\em SIGACT News}, 51(3):32--54, sep 2020.

\bibitem{alves2005linear}
J.~Alves, M.~Gra{\c{c}}a, M.~Dias, and J.~S. Ramos.
\newblock A linear algebra approach to the conjecture of {C}ollatz.
\newblock {\em Linear algebra and its applications}, 394:277--289, 2005.

\bibitem{Brualdi1970AnEP}
R.~A. Brualdi and M.~H.~A. Newman.
\newblock An enumeration problem for a congruence equation.
\newblock {\em Journal of Research of the National Bureau of Standards, Section
  B: Mathematical Sciences}, page~37, 1970.

\bibitem{conway1972iterations}
J.~H. Conway.
\newblock Unpredictable iterations.
\newblock In {\em The Ultimate Challenge: The 3x+1 Problem}, pages 219--223.
  Amer. Math. Soc., 2010.

\bibitem{coxeter2010cyclic}
H.~Coxeter.
\newblock Cyclic sequences and frieze patterns.
\newblock In {\em The Ultimate Challenge: The $3x+ 1$ {P}roblem}, pages
  211--217. Amer. Math. Soc., 2010.

\bibitem{davisphilip}
P.~J. Davis.
\newblock {\em Circulant Matrices}.
\newblock Wiley-Interscience, New York, NY, 1970.

\bibitem{dolan1987generalization}
J.~Dolan, A.~Gilman, and S.~Manickam.
\newblock A generalization of {E}verett's result on the {C}ollatz $3x+ 1$
  problem.
\newblock {\em Advances in Applied Mathematics}, 8(4):405--409, 1987.

\bibitem{everett1977iteration}
C.~Everett.
\newblock Iteration of the number-theoretic function $f (2n)= n$, $f (2n+ 1)=
  3n+ 2$.
\newblock {\em Advances in Mathematics}, 25(1):42--45, 1977.

\bibitem{gardner1972mathematical}
M.~Gardner.
\newblock Mathematical games.
\newblock {\em Scientific American}, 226(6):114--121, 1972.

\bibitem{hasse1975unsolved}
H.~Hasse.
\newblock Unsolved problems in elementary number theory.
\newblock {\em Lectures at U. Maine (Orono)}, Spring 1975.

\bibitem{3498906}
A.~S. (https://math.stackexchange.com/users/436618/angina seng).
\newblock Singular circulant matrix.
\newblock Mathematics Stack Exchange.
\newblock URL:https://math.stackexchange.com/q/3498906 (version: 2020-01-06).

\bibitem{vandermode}
T.~Hughes.
\newblock The vandermonde determinant, a novel proof, 2020.

\bibitem{KOHL2007322}
S.~Kohl.
\newblock Wildness of iteration of certain residue-class-wise affine mappings.
\newblock {\em Advances in Applied Mathematics}, 39(3):322--328, 2007.

\bibitem{kontrovich2006structure}
A.~Kontorovich and Y.~Sinai.
\newblock Structure theorem for (d, g, h)-maps.
\newblock {\em Bulletin of the Brazilian Mathematical Society}, 33(2):213--224,
  July 2002.

\bibitem{ivan1994density}
I.~Korec.
\newblock A density estimate for the $3x+1$ problem.
\newblock {\em Mathematica Slovaca}, 44(1):85--89, 1994.

\bibitem{lagarias19853x+}
J.~Lagarias.
\newblock The $3x+ 1$ {P}roblem and its generalizations.
\newblock {\em American Mathematical Monthly}, pages 3--23, 1985.

\bibitem{lagarias2010overview}
J.~Lagarias.
\newblock The $3x+1$ {P}roblem: An overview.
\newblock In {\em The Ultimate Challenge: The 3x+ 1 Problem}, pages 3--29.
  Amer. Math. Soc., 2010.

\bibitem{LEHMER197343}
D.~Lehmer.
\newblock Some properties of circulants.
\newblock {\em Journal of Number Theory}, 5(1):43--54, 1973.

\bibitem{LEHTONEN2008596}
E.~Lehtonen.
\newblock Two undecidable variants of collatz's problems.
\newblock {\em Theoretical Computer Science}, 407(1):596--600, 2008.

\bibitem{Michel2014SimulationOT}
P.~Michel.
\newblock Simulation of the collatz 3x+1 function by turing machines.
\newblock {\em arXiv: Logic}, 2014.

\bibitem{Michel2015}
P.~Michel.
\newblock Problems in number theory from busy beaver competition.
\newblock {\em Logical Methods in Computer Science}, 11(4), dec 2015.

\bibitem{michel2010generalized}
P.~Michel and M.~Margenstern.
\newblock Generalized $3x+ 1$ functions and the theory of computation.
\newblock In {\em The Ultimate Challenge: The 3x+1 Problem}, pages 105--130.
  Amer. Math. Soc., 2010.

\bibitem{stochastic1995David}
D.~G. Poole.
\newblock The stochastic group.
\newblock {\em The American Mathematical Monthly}, 102(9):798--801, 1995.

\bibitem{rawsthorne1985imitation}
D.~A. Rawsthorne.
\newblock Imitation of an iteration.
\newblock {\em Mathematics Magazine}, 58(3):172--176, 1985.

\bibitem{SBURLATI2010100}
G.~Sburlati.
\newblock On prime factors of determinants of circulant matrices.
\newblock {\em Linear Algebra and its Applications}, 432(1):100--106, 2010.

\bibitem{tau2020}
T.~Tao.
\newblock Almost all orbits of the collatz map attain almost bounded values,
  2019.

\bibitem{Hugh2004}
H.~Thomas.
\newblock The number of terms in the permanent and the determinant of a generic
  circulant matrix.
\newblock {\em J. Algebraic Comb.}, 20(1):55?60, jul 2004.

\bibitem{van2005collatz}
J.~P. Van~Bendegem.
\newblock The {C}ollatz {C}onjecture. {A} case study in mathematical problem
  solving.
\newblock {\em Logic and Logical Philosophy}, 14(1):7--23, 2005.

\bibitem{whiteman1991memoriam}
J.~Whiteman.
\newblock In memoriam: Lothar {C}ollatz.
\newblock {\em International Journal for Numerical Methods in Engineering},
  31(8):1475--1476, 1991.

\bibitem{circulant_proof}
Yutsumura.
\newblock Determinant of a general circulant matrix, 2017.

\end{thebibliography}

\end{document}